\documentclass[a4paper,11pt]{amsart}
\usepackage{amsmath,amsthm,amssymb}
\usepackage[dvipdfmx]{graphicx}
\usepackage[dvipdfmx]{hyperref}
\usepackage[all]{xy}
\usepackage{graphicx}
\usepackage{here}
\usepackage{color}
\usepackage[top=30truemm,bottom=30truemm,left=25truemm,right=25truemm]{geometry}
\usepackage{ytableau}
\usepackage{pinlabel}
\usepackage{bm}

\newtheorem{theo}{Theorem}[subsection]
\newtheorem{defi}[theo]{Definition}
\newtheorem{lem}[theo]{Lemma}
\newtheorem{rem}[theo]{Remark}
\newtheorem{prop}[theo]{Proposition}
\newtheorem{cor}[theo]{Corollary}

\newtheorem{ex}[theo]{Example}

\newtheorem{Thm}{Theorem}

\newcommand{\nc}{\newcommand}
\nc{\on}{\operatorname}

\nc{\III}{I\hspace{-.1em}I\hspace{-.1em}I }

\nc{\C}{\mathbb{C}}
\nc{\R}{\mathbb{R}}
\nc{\Q}{\mathbb{Q}}
\nc{\Z}{\mathbb{Z}}
\nc{\N}{\mathbb{N}}

\nc{\bbB}{\mathbb{B}}
\nc{\bbI}{\mathbb{I}}
\nc{\bbS}{\mathbb{S}}
\nc{\bbT}{\mathbb{T}}

\nc{\bfa}{\mathbf{a}}
\nc{\bfA}{\mathbf{A}}
\nc{\bfb}{\mathbf{b}}
\nc{\bfB}{\mathbf{B}}
\nc{\bfC}{\mathbf{C}}
\nc{\bfH}{\mathbf{H}}
\nc{\bfi}{\mathbf{i}}
\nc{\bfL}{\mathbf{L}}
\nc{\bfr}{\mathbf{r}}
\nc{\bfs}{\mathbf{s}}
\nc{\bfT}{\mathbf{T}}
\nc{\bfU}{\mathbf{U}}
\nc{\bfV}{\mathbf{V}}
\nc{\V}{\mathbf{V}}
\nc{\bfw}{\mathbf{w}}
\nc{\bflm}{{\bm \lambda}}
\nc{\bfmu}{{\bm \mu}}
\nc{\bfnu}{{\bm \nu}}

\nc{\clB}{\mathcal{B}}
\nc{\clF}{\mathcal{F}}
\nc{\clL}{\mathcal{L}}
\nc{\clO}{\mathcal{O}}
\nc{\Hc}{\mathcal{H}}

\nc{\Ss}{\mathfrak{S}}
\nc{\Sym}{\Ss}
\nc{\frb}{\mathfrak{b}}
\nc{\g}{\mathfrak{g}}
\nc{\frh}{\mathfrak{h}}
\nc{\frn}{\mathfrak{n}}

\nc{\hf}{\frac{1}{2}}
\nc{\inv}{^{-1}}

\nc{\qu}{\quad}
\nc{\qqu}{\qquad}
\nc{\la}{\langle}
\nc{\ra}{\rangle}
\nc{\ol}{\overline}
\nc{\ul}{\underline}
\nc{\vphi}{\varphi}
\nc{\vpi}{\varpi}
\nc{\vep}{\varepsilon}
\nc{\lm}{\lambda}
\nc{\Lm}{\Lambda}
\nc{\Lmj}{\Lm^{\jmath}}

\nc{\Ker}{\on{Ker}}
\nc{\im}{\on{Im}}
\nc{\Hom}{\on{Hom}}
\nc{\End}{\on{End}}
\nc{\Span}{\on{Span}}
\nc{\Lie}{\on{Lie}}
\nc{\Gr}{\on{Gr}}
\nc{\op}{\on{op}}
\nc{\tr}{\on{tr}}
\nc{\id}{\on{id}}
\nc{\rk}{\on{rk}}
\nc{\ad}{\on{ad}}
\nc{\sr}{\on{sr}}
\nc{\ME}{\on{ME}}
\nc{\EM}{\on{EM}}
\nc{\para}{\on{par}}
\nc{\sgn}{\on{sgn}}
\nc{\gr}{\on{gr}}
\nc{\ef}{\on{ef}}
\nc{\Read}{\on{R}}
\nc{\rev}{\on{rev}}

\nc{\IF}{\text{ if }}
\nc{\AND}{\text{ and }}
\nc{\Or}{\text{ or }}
\nc{\OW}{\text{ otherwise}}
\nc{\Forall}{\text{ for all }}
\nc{\Forsome}{\text{ for some }}
\nc{\lowerterms}{\text{(lower terms)}}
\nc{\higherterms}{\text{(higher terms)}}
\nc{\otherterms}{\text{(other terms)}}
\nc{\IFF}{\text{ if and only if }}

\nc{\Lcell}{\underset{L}{\sim}}
\nc{\Rcell}{\underset{R}{\sim}}
\nc{\LRcell}{\underset{LR}{\sim}}

\nc{\Ij}{\bbI^{\jmath}}
\nc{\psij}{\psi^{\jmath}}
\nc{\sigmaj}{\sigma^{\jmath}}
\nc{\ej}{e^{\jmath}}
\nc{\Oj}{\Oint^{\jmath}}
\nc{\Oint}{\mathcal{O}_{\mathrm{int}}}
\nc{\clFj}{\clF^{\jmath}}

\nc{\A}{\mathbf{A}}
\nc{\Ao}{\A}
\nc{\Ai}{\A_{\infty}}
\nc{\Lo}{\clL_0}
\nc{\Li}{\clL_\infty}
\nc{\Gj}{G^{\jmath}}

\nc{\U}{\mathbf{U}}
\nc{\UA}{\U_{\A}}
\nc{\Uj}{\bfU^{\jmath}}
\nc{\Ujo}{\U^{\jmath,0}}
\nc{\UjA}{\Uj_{\A}}
\nc{\UjLevi}{\Uj_{\on{Levi}}}

\nc{\wt}{\on{wt}}
\nc{\wtj}{\wt^{\jmath}}
\nc{\Wtj}{\on{Wt}^{\jmath}}

\nc{\Pj}{P^{\jmath}}
\nc{\Par}{\on{Par}}
\nc{\DP}{\on{DP}}
\nc{\DPj}{\DP^{\jmath}}
\nc{\ST}{\on{ST}}
\nc{\SST}{\on{SST}}
\nc{\Cj}{C^{\jmath}}
\nc{\bbBj}{\bbB^{\jmath}}

\nc{\pij}{\pi^{\jmath}}
\nc{\IIj}{\ol{\Ij}}
\nc{\Qj}{Q^{\jmath}}
\nc{\etil}{\widetilde{e}}
\nc{\ftil}{\widetilde{f}}
\nc{\xtil}{\widetilde{x}}
\nc{\Etil}{\widetilde{E}}
\nc{\Ftil}{\widetilde{F}}
\nc{\Xtil}{\widetilde{X}}
\nc{\etilmax}{\etil^{\max}}
\nc{\ftilmax}{\ftil^{\max}}
\nc{\Lambdaj}{\Lambda^{\jmath}}
\nc{\leqj}{\leq^{\jmath}}
\nc{\geqj}{\geq^{\jmath}}
\nc{\simR}{\underset{R}{\sim}}
\nc{\simL}{\underset{L}{\sim}}
\nc{\ch}{\on{ch}}
\nc{\tauj}{\tau^{\jmath}}
\nc{\Psij}{\Psi^{\jmath}}
\nc{\trilefteq}{\trianglelefteq}
\nc{\trileft}{\triangleleft}
\nc{\tririghteq}{\trianglerighteq}
\nc{\triright}{\triangleright}
\nc{\qL}{\on{qL}}
\nc{\qB}{\on{qB}}
\nc{\vo}{v \fbox{$0$}}
\nc{\vi}{v \fbox{$1$}}
\nc{\vmi}{v \fbox{$-1$}}

\nc{\Yam}{\on{Yam}}
\nc{\LR}{\on{LR}}

\nc{\TBA}{\textcolor{red}{TBA}}

\title{Crystal basis theory for a quantum symmetric pair $(\U,\Uj)$}
\author[H. Watanabe]{Hideya Watanabe}
\address{(H. Watanabe) Department of Mathematics, Tokyo Institute of Technology, 2-12-1 Oh-okayama, Meguro-ku, Tokyo 152-8550, Japan}
\email{watanabe.h.at@m.titech.ac.jp}
\date{\today}

\begin{document}
\maketitle

\begin{abstract}
We study the representation theory of a quantum symmetric pair $(\U,\Uj)$ with two parameters $p,q$ of type AIII, by using highest weight theory and a variant of Kashiwara's crystal basis theory. Namely, we classify the irreducible $\Uj$-modules in a suitable category and associate with each of them a basis at $p=q=0$, the $\jmath$-crystal basis. The $\jmath$-crystal bases have nice  combinatorial properties as the ordinary crystal bases do.
\end{abstract}

\section{Introduction}
\subsection{Quantum Schur-Weyl duality}
Jimbo \cite{J86} established a quantum analog of the classical Schur-Weyl duality. Let $U_q(\mathfrak{gl}_n)$ denote the quantum group of $\mathfrak{gl}_n$, and $\Hc(\Ss_d)$ the Hecke algebra associated with the $d$-th symmetric group $\mathfrak{S}_d$, where $n,d \in \N$. Let $\bfV$ denote the vector representation of $U_q(\mathfrak{gl}_n)$. Jimbo defined an $\Hc(\Ss_d)$-module structure on $\bfV^{\otimes d}$ by using the $R$-matrix for $\bfV \otimes \bfV$. Also, he proved that the actions of $U_q(\mathfrak{gl}_n)$ and $\Hc(\Ss_d)$ on $\bfV^{\otimes d}$ satisfy the double centralizer property, and hence, $\bfV^{\otimes d}$ decomposes as a $U_q(\mathfrak{gl}_n)$-$\Hc(\Ss_d)$-bimodule as:
\begin{align}
\bfV^{\otimes d} = \bigoplus_{\lambda \in \Lambda} L(\lambda) \boxtimes S^\lambda, \nonumber
\end{align}
where $\Lambda$ is an index set, and $\{ L(\lambda) \mid \lambda \in \Lambda \}$ and $\{ S^\lambda \mid \lambda \in \Lambda \}$ are families of pairwise nonisomorphic irreducible modules of $U_q(\mathfrak{gl}_n)$ and $\Hc(\Ss_d)$, respectively.

\subsection{Quantum Schur-Weyl duality in type $B$}
It is known that the quantum groups of type $B$ and the Hecke algebra $\Hc(W_d)$ of type $B_d$ do not form the double centralizer. Recently, Bao and Wang discovered the double centralizer property between a quantum symmetric pair and $\Hc(W_d)$ (\cite{BW13}). More precisely, let $\Uj = \Uj_r$ be a coideal subalgebra of $\U = \U_{2r+1} = U_q(\mathfrak{sl}_{2r+1})$ such that $(\U,\Uj)$ forms a quantum analog of the symmetric pair of type AIII (\cite{Le99}, \cite{Ko14}). In \cite{BW13}, Bao and Wang introduced the intertwiner $\Upsilon$, which played a central role when they defined the action of $\Hc(W_d)$ on $\bfV^{\otimes d}$, and then, proved that the actions of $\Uj$ and $\Hc(W_d)$ on $\bfV^{\otimes d}$ satisfy the double centralizer property. A variant of this work, where $\Hc(W_d)$ is replaced with the Hecke algebra of type $B_d$ with unequal parameters $(p,q)$, was done in \cite{BWW16}.

\subsection{Representation theory of $\Uj$}
From the quantum Schur-Weyl duality in type $B$, we expect that there should exist a deep connection between the representation theory of $\Uj$ and that of $\Hc(W_d)$. However, here arises a problem: although the representation theory of $\Hc(W_d)$ has been well-studied, little is known about that of $\Uj$. This paper gives some fundamental results in the representation theory of $\Uj$ by using analogs of highest weight theory and Kashiwara's crystal basis theory.

In this paper, we treat the category $\Oj$ consisting of all $\Uj$-modules $M$ satisfying the following: $M$ is decomposed into its ``weight spaces", each of which is finite-dimensional; the set of weights of $M$ is bounded from above; $M$ is ``integrable".

We begin our study by decomposing $\Uj$ into three parts. This is an analog of the triangular decomposition of $\U$. Using this decomposition of $\Uj$, we define the ``Verma module" associated with each weight. By its definition and the triangular decomposition of $\Uj$, it possesses a unique irreducible quotient. Our first main result is the following classification theorem.

\begin{Thm}\label{classification theorem}
Every $\Uj$-module in $\Oj$ is completely reducible, and each irreducible $\Uj$-module in $\Oj$ is isomorphic to the irreducible quotient of a Verma module. Moreover, the isomorphism classes of irreducible $\Uj$-modules in $\Oj$ are parametrized by the pairs of partitions of length $r+1$ and $r$.
\end{Thm}

After submitting the first version of this paper, the author was told by one of the referees that the finite-dimensional irreducible modules of $\Uj_1$ had been classified in \cite{AKR17}. When $r = 1$, our classification theorem for the irreducible modules almost coincides with their result \cite[Corollary 3.6]{AKR17}. The main difference is that they only treat the finite-dimensional modules while we do not assume the finite-dimensionality in this paper. In fact, the category $\Oj$ has a module whose dimension is infinite.

Our second main result is the existence and uniqueness theorem of $\jmath$-crystal bases, which are analogs of Kashiwara's crystal bases. Before stating our result, let us briefly review the ordinary crystal basis theory; see the original paper \cite{K90}, or a textbook \cite{HK02} for details. Let $\bfA \subset \Q(q)$ be the ring of rational functions which are regular at $q = 0$. A crystal basis $(\clL,\clB)$ of a $\bfU$-module $M$ consists of a free $\bfA$-submodule $\clL \subset M$ of rank $\dim_{\Q(q)} M$ and a $\Q$-basis $\clB$ of $\clL/q\clL$ both of which are closed under the Kashiwara operators. We often identify a crystal basis $(\clL,\clB)$ with a colored directed graph called its crystal graph. It is known that every finite-dimensional irreducible $\bfU$-module have a unique crystal basis, which extracts a combinatorial property of the module structures of the irreducible module. Moreover, it is connected with a single source.

Following Kashiwara's crystal basis theory, we introduce the notions of $\jmath$-crystal basis and its $\jmath$-crystal graph. The second main result of this paper is the following.
\begin{Thm}
Each irreducible $\Uj$-module in $\Oj$ admits a unique $\jmath$-crystal basis. Moreover, its $\jmath$-crystal graph is connected with a single source.
\end{Thm}
This theorem and the complete reducibility of $\Uj$-modules lead to the existence and uniqueness of $\jmath$-crystal basis of an arbitrary $\Uj$-module in $\Oj$. Also, as in the ordinary crystal basis theory, $\jmath$-crystal bases have the tensor product rule.

\begin{Thm}\label{ThmC}
Let $M$ be a $\Uj$-module and $N$ a $\U$-module. Suppose that $M$ admits a $\jmath$-crystal basis $(\clL,\clB)$, and that $N$ has a crystal basis $(\clL',\clB')$. Then, $(\clL \otimes \clL', \clB \otimes \clB')$ is a $\jmath$-crystal basis of $M \otimes N$. In particular, (by taking $M$ to be the trivial $\Uj$-module) the crystal basis of a $\U$-module $N$ is the $\jmath$-crystal basis of $N$.
\end{Thm}

Here, let us recall a result about the representation theory of $\Uj$ from \cite{BW13}. There, Bao and Wang introduced the notion of $\jmath$-canonical basis for a finite-dimensional based $\U$-module (in the sense of \cite[Chapter 27]{L94}). They proved that a finite-dimensional based $\U$-module $(M,B)$ admits a unique $\jmath$-canonical basis $B^\jmath := \{ T_b \mid b \in B \}$ of the form
\begin{align}\label{T_b}
T_b = b + \sum_{b' \in B,\ b' \prec b} t_{b,b'}b', \qu t_{b,b'} \in q\Z[q],
\end{align}
where $\preceq$ denotes a partial order on $B$ (see \cite[Theorem 6.24]{BW13} for details). By equation \eqref{T_b}, the $\Z[q]$-span $\Z[q]B^\jmath$ of $B^\jmath$ coincides with that of $B$, and hence the set $\{ T_b + q\Z[q]B^\jmath \mid b \in B \}$ is the crystal basis of $M$. In particular, $\{ T_b + q\Z[q]B^\jmath \mid b \in B \}$ is the $\jmath$-crystal basis of $M$ by Theorem \ref{ThmC}. Thus, the $\jmath$-crystal basis of $M$ can be thought of as a localization of the $\jmath$-canonical basis. Note that the category $\Oj$ contains objects other than finite-dimensional based $\U$-modules. For those objects, the notion of $\jmath$-canonical basis has not been defined. We expect that we can ``globalize" the $\jmath$-crystal bases of such objects; namely, we expect that there exists a basis which we should call the $\jmath$-canonical basis for each module in $\Oj$. We will treat this problem in a future work.

Finally, we mention that the $\jmath$-crystal bases have rich combinatorial properties. In particular, the $\jmath$-crystal basis of an irreducible $\Uj$-module is realized as the set of pairs of semistandard Young tableaux of given shapes. As applications, we describe explicitly irreducible decompositions of $\bfV^{\otimes d}$ (Robinson-Schensted-type correspondence) and the tensor product of an irreducible $\Uj$-module with an irreducible $\U$-module (Littlewood-Richardson-type rule).

\subsection{Organization of the paper}
This paper consists of two parts. Part \ref{part 1} concerns a highest weight theory for $\Uj$. Section \ref{Basic} is devoted to introducing the quantum group $\U = \U_{2r+1} = U_q(\mathfrak{sl}_{2r+1})$, its coideal subalgebra $\Uj = \Uj_r$, and the category $\Oj$. Also, the triangular decomposition of $\Uj$ is formulated. We classify all the irreducible $\Uj$-modules in $\Oj$ and prove the complete reducibility of $\Uj$-modules in $\Oj$ for the case $r=1$ in Section \ref{r=1case}, and for a general $r$ in Section \ref{CompleteReducibility}.

The $\jmath$-crystal basis theory is developed in Part \ref{part 2}. After introducing basic notion of combinatorial tools in Section \ref{Combinatorics}, we introduce the notion of quasi-$\jmath$-crystal basis of $\Uj$-modules in $\Oj$ in a naive way in Section \ref{WeakCrystal}. In Section \ref{jcry}, we define $\jmath$-crystal bases as quasi-$\jmath$-crystal bases satisfying additional conditions, and prove the existence and uniqueness theorem for $\jmath$-crystal bases of $\Uj$-modules in $\Oj$. As applications of $\jmath$-crystal bases, Robinson-Schensted-type correspondence and Littlewood-Richardson-type rule for $\jmath$-crystal bases are studied in Section \ref{Applications}.

We end this paper by giving an explicit irreducible decomposition of a $(\Uj,\Hc(W_d))$-bimodule $\bfV^{\otimes d}$ in Appendix \ref{appxB}.

\subsection*{Acknowledgements}
The author should like to express his gratitude to Satoshi Naito for his many pieces of advice. He is also grateful to Hironori Oya for his helpful comments on the PBW-type basis of $\U$. He would like to thank the referees for giving many useful comments, which helped to improve this paper. This work was partially supported by JSPS KAKENHI Grant Number 17J00172.

\part{Highest weight theory for $\Uj$}\label{part 1}
Throughout this paper, $p$ and $q$ are independent indeterminates.
\section{Basics of the quantum symmetric pair $(\U,\Uj)$}\label{Basic}
\subsection{Definition of $\Uj$}
Let $r \geq 1$, and set
\begin{align}
\bbI := \left\{ -\left(r-\hf \right), \ldots, -\hf, \hf, \ldots, r - \hf \right\}, \qu \Ij := \{ 1, 2, \ldots, r \}. \nonumber
\end{align}
Let $\Phi$ denote the root system of type $A_{2r}$ with simple roots $\Pi = \{ \alpha_i := \epsilon_{i-\hf} - \epsilon_{i+\hf} \mid i \in \bbI \}$, where $\{ \epsilon_i \mid i = -r, -(r-1), \ldots, r \}$ is the standard basis of the Euclidean space $\R^{2r+1}$ with the inner product $(\cdot,\cdot)$; the associated Dynkin diagram is
\begin{align}
\xygraph{
    \bullet ([]!{+(0,-.3)} {-(r-\hf)}) - [r] \cdots - [r]
    \bullet ([]!{+(0,-.3)} {-\hf}) - [r]
    \bullet ([]!{+(0,-.3)} {\hf}) - [r] \cdots - [r]
    \bullet ([]!{+(0,-.3)} {r-\hf})}. \nonumber
\end{align}
We denote the set of positive roots by $\Phi_+$ and the weight lattice by $\Lambda = \bigoplus_{i = -r}^r \Z \epsilon_i$.

Let $\U = \U_{2r+1}$ denote the quantum group $U_q(\mathfrak{sl}_{2r+1})$ of type $A_{2r}$ over $\Q(p,q)$ with generators $E_i, F_i$, and $K_i^{\pm1}$, $i \in \bbI$, subject to the following relations:
\begin{align}
\begin{split}
&K_i K_i\inv = K_i\inv K_i = 1, \\
&K_i K_j = K_j K_i, \\
&K_i E_j K_i\inv = q^{(\alpha_i, \alpha_j)} E_j, \\
&K_i F_j K_i\inv = q^{-(\alpha_i, \alpha_j)} F_j, \\
&E_i F_j - F_j E_i = \delta_{i,j} \frac{K_i - K_i\inv}{q - q\inv}, \\
&E_i^2 E_j - (q+q\inv)E_iE_jE_i + E_jE_i^2 = 0 \qu \IF |i-j| = 1, \\
&F_i^2 F_j - (q+q\inv)F_iF_jF_i + F_jF_i^2 = 0 \qu \IF |i-j| = 1, \\
&E_i E_j - E_j E_i = 0 \qu \IF |i-j| > 1, \\
&F_i F_j - F_j F_i = 0 \qu \IF|i-j| > 1.
\end{split} \nonumber
\end{align}
Let $\U^+$ denote the subalgebra of $\U$ generated by $E_i$, $i \in \bbI$.

We employ the comultiplication $\Delta$ of $\U$ given by:
\begin{align}
\Delta(K_i^{\pm1}) = K_i^{\pm1} \otimes K_i^{\pm1}, \qu \Delta(E_i) = 1 \otimes E_i + E_i \otimes K_i\inv, \qu \Delta(F_i) = F_i \otimes 1 + K_i \otimes F_i \qu \text{for } i \in \bbI. \nonumber
\end{align}

Let $(\U,\Uj)$ denote the quantum symmetric pair (in the sense of \cite{Le99}) over $\Q(p,q)$ of type AIII, that is, $\Uj$ is the subalgebra of $\U$ generated by
\begin{align}
&k_i^{\pm1} := (K_{i-\hf}K_{-(i-\hf)}\inv)^{\pm1}, \nonumber\\
&e_i := E_{i-\hf} + p^{-\delta_{i,1}} F_{-(i-\hf)} K_{i-\hf}\inv, \nonumber\\
&f_i := E_{-(i-\hf)} + p^{\delta_{i,1}} K_{-(i-\hf)}\inv F_{i-\hf}, \nonumber
\end{align}
with $i \in \Ij$. When we want to emphasize the integer $r$, we denote this subalgebra by $\Uj_r$ instead of $\Uj$.

The $\Uj$ has the following defining relations (\cite{Le99}, see also \cite{BW13}, \cite{BWW16}): for $i,j \in \Ij$,
\begin{align}\label{defrel}
\begin{split}
&k_i k_i\inv = k_i\inv k_i = 1, \\
&k_i k_j = k_j k_i, \\
&k_i e_j k_i\inv = q^{(\alpha_{i-\hf} - \alpha_{-(i-\hf)}, \alpha_{j-\hf})} e_j, \\
&k_i f_j k_i\inv = q^{-(\alpha_{i-\hf} - \alpha_{-(i-\hf)}, \alpha_{j-\hf})} f_j, \\
&e_i f_j - f_j e_i = \delta_{i,j} \frac{k_i - k_i\inv}{q - q\inv} \qu \IF (i,j) \neq (1,1), \\
&e_i^2 e_j - (q+q\inv)e_ie_je_i + e_je_i^2 = 0 \qu \IF |i-j| = 1, \\
&f_i^2 f_j - (q+q\inv)f_if_jf_i + f_jf_i^2 = 0 \qu \IF |i-j| = 1, \\
&e_i e_j - e_j e_i = 0 \qu \IF |i-j| > 1, \\
&f_i f_j - f_j f_i = 0 \qu \IF |i-j| > 1, \\
&e_1^2 f_1 - (q+q\inv)e_1f_1e_1 + f_1e_1^2 = -(q+q\inv) e_1 (pqk_1 + p\inv q\inv k_1\inv), \\
&f_1^2 e_1 - (q+q\inv)f_1e_1f_1 + e_1f_1^2 = -(q+q\inv) (pqk_1 + p\inv q\inv k_1\inv) f_1.
\end{split}
\end{align}

Also, $\Uj$ is a right coideal of $\U$, that is, $\Delta(\Uj) \subset \Uj \otimes \U$. Indeed, we have
\begin{align}\label{comult}
\begin{split}
&\Delta(k_i^{\pm 1}) = k_i^{\pm 1} \otimes k_i^{\pm 1}, \\
&\Delta(e_i) = e_i \otimes K_{i-\hf}\inv + 1 \otimes E_{i-\hf} + p^{-\delta_{i,1}} k_i\inv \otimes F_{-(i-\hf)} K_{i-\hf}\inv, \\
&\Delta(f_i) = f_i \otimes K_{-(i-\hf)}\inv + 1 \otimes E_{-(i-\hf)} + p^{\delta_{i,1}} k_i \otimes K_{-(i-\hf)}\inv F_{i-\hf} \qu \text{for } i \in \Ij.
\end{split}
\end{align}
This fact enables us to regard the tensor product $M \otimes N$ of a $\Uj$-module $M$ and a $\U$-module $N$ as a $\Uj$-module. Thanks to the coassociativity of $\Delta$, we have a natural isomorphism $M \otimes (N_1 \otimes N_2) \simeq (M \otimes N_1) \otimes N_2$ of $\Uj$-modules, where $N_1$ and  $N_2$ are $\U$-modules.

\begin{prop}
\begin{enumerate}\label{automorphisms}
\item \cite[Lemma 6.1 (3)]{BW13} There exists a unique $\Q$-algebra automorphism $\psij$ of $\Uj$ which maps $e_i,f_i,k_i,p,q$ to $e_i,f_i,k_i\inv,p\inv,q\inv$, respectively.
\item There exists a unique $\Q(p,q)$-algebra anti-automorphism $\sigmaj$ of $\Uj$ which maps $e_i,f_i,k_i$ to $f_i,e_i,k_i$, respectively.
\end{enumerate}
\end{prop}

\begin{proof}
These assertions are easily verified by the defining relations \eqref{defrel} of $\Uj$.
\end{proof}

For notational simplicity, we write $\overline{x}$ instead of $\psij(x)$ for $x \in \Uj$; it should be noted that $\psij$ is different from the restriction of the bar-involution of $\U$, which we will not use in this paper.

\subsection{Triangular decomposition of $\Uj$}
Recall Lusztig's braid group actions on $\U$.

\begin{defi}[{\cite[Chapter 37]{L94}}]\label{BraidGroupAction}\normalfont
Let $e \in \{ 1, -1 \}$. For each $i \in \bbI$, define four automorphisms $T'_{i,e}$ and $T''_{i,-e}$ on $\U$ by:
\begin{align}
T'_{i,e}(E_j) &= \begin{cases}
-K_i^e F_i & \text{if } j = i, \\
E_j \qu & \text{if } |i - j| > 1, \\
[E_j, E_i]_e \qu & \text{if } |i - j| = 1,
\end{cases} \qu \qu
T'_{i,e}(F_j) = \begin{cases}
-E_i K_i^{-e} \qu & \text{if } j = i, \\
F_j \qu & \text{if } |i - j| > 1, \\
[F_i, F_j]_{-e} \qu & \text{if } |i - j| = 1,
\end{cases} 
\nonumber\\
T''_{i,-e}(E_j) &= \begin{cases}
-F_i K_i^{-e} \qu & \text{if } j = i, \\
E_j \qu & \text{if } |i - j| > 1, \\
[E_i, E_j]_e \qu & \text{if } |i - j| = 1,
\end{cases} \qu \qu
T''_{i,-e}(F_j) = \begin{cases}
-K_i^{e} E_i \qu & \text{if } j = i, \\
F_j \qu & \text{if } |i - j| > 1, \\
[F_j, F_i]_{-e} \qu & \text{if } |i - j| = 1,
\end{cases} \nonumber\\
T'_{i,e}(K_j) &= T''_{i,-e}(K_j) = \begin{cases}
K_i\inv \qu & \text{if } j = i, \\
K_j \qu & \text{if } |i - j| > 1, \\
K_i K_j \qu & \text{if } |i - j| = 1;
\end{cases} \nonumber
\end{align}
\end{defi}

\noindent here we set $[X,Y]_e := XY - q^eYX$. By {\cite[Theorem 39.4.3]{L94}}, for each $e \in \{1, -1 \}$, the families $\{ T'_{i,e} \mid i \in \bbI \}$ and $\{ T''_{i,-e} \mid i \in \bbI \}$ both satisfy the braid relation of type $A_{2r}$. Let $W(\bbI)$ denote the Weyl group of type $A_{2r}$ with simple reflections $\{ s_i \mid i \in \bbI \}$. Then the PBW-type basis of $\U^+$ is described as follows.

\begin{defi}\normalfont
Let $\bfi = (i_1, \ldots, i_N)$ be a reduced word for the longest element $w_0 \in W(\bbI)$. The root vectors $E_j(\bfi)$, $j = 1, \ldots, N$, associated with $\bfi$ are given by:
\begin{align}
E_1(\bfi) = E_{i_1}, \qu\qu E_j(\bfi) = T''_{i_1,1} \cdots T''_{i_{j-1},1}(E_{i_j}). \nonumber
\end{align}
For each positive root $\alpha$, we set $E_\alpha(\bfi) := E_j(\bfi)$ if $\alpha = s_{i_1} \cdots s_{i_{j-1}}(\alpha_{i_j})$.
\end{defi}

\begin{theo}[{\cite[4.2]{L90}}]
Let $\bfi = (i_1, \ldots, i_N)$ be a reduced word for $w_0 \in W(\bbI)$. Then, the ordered monomials in the root vectors associated with $\bfi$ form a linear basis of $\U^+$.
\end{theo}

We define a filtration of $\Uj$ by setting $\deg(e_i) = \deg(f_i) = 1$ and $\deg(k_i) = 0$ for $i \in \Ij$, and set $\Uj(m)$ to be the subspace of $\Uj$ spanned by the elements of the form $x_{i_1} \cdots x_{i_l}$ with $x \in \{ e,f,k^{\pm1} \}$, $i_j \in \Ij$, and $\sum_{j} \deg(x_{i_j}) \leq m$. Set $\gr\Uj := \bigoplus_{m} \Uj(m) / \Uj(m-1)$ to be the associated graded algebra. Let $\Ujo$ denote the subalgebra of $\bfU$ generated by $k_i^{\pm1}$, $i \in \Ij$. Then, by the defining relation (equation \eqref{defrel}) of $\Uj$, there exists a unique surjective algebra homomorphism $\bfU^+ \otimes \Ujo \rightarrow \gr\Uj$ given by
$$
E_{i-\hf} \mapsto e_i, \qu E_{-(i-\hf)} \mapsto f_i, \qu k_i^{\pm1} \mapsto k_i^{\pm1}.
$$
We define a surjective linear map $\gr : \U^+ \bigotimes \Ujo \rightarrow \Uj$ to be the composite map of the surjection $\U^+ \bigotimes \Ujo \rightarrow \gr \Uj$ and the linear isomorphism $\gr\Uj \simeq \Uj$.

Recall that $\Phi_+ = \{ \epsilon_i - \epsilon_j \mid -r \leq i < j \leq r \}$ denotes the set of positive roots of $\Phi$ with respect to the simple roots $\Pi = \{ \epsilon_i - \epsilon_{i+1} \mid -r \leq i < r \}$. We decompose $\Phi_+$ into three parts as:
\begin{align}
\Phi_+ &= \Phi_{< 0} \sqcup \Phi_0 \sqcup \Phi_{> 0}, \nonumber\\
\Phi_{< 0} &:= \{ \epsilon_i - \epsilon_j \mid i + j < 0 \}, \nonumber\\
\Phi_{0} &:= \{ \epsilon_i - \epsilon_j \mid i + j = 0 \}, \nonumber\\
\Phi_{> 0} &:= \{ \epsilon_i - \epsilon_j \mid i + j > 0 \}. \nonumber
\end{align}
For example, when $r=3$, the positive roots are displayed as follows:
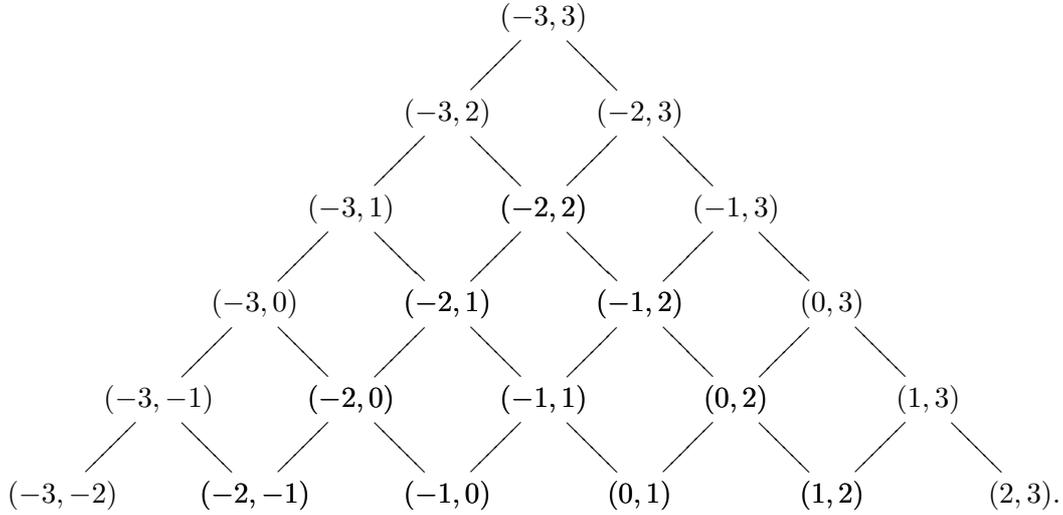
\begin{align}
\xygraph{
{(-3,3)} (
             -[ld]{(-3,2)} (
                                -[ld]{(-3,1)} (
                                                   -[ld]{(-3,0)} (
                                                                      -[ld]{(-3,-1)} (-[ld]{(-3,-2)},-[rd]{(-2,-1)}),
                                                                      -[rd]{(-2,0)} (-[ld]{(-2,-1)},-[rd]{(-1,0)})),
                                                   -[rd]{(-2,1)} (
                                                                      -[ld]{(-2,0)},
                                                                      -[rd]{(-1,1)} (-[ld]{(-1,0)},-[rd]{(0,1)}))),
                                -[rd]{(-2,2)} (
                                                   -[ld]{(-2,1)},
                                                   -[rd]{(-1,2)} (
                                                                      -[ld]{(-1,1)},
                                                                      -[rd]{(0,2)} (-[ld]{(0,1)},-[rd]{(1,2)})))),
             -[rd]{(-2,3)} (
                                -[ld]{(-2,2)},
                                -[rd]{(-1,3)} (
                                                   -[ld]{(-1,2)},
                                                   -[rd]{(0,3)} (
                                                                      -[ld]{(0,2)},
                                                                      -[rd]{(1,3)} (-[ld]{(1,2)},-[rd]{(2,3).})))))
} \nonumber
\end{align}
Here, $(i,j)$ represents $\epsilon_i - \epsilon_j$. Then, the roots in $\Phi_0$ lie on the vertical line through $(-3,3)$, those in $\Phi_{<0}$ on the left to the line, and those in $\Phi_{>0}$ on the right.

Here, we recall the notion of reflection orders (or convex orders).

\begin{defi}\normalfont
A total order $\preceq$ on $\Phi_+$ is said to be a reflection order if it satisfies the following: for each $\alpha, \beta \in \Phi_+$ and $a, b \in \R_{>0}$, if $a \alpha + b \beta \in \Phi_+$ and $\alpha \prec \beta$, then $\alpha \prec a \alpha + b \beta \prec \beta$.
\end{defi}

\begin{prop}[{\cite[Proposition 2.13]{D93}}]\label{reflection orders}
Let $\bfi = (i_1, \ldots, i_N)$ be a reduced word for $w_0 \in W(\bbI)$. Set $\alpha_j(\bfi) := s_{i_1} \cdots s_{i_{j-1}}(\alpha_{i_j})$. Then, the total order $\preceq$ on $\Phi_+$ defined by $\alpha_1(\bfi) \prec \cdots \prec \alpha_N(\bfi)$ is a reflection order. Moreover, this correspondence gives a bijection between the set of reduced words for $w_0 \in W(\bbI)$ and the set of reflection orders on $\Phi_+$.
\end{prop}

\begin{lem}\label{reforder}
There exists a reflection order $\preceq$ on $\Phi_+$ such that
\begin{align}\label{AssumptionOnReforder}
\Phi_{<0} \prec \Phi_0 \prec \Phi_{>0}.
\end{align}
Here, for subsets $A, B \subset \Phi_+$, $A \prec B$ means that $\alpha \prec \beta$ for all $\alpha \in A$ and $\beta \in B$.
\end{lem}

\begin{proof}
It suffices to construct such a reflection order. For simplicity, we write $(i,j)$ instead of $\epsilon_i -\epsilon_j$ for $i < j$. We decompose $\Phi_{< 0}$ into $\Phi_{<0,-} := \{ (i,j) \in \Phi_{<0} \mid j \leq 0 \}$ and $\Phi_{<0,+} := \{ (i,j) \in \Phi_{<0} \mid j > 0 \}$. Similarly, we set $\Phi_{>0,-} := \{ (i,j) \in \Phi_{>0} \mid i < 0 \}$ and $\Phi_{>0,+} := \{ (i,j) \in \Phi_{>0} \mid i \geq 0 \}$. Let us define a total order $\preceq$ on $\Phi_+$ by:
\begin{enumerate}
\item $\Phi_{<0,-} \prec \Phi_{<0,+} \prec \Phi_{0} \prec \Phi_{>0,-} \prec \Phi_{>0,+}$;
\item for $(i,j), (i',j') \in \Phi_{<0,-}$, $(i,j) \prec (i',j')$ if and only if $i < i'$ or ($i = i'$ and $j < j'$);
\item for $(i,j), (i',j') \in \Phi_{<0,+}$, $(i,j) \prec (i',j')$ if and only if $j < j'$ or ($j = j'$ and $i < i'$);
\item for $(i,j), (i',j') \in \Phi_{0}$, $(i,j) \prec (i',j')$ if and only if $j < j'$;
\item for $(i,j), (i',j') \in \Phi_{>0,-}$, $(i,j) \prec (i',j')$ if and only if $i < i'$ or ($i = i'$ and $j < j'$);
\item for $(i,j), (i',j') \in \Phi_{>0,+}$, $(i,j) \prec (i',j')$ if and only if $j < j'$ or ($j = j'$ and $i < i'$).
\end{enumerate}
Then, $\preceq$ is a reflection order on $\Phi_+$ satisfying $\Phi_{<0} \prec \Phi_0 \prec \Phi_{>0}$; the proof is straightforward.
\end{proof}

\begin{ex}\normalfont
When $r=3$, this total order is given as follows:
\begin{align}
\begin{split}
&(-3,-2) \prec (-3,-1) \prec (-3,0) \prec (-2,-1) \prec (-2,0) \prec (-1,0) \\
\prec &(-3,1) \prec (-2,1) \prec (-3,2) \\
\prec &(-1,1) \prec (-2,2) \prec (-3,3) \\
\prec &(-2,3) \prec (-1,2) \prec (-1,3) \\
\prec &(0,1) \prec (0,2) \prec (1,2) \prec (0,3) \prec (1,3) \prec (2,3).
\end{split} \nonumber
\end{align}
\end{ex}

Fix a reflection order $\preceq$ satisfying condition \eqref{AssumptionOnReforder} in Lemma \ref{reforder}. Let $\bfi$ be the reduced word for $w_0 \in W(\bbI)$ corresponding to $\preceq$ under the bijection of Proposition \ref{reflection orders}. We set $E_{i,j} := E_{\epsilon_i-\epsilon_j}(\bfi)$ for $-r \leq i < j \leq r$. For each $i,j$, define $E'_{i,j} := \gr(E_{i,j})$, and set
\begin{align}
f_{-j,-i} := E'_{i,j} \ \IF i + j < 0, \qu h'_i := E'_{-i,i}, \qu e_{i,j} := E'_{i,j} \ \IF i + j > 0. \nonumber
\end{align}

\begin{prop}
The surjection $\gr: \U^+ \otimes \Ujo \rightarrow \Uj$ is a linear isomorphism.
\end{prop}

\begin{proof}
By the construction, each $E'_{i,j}$ is of the form $E'_{i,j} \in E_{i,j} + \sum_{\nu} \bfU_{\nu}$, where $\nu$ runs through all integral weight lower than $\epsilon_i-\epsilon_j$, and $\bfU_{\nu}$ denote the weight space of $\bfU$ of weight $\nu$. More generally, each ordered monomial $\prod_{i < j} (E'_{i,j})^{a_{i,j}}$, $a_{i,j} \in \Z_{\geq 0}$ is the sum of $\prod_{i<j} E_{i,j}^{a_{i,j}}$ and a linear combination of weight vectors in $\bfU$ of weight lower than $\sum_{i<j} a_{i,j}(\epsilon_i-\epsilon_j)$. This observation implies the following: since the ordered monomials in $E_{i,j}$'s form a linear basis of $\bfU^+$, the ordered monomials in $E'_{i,j}$'s are linearly independent. In particular, the surjection $\gr:\bfU^+ \otimes \Ujo \rightarrow \Uj$ is a linear isomorphism.
\end{proof}

\begin{cor}
The ordered monomials $\left( \prod_{i+j<0} f_{-j,-i}^{a_{i,j}} \right) \left( \prod_i (h'_i)^{b_i} \right) \left( \prod_{i=1}^r k_i^{d_i} \right) \left( \prod_{i+j>0} e_{i,j}^{c_{i,j}} \right)$, $a_{i,j},b_i,c_{i,j} \in \Z_{\geq 0}$, $d_i \in \Z$ form a linear basis of $\Uj$.
\end{cor}

Let us compute some of the root vectors. By {\cite[Lemma 1]{LS91}} (with a slight modification), we have
\begin{align}
\begin{split}
E_{i-1,j} &= [E_{i,j},E_{i-\hf}]_{1} \qu \IF (i-1,i) \prec (i,j), \\
E_{i,j+1} &= [E_{j+\hf},E_{i,j}]_{1} \qu \IF (i,j) \prec (j,j+1). 
\end{split} \nonumber
\end{align}
In particular, it holds that
\begin{align}
E_{-1,1} = [E_{\hf},E_{-\hf}]_{1}, \qu E_{-(i+1),i+1} = \left[[E_{(i+\hf)},E_{-i,i}]_{1},E_{-(i+\hf)} \right]_{1} \qu \text{for } 1 \leq i \leq r-1. \nonumber
\end{align}
Applying $\gr$, we obtain
\begin{align}\label{h_i's}
h'_1 = [e_1,f_1]_{1}, \qu h'_{i+1} = \bigl[[e_{i+1},h'_i]_{1},f_{i+1} \bigr]_{1}.
\end{align}
This shows that the $h'_i$'s are independent of the choice of a reflection order $\preceq$ satisfying condition \eqref{AssumptionOnReforder} in Lemma \ref{reforder}.

Let $\Uj_{<0}$ (resp., $\Uj_0, \Uj_{>0}$) denote the subspace of $\Uj$ spanned by all ordered monomials in $f_{-j,-i}$ (resp., $h_i, e_{i,j}$). Then, we have an isomorphism of vector spaces
\begin{align}
\Uj \simeq \Uj_{<0} \bigotimes \left( \Uj_0 \bigotimes \Ujo \right) \bigotimes \Uj_{>0}. \nonumber
\end{align}
We call this linear isomorphism the triangular decomposition of $\Uj$ associated with the reflection order $\preceq$, and $\Uj_{<0}$ (resp., $\Uj_0 \bigotimes \Ujo$, $\Uj_{>0}$) the negative part (resp., Cartan part, positive part) of $\Uj$. The triangular decomposition enables us to establish an analog of highest weight theory for the representation theory of $\Uj$.

\begin{rem}\normalfont
Unlike the ordinary triangular decomposition of a quantum group, the negative part, the Cartan part, and the positive part of $\Uj$ are just subspaces, not subalgebras. In addition, the negative part and the positive part may depend on the choice of a reflection order. However, by equation \eqref{h_i's}, the Cartan part is independent of such a choice.
\end{rem}

\subsection{Verma modules and their irreducible quotients}\label{2.3}
Recall that $\R^{2r+1} = \bigoplus_{i = -r}^r \R \epsilon_i$ is the Euclidean space with standard basis $\{ \epsilon_i \mid -r \leq i \leq r \}$ with respect to the inner product $(\cdot,\cdot)$, and $\alpha_i = \epsilon_{i-\hf} - \epsilon_{i+\hf}$, $i \in \bbI$, are the simple roots. Set $\beta_i := \alpha_{i-\hf} - \alpha_{-(i-\hf)} = \epsilon_{i-1} - \epsilon_i - \epsilon_{-i} + \epsilon_{-(i-1)}$ for $i \in \Ij$.

\begin{defi}\normalfont
Let $J := \{ \lambda \in \R^{2r+1} \mid (\beta_i, \lambda) = 0 \text{ for all } i \in \Ij \}$. Then the bilinear form $(\cdot, \cdot)$ on $\R^{2r+1} \times \R^{2r+1}$ induces a bilinear map $\left( \bigoplus_{i \in \Ij} \R \beta_i \right) \times \left( \R^{2r+1}/J \right) \rightarrow \R$, which we also denote by $(\cdot, \cdot)$. For each $i \in \Ij$, there exists a unique $\delta_i \in \R^{2r+1}/J$ such that
\begin{align}
(\beta_j, \delta_i) = \delta_{i,j} \qu \text{ for } i,j \in \Ij. \nonumber
\end{align}
Set $\Lambdaj := \sum_{i \in \Ij} \Z \delta_i$ and $\Lambdaj_+ := \sum_{i \in \Ij} \Z_{\geq 0} \delta_i$. Also, we set $\gamma_i := \epsilon_{i-1} - \epsilon_{i} + J \in \Lambdaj$.
\end{defi}
By the definitions, we have
\begin{align}
(\beta_i,\gamma_j) = (\alpha_{i-\hf}-\alpha_{-(i-\hf)},\alpha_{j-\hf}) = \begin{cases}
3 & \IF i = j = 1, \\
2 & \IF i = j \neq 1, \\
-1 & \IF |i-j| = 1, \\
0 & \IF |i-j| > 1.
\end{cases} \nonumber
\end{align}
Define a partial order $\leq$ on $\Lambdaj$ by:
\begin{align}\label{dominance order}
\mu \leq \lambda \IFF \lambda - \mu \in \sum_{i \in \Ij} \Z_{\geq 0} \gamma_i.
\end{align}

For a $\Uj$-module $M$ and $m \in M$, we say that $m$ is of weight $\lambda \in \Lambda^\jmath$ if it satisfies
\begin{align}
k_i m = q^{(\beta_i , \lambda)}m \nonumber
\end{align}
for all $i \in \Ij$; we denote by $M_\lambda$ the subspace consisting of all $m \in M$ of weight $\lambda$.

\begin{lem}
Let $M$ be a $\Uj$-module and $\lambda \in \Lambdaj$. For each $i \in \Ij$, we have
\begin{align}
f_i(M_\lambda) \subset M_{\lambda - \gamma_i}, \qqu e_i(M_\lambda) \subset M_{\lambda+\gamma_i}. \nonumber
\end{align}
\end{lem}

\begin{proof}
This follows immediately from the relations $k_i f_j k_i\inv = q^{(\beta_i,-\gamma_j)} f_j$ and $k_i e_j k_i\inv = q^{(\beta_i,\gamma_j)} e_j$.
\end{proof}

Recall the triangular decomposition of $\Uj$
\begin{align}
\Uj \simeq \Uj_{<0} \bigotimes \left( \Uj_0 \bigotimes \Ujo \right) \bigotimes \Uj_{>0}, \nonumber
\end{align}
and the root vectors $f_{-j,-i}, h'_i, e_{i,j}$ associated with a reflection order satisfying condition \eqref{AssumptionOnReforder} in Lemma \ref{reforder}. Let $(\Uj_{>0})_+$ denote the subspace of $\Uj_{>0}$ spanned by all ordered monomials in $e_{i,j}$'s other than $1$.

\begin{defi}\normalfont
Let $\lambda \in \Lambda^\jmath$ and $H'_i \in \Q(p,q)$, $i = 1,2, \ldots ,r$. The Verma module $V'(\lambda;\bfH')$ over $\Uj$ with highest weight $\lambda$ associated with $\bfH' := (H'_1,\ldots,H'_r) \in \Q(p,q)^r$ is defined to be
\begin{align}
V'(\lambda;\bfH') := \Uj/I(\lambda;\bfH'), \nonumber
\end{align}
where $I(\lambda;\bfH')$ denotes the left ideal of $\Uj$ generated by $(\Uj_{>0})_+$ and $k_i - q^{(\beta_i,\lambda)}$, $h'_i - H'_i$ for $i \in \Ij$.
\end{defi}

\begin{rem}\normalfont
Verma modules $V'(\lm;\bfH')$ can be $0$.
\end{rem}

By the triangular decomposition of $\Uj$, a nonzero Verma module $V'(\lambda;\bfH')$ has a unique maximal submodule, and hence, it has a unique irreducible quotient. We denote it by $L'(\lambda;\bfH')$ and call it the irreducible highest weight $\Uj$-module with highest weight $\lambda$ associated with $\bfH'$, or simply, with highest weight $(\lambda;\bfH')$.

\begin{defi}\normalfont
A nonzero $\Uj$-module $M$ is called a highest weight module with highest weight $(\lambda;\bfH') \in \Lambdaj \times \Q(p,q)^r$ if there exists $m \in M_\lambda$ such that $(\Uj_{>0})_+m = 0$, $h'_im = H'_im$ for $i \in \Ij$, and $M = \Uj m$. We call such an $m$ a highest weight vector of $M$ with highest weight $(\lambda;\bfH')$.
\end{defi}

Our definition of highest weight modules over $\Uj$ depends on the choice of a reflection order satisfying condition \eqref{AssumptionOnReforder} in Lemma \ref{reforder}. However, their $\Uj$-module structure is independent of such a choice, as we explain below.

Let $M$ be a highest weight $\Uj$-module with highest weight $(\lambda;\bfH')$ associated with a reflection order $\preceq$. Let $v \in M$ be a highest weight vector. Take another reflection order $\preceq'$, and denote the corresponding root vectors by $f_{i,j}', h''_i, e_{i,j}'$. Then, we see from equation \eqref{h_i's} that $h''_i = h'_i$. Also, by the triangular decomposition associated with $\prec$, we have
\begin{align}
e'_{i,j} \in \sum_{\substack{\nu, \mu \in \Lambdaj_{+} \\ \nu < \mu}} (\Uj_{<0})_{-\nu} \otimes (\Uj_0 \otimes \Ujo) \otimes (\Uj_{>0})_\mu; \nonumber
\end{align}
here, $(\Uj_{< 0})_{-\nu} := \{ x \in \Uj_{<0} \mid k_i x k_i\inv = q^{(\beta_i,-\nu)}x \text{ for all } i \in \Ij \}$, and define $(\Uj_{>0})_\mu$ similarly. Therefore, it holds that $e'_{i,j}v = 0$ for all $i,j$. In addition, by expanding $f_{i,j}$ in ordered monomials in $f'_{i,j},h''_i,e'_{i,j}$, we see that $f_{i,j}v$ is a linear combination of $f'_{i,j}v$'s. From these, we conclude that $M$ is a highest weight module with highest weight $(\lambda;\bfH')$ associated with $\preceq'$. In particular, if we denote Verma modules and their irreducible quotients associated with $\preceq'$ by $V''(\cdot;\cdot)$ and $L''(\cdot;\cdot)$, respectively, then we have
\begin{align}
V'(\lambda;\bfH') = V''(\lambda;\bfH'), \qu L'(\lambda;\bfH') = L''(\lambda;\bfH'). \nonumber
\end{align}
Hence, in what follows, we use only the reflection order given in the proof of Lemma \ref{reforder}.

Let $\Oj$ denote the category of all $\Uj$-modules $M$ satisfying the following:
\begin{enumerate}
\item[$(M1)$] $M$ is decomposed into weight spaces, i.e., $M = \bigoplus_{\lambda \in \Lambda^\jmath} M_\lambda$.
\item[$(M2)$] Each weight space is finite-dimensional.
\item[$(M3)$] There exist finitely many weights $\mu_1, \ldots \mu_n \in \Lambdaj$ such that each weight $\lambda \in \Lambdaj$ for which $M_\lambda \neq 0$ satisfies $\lambda \leq \mu_i$ for some $i = 1, \ldots, n$.
\item[$(M4)$] $e_i$ and $f_i$ act on $M$ locally nilpotently, that is, for each $m \in M$, there exists $N \in \N$ such that $e_i^Nm = 0 = f_i^Nm$.
\end{enumerate}
Note that Verma modules and their irreducible quotients are not necessarily objects of $\Oj$, i.e., the actions of $f_i$ on these modules are not always locally nilpotent. Also, $\Oj$ has an infinite-dimensional object as we will see in the next section.

\section{The case $r = 1$}\label{r=1case}
\subsection{Classification of the irreducible modules in $\Oj$}
We introduce some more notation.
\begin{defi}\normalfont
\begin{enumerate}
\item[$(1)$] For $n \in \Z$, $[n] := \frac{q^n - q^{-n}}{q - q\inv}$.
\item[$(2)$] For $n \in \Z_{>0}$, $[n]! := \prod_{i = 1}^n [i]$; we set $[0]! := 1$.
\item[$(3)$] For $x \in \U$ and $n \in \Z_{>0}$, $x^{(n)} := \frac{x^n}{[n]!}$; we set $x^{(0)} := 1$, and $x^{(n)} := 0$ if $n < 0$.
\item[$(4)$] For $x, y \in \U$ and $a \in \Z$, $[x,y]_a := xy - q^ayx$.
\item[$(5)$] For an invertible element $h$, $\{ h \} := h + h\inv$.
\item[$(6)$] For an integer $n \in \Z$, $\{ n \} := \{ pq^n \} = pq^n + p\inv q^{-n}$.
\end{enumerate}
\end{defi}
In the case $r=1$, the root vectors are
\begin{align}
f_{0,1} = f_1, \qu h'_1 = [e_1, f_1]_1, \qu e_{0,1} = e_1. \nonumber
\end{align}

\begin{lem}\label{h1}
In $\Uj_1$, we have
\begin{align}
[h'_1,f_1]_{-1} = - [2] \{ pqk_1 \} f_1, \qu [e_1,h'_1]_{-1} = - [2] e_1 \{ pqk_1 \}. \nonumber
\end{align}
\end{lem}
\begin{proof}
By equation \eqref{defrel}.
\end{proof}

\begin{lem}\label{e1f1^n}
For each $n \in \Z_{\geq 0}$, we have
\begin{align}
e_1 f_1^{(n)} = f_1^{(n-1)} \left(h'_1 - [n-1]\{ pq^{-n}k_1 \} \right) + q^n f_1^{(n)}e_1. \nonumber
\end{align}
\end{lem}

\begin{proof}
We prove the assertion by induction on $n$. This is trivial when $n = 0$. Assume that the assertion holds for a fixed $n \in \Z_{\geq 0}$. Then, we compute as follows:
\begin{align}
e_1 f_1^{(n+1)} &= \frac{1}{[n+1]} e_1 f_1^{(n)} f_1 \nonumber\\
&= \frac{1}{[n+1]} \left(f_1^{(n-1)} \left(h'_1 - [n-1]\{ pq^{-n}k_1 \} \right) + q^n f_1^{(n)}e_1 \right) f_1 \nonumber\\
&= \frac{1}{[n+1]} \left(f_1^{(n-1)} \left(q\inv f_1 h'_1 - [2] \{ pqk_1 \} f_1 - [n-1]\{ pq^{-n}k_1 \} f_1 \right) + q^n f_1^{(n)} (h'_1 + q f_1 e_1) \right) \nonumber\\
&= \frac{1}{[n+1]} \left(f_1^{(n-1)} \left(q\inv f_1 h'_1 - [2] f_1 \{ pq^{-2}k_1 \} - [n-1] f_1 \{ pq^{-n-3}k_1 \} \right) + q^n f_1^{(n)} (h'_1 + q f_1 e_1) \right) \nonumber\\
&= \frac{1}{[n+1]} f_1^{(n)} \left(q\inv [n] h'_1 - [2][n] \{ pq^{-2}k_1 \} - [n-1][n] \{ pq^{-n-3}k_1 \} + q^n h'_1 \right) + q^{n+1} f_1^{(n+1)} e_1 \nonumber\\
&= f_1^{(n)} \left(h'_1 - [n]\{ pq^{-n-1}k_1 \} \right) + q^{n+1} f_1^{(n+1)}e_1; \nonumber
\end{align}
the second equality follows from our inductive hypothesis, the third from Lemma \ref{h1}, and the rest is straightforward. This proves the lemma.
\end{proof}

Note that when $r = 1$, we have $\Lambdaj = \Z \delta_1$ and $\gamma_1 = 3\delta_1$. Let $M \in \Oj$. By the definition of $\Oj$, there exists $a \in \Z$ such that $M_{a \delta_1} \neq \{ 0 \}$ and $M_{(a+3) \delta_1} = \{0\}$. Since the action of $h'_1$ preserves weights, it defines a linear endomorphism of $M_{a \delta_1}$. In order to consider the Jordan canonical form for the action of $h'_1$ on $M_{a \delta_1}$, we extend the base field $\Q(p,q)$ to its algebraic closure $\overline{\Q(p,q)}$ until the proof of Proposition \ref{EigenvalueOfh1}.  Let us write the Jordan canonical form as:
\begin{align}
\begin{pmatrix}
J_{d_1}(\mu_1) & & & \\
& J_{d_2}(\mu_2) & & \\
& & \ddots & \\
& & & J_{d_m}(\mu_m)
\end{pmatrix}, \nonumber
\end{align}
where $J_{d_i}(\mu_i)$ denotes the Jordan block of size $d_i$ whose eigenvalue is $\mu_i \in \overline{\Q(p,q)}$. We take a basis $\{ v_{j, k} \mid j = 1, \ldots, m,\ k = 1, \ldots, d_j \}$ of $M_{a \delta_1}$ in such a way that
\begin{align}
h'_1 v_{j, k} = \mu_j v_{j, k} + v_{j, k-1} \nonumber
\end{align}
for all $j = 1, \ldots, m,\ k = 1, \ldots, d_j$, where $v_{j, 0} := 0$. By Lemma \ref{e1f1^n}, we have
\begin{align}\label{e1f1^n'}
e_1 f_1^{(n)} v_{j,k} = (\mu_j - [n-1] \{a-n\}) f_1^{(n-1)} v_{j,k} + f_1^{(n-1)} v_{j,k-1}.
\end{align}

\begin{prop}\label{EigenvalueOfh1}
We have $\mu_j = [N_j] \{a-N_j-1\}$ for some $N_j \in \Z_{\geq 0}$. In particular, each $\mu_j$ belongs to $\Q(p,q)$.
\end{prop}

\begin{proof}
Consider the case $k = 1$. By the local nilpotency of $f_1$, there exists a unique nonnegative integer $N_j$ such that
\begin{align}
f_1^{(N_j)} v_{j,1} \neq 0 \AND f_1^{(N_j+1)} v_{j,1} = 0. \nonumber
\end{align}
Then, by equation \eqref{e1f1^n'}, we have
\begin{align}
0 = e_1 f_1^{(N_j+1)} v_{j,1} = (\mu_j - [N_j] \{a-N_j-1\}) f_1^{(N_j)} v_{j,1}. \nonumber
\end{align}
Since $f_1^{(N_j)} v_{j,1} \neq 0$, we conclude that $\mu_j = [N_j] \{a-N_j-1\}$, as desired.
\end{proof}

\begin{prop}\label{diagonalizable}
Each $d_j$ is equal to $1$, that is, $h'_1$ is diagonalizable on $M_{a \delta_1}$.
\end{prop}

\begin{proof}
We use the notation $N_j$ in the proof of Proposition \ref{EigenvalueOfh1}. Assume, for a contradiction, that there exists $d_j > 1$. By equation \eqref{e1f1^n'}, we have
\begin{align}
e_1 f_1^{(n)} v_{j,2} = (\mu_j - [n - 1]\{a-n\}) f_1^{(n-1)} v_{j,2} + f_1^{(n-1)} v_{j,1} \nonumber
\end{align}
for all $n \geq 0$. Let $N'_j$ denote the unique nonnegative integer such that
\begin{align}
f_1^{(N'_j)} v_{j,2} \neq 0, \AND f_1^{(N'_j+1)} v_{j,2} = 0. \nonumber
\end{align}
When $N'_j > N_j$, we have
\begin{align}
0 = (\mu_j - [N'_j]\{a-N'_j-1\}) f_1^{(N'_j)} v_{j,2} + f_1^{(N'_j)} v_{j,1} = (\mu_j - [N'_j]\{a-N'_j-1\}) f_1^{(N'_j)} v_{j,2}. \nonumber
\end{align}
This implies that $\mu_j = [N'_j]\{a-N'_j-1\} \neq [N_j]\{a-N_j-1\}$, which causes a contradiction. When $N'_j = N_j$, we have
\begin{align}
0 = (\mu_j - [N_j]\{a-N_j-1\}) f_1^{(N_j)} v_{j,2} + f_1^{(N_j)} v_{j,1} = f_1^{(N_j)} v_{j,1}. \nonumber
\end{align}
This contradicts the definition of $N_j$. When $N'_j < N_j$, we have
\begin{align}
0 = (\mu_j - [N'_j]\{a-N'_j-1\}) f_1^{(N'_j)} v_{j,2} + f_1^{(N'_j)} v_{j,1}. \nonumber
\end{align}
Applying $e_1^{N'_j}$ on both sides, we obtain
\begin{align}
0 = \prod_{l = 1}^{N'_j+1}(\mu_j - [l-1]\{ a-l \})v_{j,2} + Xv_{j,1} \qu \text{for some } X \in \Q(p,q). \nonumber
\end{align}
Since the coefficient of $v_{j,2}$ is nonzero, this contradicts the linear independence of $v_{j,1}$ and $v_{j,2}$. This proves the proposition.
\end{proof}

\begin{theo}\label{IrreducibleModulesForr=1}
For each $a \in \Z$ and $b \in \Z_{\geq 0}$, there exists a unique $(b + 1)$-dimensional irreducible $\Uj_1$-module $L(a;b) \in \Oj$ such that
\begin{align}
&L(a;b) = \bigoplus_{n = 0}^b \Q(p,q) v_n, \nonumber\\
&v_n = f_1^{(n)}v_0, \qu k_1 v_0 = q^a v_0, \qu h'_1 v_0 = [b] \{a-b-1\} v_0. \nonumber
\end{align}
Conversely, each irreducible $\Uj_1$-module in $\Oj$ is isomorphic to $L(a;b)$ for some $a \in \Z$ and $b \in \Z_{\geq 0}$.
\end{theo}

\begin{proof}
It is straightforward to show that $L(a;b)$ is a $(b+1)$-dimensional irreducible $\Uj_1$-module, and so we omit the details. Let $V \in \Oj$ be an irreducible $\Uj_1$-module. By the definition of $\Oj$, there exists an integer $a \in \Z$ such that $V_{a\delta_1} \neq 0$ and $e_1V_{a \delta_1} = 0$. Also, by Propositions \ref{EigenvalueOfh1} and \ref{diagonalizable}, there exist $b \in \Z_{\geq 0}$ and $v \in V_{a \delta_1} \setminus \{ 0 \}$ such that $f_1^{(b)} v \neq 0$, $f_1^{(b+1)} v = 0$, and $h'_1 v = [b]\{a-b-1\}v$. Hence the $\Uj_1$-submodule generated by $v$ is identical to $\bigoplus_{n = 0}^b \Q(p,q) f_1^{(n)} v$, which is isomorphic to $L(a;b)$ by the definitions of $v,a$, and $b$. Since $V$ is irreducible, we have $V = \Uj_1 v \simeq L(a;b)$. This proves the theorem.
\end{proof}

\begin{rem}\normalfont
By this theorem, one can easily see that the category $\Oj$ has an infinite-dimensional object. For instance, consider $\bigoplus_{a \leq 0} L(a;0)$. This is an infinite-dimensional $\Uj_1$-module satisfying the defining conditions of $\Oj$.
\end{rem}

Note that $L(a;b)$ is the irreducible quotient $L'(\lambda;\bfH')$ of the Verma module $V'(\lambda;\bfH')$ with highest weight $(\lambda;\bfH') = (a\delta_1;[b]\{a-b-1\})$. Hence, Theorem \ref{IrreducibleModulesForr=1} gives a necessary and sufficient condition for $L'(\lambda;\bfH')$ to be an object of $\Oj$.

\begin{cor}
Let $a \in \Z$ and $H'_1 \in \Q(p,q)$. Then, the irreducible highest weight module $L'(a\delta_1;H'_1)$ belongs to $\Oj$ if and only if $H'_1 = [b]\{a-b-1\}$ for some $b \in \Z_{\geq 0}$. Moreover, the assignment $(a,b) \mapsto [L(a;b)]$, where $[L(a;b)]$ denotes the isomorphism class of $L(a;b)$, gives a bijection from $\Z \times \Z_{\geq 0}$ to the set of isomorphism classes of irreducible $\Uj_1$-modules in $\Oj$.
\end{cor}

\subsection{Complete reducibility}
Set $z_1 := h'_1 + \frac{[2]pq}{1-q^2}k_1 + \frac{[2]p\inv q\inv}{1-q^{-4}}k_1\inv \in \Uj_1$.

\begin{lem}\label{z1-eigen}
In $\Uj_1$, we have
\begin{align}
z_1 f_1 = q\inv f_1 z_1, \qu z_1 e_1 = q e_1 z_1. \nonumber
\end{align}
\end{lem}

\begin{proof}
By Lemma \ref{h1} and the equalities
\begin{align}
[k_1,f_1]_{-1} = (1-q^2)k_1f_1, \qu [k_1\inv,f_1]_{-1} = (1-q^{-4})k_1\inv f_1, \nonumber
\end{align}
it follows that $z_1 f_1 = q\inv f_1 z_1$. Noting that $z_1$ is invariant under the anti-automorphism $\sigmaj$ defined in Proposition \ref{automorphisms} $(2)$, we obtain the other equality.
\end{proof}

Let $a \in \Z$ and $b \in \Z_{\geq 0}$, and take a highest weight vector $v \in L(a;b)$. Then we have
\begin{align}
z_1f_1^{(n)}v = q^{-n}\left( [b]\{a-b-1\} + \frac{[2]pq^{1+a}}{1-q^2} + \frac{[2]p\inv q^{-1-a}}{1-q^{-4}} \right) f_1^{(n)}v. \nonumber
\end{align}
Denoting by $z_1(a,b,n)$ the coefficient of $f_1^{(n)}v$ on the right-hand side, one has
\begin{align}
z_1(a,b,n) = -\frac{pq^{a-b-n}(q^{b+1}+q^{-b-1})}{q-q\inv} + \frac{p\inv q^{-a+2b-n+1}}{q-q\inv}. \nonumber
\end{align}
Using this, one can verify that the function $\Z^3 \rightarrow \Q(p,q),\ (a,b,n) \mapsto z_1(a,b,n)$, is injective.

\begin{lem}\label{split}
Let $M \in \Oj$, $a, a' \in \Z$, and $b, b' \in \Z_{\geq 0}$. Then, each short exact sequence of the form
\begin{align}\label{ses}
0 \rightarrow L(a;b) \xrightarrow[]{\iota} M \xrightarrow[]{\pi} L(a';b') \rightarrow 0
\end{align}
splits.
\end{lem}

\begin{proof}
Let $v \in L(a',b')$ be a highest weight vector, and take $u \in \pi\inv(v)$. Since $\Uj_1$-module homomorphisms preserve generalized eigenspaces of $z_1$, we may assume that $u$ is a generalized eigenvector of $z_1$ with eigenvalue $z_1(a',b',0)$. Then, $e_1u$ is a generalized eigenvector of $z_1$ with eigenvalue $z_1(a',b',-1)$. Since $\pi(e_1u) = e_1\pi(u) = e_1 v = 0$, it follows that $e_1u \in \iota(L(a',b'))$. However, the eigenvalues of $z_1$ on $L(a,b)$ are $z_1(a,b,n)$, $0 \leq n \leq b$. Therefore, $e_1u = 0$, and hence we obtain a section $v \mapsto u$ of $\pi$. This proves the lemma.
\end{proof}

Now, the complete reducibility of $\Uj$-modules in $\Oj$ follows from a standard argument; see, for example, {\cite[Section 3.5]{HK02}}.

\begin{theo}
Every $\Uj_1$-module in $\Oj$ is completely reducible.
\end{theo}

\begin{cor}\label{CharacterizationOf1highestvec}
Let $M \in \Oj$. Then, $M$ is decomposed into a direct sum of $z_1$-eigenspaces with possible eigenvalues $z_1(a,b,n)$, $a \in \Z, 0 \leq n \leq b$. In particular, if $z_1 m = z_1(a,b,0)m$ for some $m \in M$, then $e_1m = 0$.
\end{cor}

\section{Complete reducibility and the irreducible modules}\label{CompleteReducibility}
Throughout this section, we fix $e \in \{ 1,-1 \}$.
\subsection{Braid group action on $\Uj$}
%
%
\begin{prop}[{\cite[4.5]{KP11}}]
For $i \in \Ij \setminus \{1\}$, there exist unique automorphisms $\tau'_{i,e}$ and $\tau''_{i,-e}$ on $\Uj$ satisfying the following:
\begin{align}
\tau'_{i,e}(e_j) &= \begin{cases}
-k_i^e f_i & \text{if } j = i, \\
e_j \qu & \text{if } |i - j| > 1, \\
[e_j, e_i]_e \qu & \text{if } |i - j| = 1,
\end{cases} \qu \qu
\tau'_{i,e}(f_j) = \begin{cases}
-e_i k_i^{-e} \qu & \text{if } j = i, \\
f_j \qu & \text{if } |i - j| > 1, \\
[f_i, f_j]_{-e} \qu & \text{if } |i - j| = 1,
\end{cases} 
\nonumber\\
\tau''_{i,-e}(e_j) &= \begin{cases}
-f_i k_i^{-e} \qu & \text{if } j = i, \\
e_j \qu & \text{if } |i - j| > 1, \\
[e_i, e_j]_e \qu & \text{if } |i - j| = 1,
\end{cases} \qu \qu
\tau''_{i,-e}(f_j) = \begin{cases}
-k_i^{e} e_i \qu & \text{if } j = i, \\
f_j \qu & \text{if } |i - j| > 1, \\
[f_j, f_i]_{-e} \qu & \text{if } |i - j| = 1,
\end{cases} \nonumber\\
\tau'_{i,e}(k_j) &= \tau''_{i,-e}(k_j) = \begin{cases}
k_i\inv \qu & \text{if } j = i, \\
k_j \qu & \text{if } |i - j| > 1, \\
k_i k_j \qu & \text{if } |i - j| = 1.
\end{cases} \nonumber
\end{align}
Moreover, $\{ \tau'_{i,e} \}_{i \in \Ij \setminus \{1\}}$ and $\{ \tau''_{i,-e} \}_{i \in \Ij \setminus \{1\}}$ satisfy the braid relation of type $A_{r-1}$.
\end{prop}

\begin{proof}
Set $\tau_i := \tau'_{i,e}$ (resp., $\tau''_{i,-e}$), $i \in \Ij \setminus\{1\}$. We need to verify that the relations in \eqref{defrel} hold if we replace $e_i, f_i, k_i$ by $\tau_j(e_i), \tau_j(f_i), \tau_j(k_i)$, respectively. By comparing equations above with Definition \ref{BraidGroupAction}, one immediately finds that the nontrivial assertions are
\begin{align}
\begin{split}
&\tau_2(e_1)^2 \tau_2(f_1) - (q+q\inv)\tau_2(e_1)\tau_2(f_1)\tau_2(e_1) + \tau_2(f_1)\tau_2(e_1)^2 \\
&\hspace{5cm}= -(q+q\inv) \tau_2(e_1) (pq\tau_2(k_1) + p\inv q\inv \tau_2(k_1)\inv), \\
&\tau_2(f_1)^2 \tau_2(e_1) - (q+q\inv)\tau_2(f_1)\tau_2(e_1)\tau_2(f_1) + \tau_2(e_1)\tau_2(f_1)^2 \\
&\hspace{5cm}= -(q+q\inv) (pq\tau_2(k_1) + p\inv q\inv \tau_2(k_1)\inv) \tau_2(f_1).
\end{split} \nonumber
\end{align}
These are checked by direct calculation, or by means of a computer program GAP \cite{GAP16} with a package Quagroup (see {\cite[4.5]{KP11}}). Also, one can verify the braid relation in the same way as for the braid group action on $\U$. This proves the proposition.
\end{proof}

\subsection{Braid group action on $\Uj$-modules}
In this subsection, we define a braid group action on $\Uj$-modules in $\Oj$. Since $\Uj$-contains $(r-1)$ $\mathfrak{sl}_2$-triples $(e_i,k_i,f_i)$, $i \in \Ij \setminus \{1\}$, most parts of the propositions in this subsection follow form the ordinary quantum group theory. Hence, we omit the details.

\begin{defi}\normalfont
Let $M \in \Oj$. For each $i \in \Ij \setminus \{1\}$, we define two linear automorphisms $\tau'_{i,e}$ and $\tau''_{i,e}$ on $M$ by:
\begin{align}
\tau'_{i,e}(m) &= \sum_{\substack{a, b, c \in \Z_{\geq 0} \\ a - b + c = n}} (-q)^b q^{e(-ac+b)} f_i^{(a)} e_i^{(b)} f_i^{(c)} m, \nonumber\\
\tau''_{i,e}(m) &= \sum_{\substack{a, b, c \in \Z_{\geq 0} \\ -a + b - c = n}} (-q)^b q^{e(-ac+b)} e_i^{(a)} f_i^{(b)} e_i^{(c)} m, \nonumber
\end{align}
where $n \in \Z$, and $m \in M$ is such that $k_i m = q^n m$.
\end{defi}

\begin{prop}[see {\cite[Proposition 5.2.2]{L94}}]
Let $M \in \Oj$, $i \in \Ij \setminus \{1\}$, and let $\lambda \in \Lambdaj$ be such that $(\beta_i, \lambda) \geq 0$, $j \in \{ 0, 1, \ldots, (\beta_i, \lambda) \}$; we set $h := (\beta_i,\lambda) - j$.
\begin{enumerate}
\item If $\eta \in M_\lambda$ is such that $e_i \eta = 0$, then $\tau'_{i,e}(f_i^{(j)} \eta) = (-1)^j q^{e(jh + j)} f_i^{(h)} \eta$.
\item If $\xi \in M_{-\lambda}$ is such that $f_i \xi = 0$, then $\tau''_{i,e}(e_i^{(j)} \xi) = (-1)^j q^{(e(jh + j))} e_i^{(h)} \xi$.
\end{enumerate}
\end{prop}

\begin{prop}[see {\cite[Proposition 5.2.3]{L94}}]\label{BraidGroupActionOnWeightVec}
Let $M \in \Oj$, $i \in \Ij \setminus \{1\}$, and $m \in M_\lambda$.
\begin{enumerate}
\item We have $\tau'_{i,e} \tau''_{i,-e} = \id_M = \tau''_{i,-e} \tau'_{i,e}$.
\item We have $\tau''_{i,e}(m) = (-1)^{(\beta_i,\lambda)} q^{e(\beta_i, \lambda)} \tau'_{i,e}(m)$.
\end{enumerate}
\end{prop}

\begin{prop}[see {\cite[Proposition 37.1.2]{L94}}]\label{BraidGroupActionOnModule}
Let $M \in \Oj$ and $i \in \Ij \setminus \{1\}$. Then, for each $m \in M$ and $x \in \Uj_r$, we have
\begin{align}
\tau'_{i,e}(x m) = \tau'_{i,e}(x) \tau'_{i,e}(m), \qu \qu \tau''_{i,e}(x m) = \tau''_{i,e}(x) \tau''_{i,e}(m). \nonumber
\end{align}
\end{prop}

In what follows, we write $\tau_i = \tau''_{i,1}$ for $i \in \Ij \setminus \{1\}$.

\subsection{Classification of the irreducible modules in $\Oj$}
Recall the triangular decomposition $\Uj = \Uj_{<0} \otimes (\Uj_0 \otimes \Ujo) \otimes \Uj_{>0}$ associated with the reflection order $\preceq$ defined in the proof of Lemma \ref{reforder}. Also, recall from \eqref{h_i's} in Section \ref{2.3}, the explicit form of the root vectors $h'_i \in \Uj_0$, $i \in \Ij = \{ 1,\ldots,r \}$. We remark that an irreducible highest weight module is determined by the eigenvalues of $k_i$'s and $h'_i$'s for a highest weight vector. However, $h'_i$'s are sometimes difficult to deal with.

\begin{prop}
Let $V'(\lambda;\bfH')$ be the Verma module with highest weight $(\lambda;\bfH')$, and $v \in V'(\lm;\bfH')$ a highest weight vector. Then, $\bfH'$ is determined by the $\tau_i \cdots \tau_2(h'_1)$-eigenvalue of $v$ for $i \in \Ij$.
\end{prop}

\begin{proof}
For each $i \in \Ij$, set $\ef(i) := e_i \cdots e_2 e_1 f_1 f_2 \cdots f_i$. By equation \eqref{h_i's}, the $h'_i$ is of the form
\begin{align}
h'_i = \sum_{\sigma \in \Ss_{2i}} a_i(\sigma)x_{\sigma(1)} \cdots x_{\sigma(2i)}, \nonumber
\end{align}
where $\Ss_{2i}$ denotes the $2i$-th symmetric group, $a_i(\sigma) \in \Q(q)$, $x_j = e_{i+1-j}$ for $1 \leq j \leq i$, and $x_j = f_{j-i}$ for $i+1 \leq j \leq 2i$. From this, noting that $v$ is a highest weight vector, we deduce that $h'_iv$ is of the form
\begin{align}
h'_iv = \left( \ef(i) + \sum_{1 \leq j < i} g_{j} \ef(j) \right)v, \nonumber
\end{align}
where $g_{j} \in \Q(q)$. Therefore, the $\ef(j)$-eigenvalue of $v$ for $j \leq i$ determine the $h'_i$-eigenvalue $H'_i$.

Also, $\tau_i \cdots \tau_2(h'_1)$ is of the form
\begin{align}
\tau_i \cdots \tau_2(h'_1) = \sum_{\sigma \in \Ss_{2i}} b_i(\sigma)x_{\sigma(1)} \cdots x_{\sigma(2i)}, \nonumber
\end{align}
where $b_i(\sigma) \in \Q(q)$. In the same way as above, the $\tau_i \cdots \tau_2(h'_1)$-eigenvalue of $v$ is determined by the $\ef(j)$-eigenvalue of $v$ for $j \leq i$. Conversely, the $\tau_j \cdots \tau_2(h'_1)$-eigenvalue of $v$ for $j \leq i$ altogether determine the $\ef(j)$-eigenvalue of $v$ for $j \leq i$, which, in turn, determine the $h'_i$-eigenvalue $H'_i$ of $v$. This proves the proposition.
\end{proof}

This proposition enables us to replace $h'_iv$ with $\tau_i \cdots \tau_2(h'_1)v$ for $i \in \Ij$. Then, we define $h_i$, $i \in \Ij$, by $h_1 := [e_1,f_1]_1$ and $h_i := \tau_i \cdots \tau_2(h_1)$. Also, we set $V(\lm;\bfH) := V'(\lm;\bfH')$ and $L(\lm;\bfH) := L'(\lm;\bfH')$, where $\bfH = (H_1,\ldots,H_r)$ is uniquely determined by the equations $h_i v = H_i v$, $i \in \Ij$, where $v \in V'(\lm;\bfH')$ is a highest weight vector.

Let $L \in \Oj$ be an irreducible $\Uj$-module. By condition $(\mathrm{M}3)$, there exists $\lambda \in \Lambdaj$ such that $L_\lambda \neq 0$ and $L_\mu = 0$ for all $\mu > \lambda$. By the case $r=1$, $h_1$ acts on $L_\lambda$ semisimply.

\begin{lem}\label{[h_1,h_2]}
We have
\begin{align}
[h_1,h_2]_0 = [h_1,(q-q\inv)(f_2[e_2,h_1]_1 - p\inv q^2 f_2e_2k_1\inv)]_0 \in \Uj(e_2, e_2h_1, e_2h_1^2), \nonumber
\end{align}
where $\Uj(e_2, e_2h_1, e_2h_1^2)$ denotes the left ideal of $\Uj$ generated by $e_2, e_2h_1, e_2h_1^2$.
\end{lem}

\begin{proof}
By direct calculation (or by using GAP).
\end{proof}

This lemma implies that $[h_1,h_2]_0 L_\lambda = 0$; namely, the actions of $h_1$ and $h_2$ commute with each other on $L_\lambda$.

\begin{lem}\label{InvarianceOfh_i}
Let $i,j \in \Ij$. If $j \neq i,i+1$, then we have $\tau_j(h_i) = h_i$.
\end{lem}

\begin{proof}
The assertion in the case $j > i+1$ follows from the definitions of $\tau_j$ and $h_i$. When $j < i$, by the braid relation for the $\tau_j$'s, we see that
\begin{align}
\begin{split}
\tau_j(h_i) &= \tau_j (\tau_i \tau_{i-1} \cdots \tau_2)(h_1) \\
&= \tau_i \cdots \tau_{j+2} \tau_j \tau_{j+1} \tau_j \cdots \tau_2(h_1) \\
&= \tau_i \cdots \tau_{j+2} \tau_{j+1} \tau_j \tau_{j+1} \tau_{j-1} \cdots \tau_2(h_1) \\
&= \tau_i \cdots \tau_{j+2} \tau_{j+1} \tau_j \cdots \tau_2 \tau_{j+1}(h_1) \\
&= \tau_i \cdots \tau_2(h_1) = h_i.
\end{split} \nonumber
\end{align}
This proves the lemma.
\end{proof}

\begin{prop}
Let $L \in \Oj$ be an irreducible module. Take $\lambda \in \Lambdaj$ such that $L_\lambda \neq 0$ and $L_\mu = 0$ for all $\mu > \lambda$. Then, the actions of $h_1, \ldots, h_r$ commute with each other on $L_\lambda$.
\end{prop}

\begin{proof}
Let $i,j \in \Ij$ be such that $j < i$. By Lemma \ref{InvarianceOfh_i},
\begin{align}
[h_j,h_i]_0 = \tau_j \cdots \tau_2([h_1,h_i]_0) = \tau_j \cdots \tau_2 \tau_i \cdots \tau_3([h_1,h_2]_0). \nonumber
\end{align}
Also, by Lemma \ref{[h_1,h_2]},
\begin{align}
\tau_j \cdots \tau_2 \tau_i \cdots \tau_3([h_1,h_2]_0) \in \Uj(\tau_{j,i}(e_2), \tau_{j,i}(e_2)h_j, \tau_{j,i}(e_2)h_j^2), \nonumber
\end{align}
where $\tau_{j,i}$ denotes $\tau_j \cdots \tau_2 \tau_i \cdots \tau_3$. Since $\tau_{j,i}(e_2) \in \Uj (\Uj_{>0})_+$, the vectors $\tau_{j,i}(e_2)h_j^{l}$, $l = 0,1,2$, act on $L_\lambda$ by $0$. This proves the proposition.
\end{proof}

As a corollary of this proposition, we can take a simultaneous eigenvector $v \in L_\lambda$ for $h_1, \ldots, h_r$. Let $H_i \in \Q(p,q)$ denote the eigenvalue of $h_i$. Then the submodule generated by $v$ is a highest weight module. Since $L$ is irreducible, we conclude that $L$ is isomorphic to $L(\lm;\bfH)$ for some $\lm \in \Lambdaj$, $\bfH \in \Q(p,q)^r$.

Next, we investigate a necessary condition for $L(\lm;\bfH)$ being a nonzero object of $\Oj$.

\begin{lem}\label{decomp of f_2^n v}
Let $M \in \Oj$, $v \in M$ be such that $e_1v = 0$ and $h_1v = [b]\{ a-b-1 \}$ for some $a \in \Z$, $b \in \Z_{\geq 0}$. Let $n \in \Z_{>0}$. Then, there exist unique $v_0,v_1,\ldots,v_n \in M$ satisfying the following:
\begin{enumerate}
\item $f_2^{(n)}v = \sum_{k=0}^n v_k$.
\item $e_1 v_k = 0$, $h_1 v_k = [b+k]\{ a+n-(b+k)-1 \}v_k$.
\end{enumerate}
\end{lem}

\begin{proof}
As the proof of this lemma needs some lengthy calculation, we put it in the end of this section.
\end{proof}

\begin{lem}\label{proof of (*)}
Let $M \in \Oj$ be a $\Uj_2$-module, $v \in M$ a highest weight vector with $k_iv = q^{a_i}v$, $h_iv = H_iv$ for some $a_1 \in \Z$, $a_2 \in \Z_{\geq 0}$, $H_1,H_2 \in \Q(p,q)$. If $H_1 = [b_1]\{a_1-b_1-1\}$ for some $b_1 \in \Z_{\geq 0}$, then $H_2 = [b_1+b_2]\{a_1+a_2-(b_1+b_2)-1\}$ for some $0 \leq b_2 \leq a_2$.
\end{lem}

\begin{proof}
By the representation theory for $U_q(\mathfrak{sl}_2)$, we have $f_2^{a_2+1} v = 0$, and $\tau_2\inv(v) = f_2^{(a_2)}v$. Set $u := \tau_2\inv(v)$. We claim that $e_1 u = 0$ and $h_1 u = H_2 u$. The former is true as we have $e_1f_2 = f_2e_1$. The latter follows from an easy calculation
$$
h_1 u = \tau_2\inv(\tau_2(h_1) v) = \tau_2\inv(h_2 v) = H_2 u.
$$
Then, by the case $r=1$, $H_2$ must be of the form $H_2 = [b]\{ a_1+a_2-b-1 \}$ for some $0 \leq b \leq a_1+a_2$. Here, by Lemma \ref{decomp of f_2^n v}, it must hold that $b = b_1 + b_2$ for some $0 \leq b_2 \leq a_2$. This proves the assertion.
\end{proof}

\begin{theo}\label{Classification}
Each irreducible module in $\Oj$ is isomorphic to $L(\lambda;\bfH)$ for some $\lambda \in \Lambda$ and $\bfH = (H_1,\ldots,H_r) \in \Q(p,q)^r$ satisfying the following:
\begin{enumerate}
\item $a_i := (\beta_i,\lambda) \geq 0$ for each $i \in \Ij \setminus \{1\}$.
\item For each $i \in \Ij$, there exists $b_i \in \Z_{\geq 0}$ such that $0 \leq b_i \leq a_i$ for $i \in \Ij \setminus \{1\}$ and $H_i = [b_1+\cdots+b_i]\{ a_1+\cdots+a_i-(b_1+\cdots+b_i)-1 \}$ for $i \in \Ij$.
\end{enumerate}
\end{theo}

\begin{proof}
We have shown that each irreducible module in $\Oj$ is isomorphic to $L(\lambda;\bfH)$ for some $\lambda \in \Lambdaj$ and $\bfH \in \Q(p,q)^r$. It is easy to verify that $L(\lambda;\bfH)$ belongs to $\Oj$ if and only if $f_i^Nv = 0$, $i \in \Ij$, for a sufficiently large $N$, where $v \in L(\lambda;\bfH)$ is a highest weight vector. By the case $r=1$, the equality $f_1^{N}v = 0$ is equivalent to the existence of $b_1 \in \Z_{\geq 0}$ satisfying the equality $H_1 = [b_1]\{ a_1-b_1-1 \}$. Also, by the representation theory of $U_q(\mathfrak{sl}_2)$, the condition $f_i^Nv = 0$, $i \geq 2$, implies $a_i \geq 0$.

It remains to determine the possible values of $H_2,\ldots,H_r$. By Lemma \ref{proof of (*)}, there exists $b_2 \in \Z_{\geq 0}$ such that $b_2 \leq a_2$ and $H_2 = [b_1+b_2] \{ a_1+a_2 - (b_1+b_2) - 1 \}$. Let $i \geq 3$, and assume that for all $j < i$, $H_j = [b_1+\cdots+b_j]\{a_1+\cdots+a_j-(b_1+\cdots+b_j)-1\}$ for some $0 \leq b_j \leq a_j$. Set $T_i := (\tau_{i-1}\tau_i) \cdots (\tau_3\tau_4)(\tau_2\tau_3)$, and consider the subalgebra $T_i(\Uj_2) \subset \Uj$. We have $T_i(k_1) = k_1 \cdots k_{i-1}$, $T_i(k_2) = k_i$, $T_i(h_1) = h_{i-1}$, and $T_i(h_2) = h_i$. If we regard $L$ as a $\Uj_2$-module via the algebra homomorphism $T_i : \Uj_2 \rightarrow \Uj_r$, the $v$ is a highest weight vector such that
$$
k_1v = q^{a_1+\cdots+a_{i-1}}v,\ k_2v = q^{a_i}v,\ h_1v = H_{i-1}v,\ h_2v = H_iv.
$$
By lemma \ref{proof of (*)}, $H_i$ must be of the form $[b_1+\cdots + b_{i-1} + b_i]\{ a_1+\cdots+a_{i-1}+a_i-(b_1+\cdots+b_{i-1}+b_i)-1 \}$ for some $0 \leq b_i \leq a_i$. This proves the theorem.
\end{proof}

From now on, we write $L(\bfa;\bfb)$ instead of $L(\lambda;\bfH)$, where $\bfa = (a_1,\ldots,a_r)$ and $\bfb = (b_1,\ldots,b_r)$ are such that $a_i = (\beta_i,\lambda)$, $H_i = [b_1+\cdots+b_i]\{ (a_1+\cdots+a_i)-(b_1+\cdots+b_i)-1 \}$. This causes no confusion since $\bfa \in \Z^r$, while $\lm \in \Lmj$. We call $L(\bfa;\bfb)$ the irreducible highest weight $\Uj$-module with highest weight $(\bfa;\bfb)$.

\begin{cor}
Let $\lambda \in \Lambdaj$ and $\bfH \in \Q(p,q)^r$. Then, $L(\lambda;\bfH) $ belongs to $\Oj$ if and only if $L(\lambda;\bfH) = L(\bfa;\bfb)$ for some $(\bfa,\bfb) \in \Z^r \times \Z_{\geq 0}^r$ such that $a_i \geq b_i$, $i \in \Ij \setminus\{1\}$. Moreover, the assignment $(\bfa,\bfb) \mapsto [L(\bfa;\bfb)]$, where $[L(\bfa;\bfb)]$ denotes the isomorphism class of $L(\bfa;\bfb)$, gives a bijection from $\{ (\bfa,\bfb) \in \Z^r \times \Z_{\geq 0}^r \mid a_i \geq b_i,\ i \in \Ij \setminus\{1\}, \AND L(\bfa;\bfb) \neq 0 \}$ to the set of isomorphism classes of irreducible $\Uj$-modules in $\Oj$.
\end{cor}

\begin{rem}\normalfont
We will prove that $L(\bfa;\bfb) \neq 0$ for all such $(\bfa;\bfb)$ by using the crystal basis theory for $(\bfU,\Uj)$ in subsection \ref{irr decomp of V^d}.
\end{rem}

\subsection{Complete reducibility}
In this subsection only, we set $A := \Uj$, and write $B$ for $\Uj$ with $p$ replaced by $p\inv q$. In order to avoid confusion, we denote by $e_i^A, f_i^A, k_i^{A,\pm1}$ the generators of $A$, and by $e_i^B, f_i^B, k_i^{B,\pm1}$ those of $B$. Consider the anti-algebra homomorphism $S:A \rightarrow B$ over $\Q(p,q)$ defined by:
\begin{align}
S(e_i^A) = -e_i^Bk_i^B, \qu S(f_i^A) = -k_i^{B,-1} f_i^B, \qu S(k_i^A) = k_i^{B,-1}. \nonumber
\end{align}
It is straightforwardly checked that $S$ is indeed an anti-algebra homomorphism. In addition, $S$ is invertible:
\begin{align}
S\inv(e_i^B) = -k_i^Ae_i^A, \qu S\inv(f_i^B) = -f_i^Ak_i^{A,-1}, \qu S\inv(k_i^B) = k_i^{A,-1}. \nonumber
\end{align}

For an $A$-module $M \in \Oj$, define a $B$-module $S_*(M) := M^\vee$ by:
\begin{align}
(x \cdot g)(m) = g(S\inv(x) \cdot m) \qu \text{for } x \in B,\ g \in M^\vee,\ m \in M, \nonumber
\end{align}
where $M^\vee$ denotes the restricted dual of $M$, i.e., $M^\vee = \bigoplus_{\lambda \in \Lambdaj} \Hom_{\Q(p,q)}(M_\lambda,\Q(p,q))$. Similarly, we associate an $A$-module $S^*(N)$ with each $B$-module $N$.

\begin{lem}
Let $L \in \Oj$ be the irreducible highest weight $A$-module with highest weight $(\lambda;\bfH)$. Then, $S_*(L)$ is the irreducible lowest weight $B$-module with lowest weight $(-\lambda;\bfH')$ for some $\bfH' \in \Q(p,q)^r$.
\end{lem}

\begin{proof}
Let $v \in L$ be a highest weight vector, and let $g \in S_*(L)$ be a unique element satisfying $g(v) = 1$ and $g(u) = 0$ for all $u \in L_{\mu}$, $\mu < \lambda$. Then, we have
\begin{align}
&(k_i^Bg)(v) = g(k_i^{A,-1} v) = q^{-(\beta_i,\lambda)}g(v), \nonumber\\
&(h_i^B g)(v) = g(S\inv(h_i^B) v), \nonumber
\end{align}
where $h_i^B \in B$ is the elements corresponding to $h_i \in \Uj$. Since $S\inv(h_i^B) v \in L_\lambda = \Q(p,q)v$, we have $S\inv(h_i^B)v = H'_i v$ for some $H'_i \in \Q(p,q)$, and hence $h_i^Bg = H'_ig$. Therefore, $Bg$ is a lowest weight module with lowest weight $(-\lambda;H'_1,\ldots,H'_r)$.

Now, it remains to show that $S_*(L)$ is irreducible. Suppose that $N \subset S_*(L)$ is a $B$-submodule. Then $S^*(N)$ is a quotient of $S^*(S_*(L)) \simeq L$. Since $L$ is irreducible, $S^*(N)$ is identical either to $0$ or to $L$, and hence $N$ is identical either to $0$ or to $S_*(L)$. Thus, $S_*(L)$ is irreducible. This proves the lemma.
\end{proof}

\begin{lem}
Let $M$ be an $A$-module in $\Oj$. Suppose that $M$ contains an irreducible submodule $L \simeq L(\lambda;\bfH)$ for some $\lambda \in \Lambdaj$ and $\bfH \in \Q(p,q)^r$. Then, $M \simeq L \oplus (M/L)$.
\end{lem}

\begin{proof}
It suffices to show that the short exact sequence
\begin{align}
0 \xrightarrow[]{} L \xrightarrow[]{\iota} M \xrightarrow[]{\pi} M/L \xrightarrow[]{} 0 \nonumber
\end{align}
splits. By the previous lemma, $S_*(M)$ has an irreducible submodule $S_*(L)$. Applying $S^*$ to the inclusion $S_*(L) \hookrightarrow S_*(M)$, we obtain a surjection $M \twoheadrightarrow L$ of $A$-modules. Since the composite map $L \xrightarrow[]{\iota} M \twoheadrightarrow L$ is nonzero, it follows from Schur's lemma that this composite map is an isomorphism of $A$-modules. By composing the inverse of this isomorphism with the surjection $M \twoheadrightarrow L$, we obtain a retraction of $\iota$. This proves the lemma.
\end{proof}

Now, the complete reducibility of the $\Uj$-modules in $\Oj$, and the consequences below, follow from a standard argument; see, for example, {\cite[Section 3.5]{HK02}}.

\begin{theo}\label{4.4.3}
Every $\Uj$-module in $\Oj$ is completely reducible.
\end{theo}

\begin{cor}
Every highest weight module in $\Oj$ is irreducible.
\end{cor}


\begin{theo}
Let $M \in \Oj$. Irreducible decomposition of $M$ is unique in the following sense. If we have two irreducible decompositions $M = \bigoplus_{j \in J} L_j = \bigoplus_{k \in K} L^k$ for some index sets $J$ and $K$, then there exists a bijection $\phi:J \rightarrow K$ such that $L_j \simeq L^{\phi(j)}$ for all $j \in J$. Moreover, for each $j \in J$, the number of $j' \in J$ such that $L_{j'} \simeq L_j$ is finite.
\end{theo}

%

\subsection{Proof of Lemma \ref{decomp of f_2^n v}}
Throughout this subsection, we fix a $\Uj_2$-module $M \in \Oj$. Recall from the case $r = 1$ that $M$ is decomposed as:
\begin{align}
M &= \bigoplus_{\substack{a \in \Z \\ b ,n \in \Z_{\geq 0}}} M_{a,b,n}, \nonumber\\
M_{a,b,0} &= \{ u \in M \mid e_1 u = 0, \ k_1 u = q^{a} u, \ h_1 u = [b]\{a-b-1\} u \}, \nonumber\\
M_{a,b,n} &= f_1^{(n)}(M_{a,b,0}). \nonumber
\end{align}

Recall that $h_1 = [e_1,f_1]_1$ and $h_2 = \tau_2(h_1)$. Set $f'_2 := q^{-2} \bigl[e_1,[f_1,f_2]_1 \bigr]_1 - p\inv q\inv f_2 k_1\inv$. For each $a \in \Z$ and $b, n \in \Z_{\geq 0}$, we define $f'_{2,i}(a,b,n) \in \Uj$, $i = 1, 2, 3$, by
\begin{align}
f'_{2,1}(a,b,n) &:= q^{b-n-1} \overline{f'_2} + (pq^{a-b} - p\inv q^{-a+b})f_2 - q^{-b+n+1}f'_2, \nonumber\\
f'_{2,2}(a,b,n) &:= p q^{a-b-n-2} \overline{f'_2} - (q^{b+1} + q^{-b-1}) f_2 + p\inv q^{-a+b+n+2} f'_2, \nonumber\\
f'_{2,3}(a,b,n) &:= q^{-n-2} \overline{f'_2} + (p q^{a-2b-1} - p\inv q^{-a+2b+1}) f_2 - q^{n+2} f'_2. \nonumber
\end{align}
Also, we define three linear maps $f'_{2,i}$, $i = 1,2,3$, by
\begin{align}
f'_{2,i}(m) := f'_{2,i}(a,b,n)m \qu \text{for } m \in M_{a,b,n}. \nonumber
\end{align}

Set $h''_1 := h_1 + \frac{p\inv q k_1\inv}{q-q\inv}$.
\begin{lem}\label{A.1.1}
We have the following:
\begin{enumerate}
\item $[h''_1,f_2]_1 = q^2 f'_2$.
\item $[h''_1,\overline{f'_2}]_1 = q^2 f_2$.
\item $[h''_1,f'_2]_{-1} = -p \left(q^{-3}\overline{f'_2} - [2] f'_2 - f_2 \left(q\inv(q-q\inv)\overline{h''_1} + [2] p\inv q\inv k_1\inv \right) \right) k_1$.
\end{enumerate}
\end{lem}

\begin{proof}
This is easy and straightforward.
\end{proof}

\begin{prop}\label{EigenvectorsOfh1}
Let $a \in \Z$, $b, n \in \Z_{\geq 0}$, and $m \in M_{a,b,n}$. Then, we have
\begin{align}
f'_{2,1}(m) \in M_{a+1,b+1,n}, \qu f'_{2,2}(m) \in M_{a+1,b,n}, \qu f'_{2,3}(m) \in M_{a-2,b-1,n-1}. \nonumber
\end{align}
\end{prop}

\begin{proof}
Since $h_1$ and $k_1$ act on $m$ as scalar multiplication, so does $h''_1$; explicitly, we have $h''_1 m = h''_1(a,b,n) m$, where
\begin{align}
h''_1(a,b,n) := [n+1][b-n]\{a-b-n-1\} - q[n][b-n+1]\{a-b-n\} + \frac{p\inv q^{-a+3n+1}}{q-q\inv}. \nonumber
\end{align}
By Lemma \ref{A.1.1}, we have
\begin{align}
h''_1 \overline{f'_2} m &= q h''_1(a,b,n) \overline{f'_2} m + q^2 f_2 m, \nonumber\\
h''_1 f_2 m &= q h''_1(a,b,n) f_2 m + q^2 f'_2 m, \nonumber\\
h''_1 f'_2 m &= q\inv h''_1(a,b,n) f'_2 m - p \left(q^{-3}\overline{f'_2} - [2] f'_2 - f_2 \left(q\inv(q-q\inv)\overline{h''_1(a,b,n)} + [2] p\inv q\inv q^{-a+3n} \right) \right) q^{a-3n} \nonumber\\
&= -pq^{a-3n-3}\overline{f'_2} m + pq^{a-3n} \left( q\inv(q-q\inv)\overline{h''_1(a,b,n)} + [2]p\inv q^{-a+3n-1} \right) f_2 m \nonumber\\
&\ + (q\inv h''_1(a,b,n) + pq^{a-3n}[2])f'_2 m. \nonumber
\end{align}
Therefore, $h''_1$ defines a linear endomorphism on the vector space spanned by $\{ \overline{f'_2} m, f_2 m, f'_2 m \}$ whose representation matrix is
\begin{align}\label{MatrixOfh1}
\begin{pmatrix}
q h''_1(a,b,n) & 0 & -pq^{a-3n-3} \\
q^2 & q h''_1(a,b,n) & pq^{a-3n-1}(q-q^{-1})\overline{h''_1(a,b,n)} + q^{-1}[2] \\
0 & q^2 & q\inv h''_1(a,b,n) + pq^{a-3n}[2] 
\end{pmatrix}. 
\end{align}
Hence, in order to prove Proposition \ref{EigenvectorsOfh1}, it suffices to show that the following three vectors
\begin{align}
\begin{pmatrix}
q^{b-n-1} \\
pq^{a-b} - p\inv q^{-a+b} \\
-q^{-b+n+1} \\
\end{pmatrix}, \qqu 
\begin{pmatrix}
p q^{a-b-n-2} \\
-(q^{b+1} + q^{-b-1}) \\
p\inv q^{-a+b+n+2} \\
\end{pmatrix}, \qqu 
\begin{pmatrix}
q^{-n-2} \\
pq^{a-2b-1} - p\inv q^{-a+2b+1} \\
-q^{n+2} \\
\end{pmatrix} \nonumber
\end{align}
are eigenvectors of the matrix \eqref{MatrixOfh1} with eigenvalues $h''_1(a+1,b+1,n)$, $h''_1(a+1,b,n)$, and $h''_1(a-2,b-1,n-1)$, respectively. This can be checked by using a computer, or possibly by direct calculation.
\end{proof}

We normalize $f'_{2,i}$ as follows:
\begin{align}
f_{2,1}(a,b,n) &:= \frac{1}{(q^{b+1} - q^{-b-1}) \{a-2b-1\}}f'_{2,1}(a,b,n), \nonumber\\
f_{2,2}(a,b,n) &:= -\frac{1}{\{a-b\} \{a-2b-1\}}f'_{2,2}(a,b,n), \nonumber\\
f_{2,3}(a,b,n) &:= -\frac{1}{(q^{b+1} - q^{-b-1}) \{a-b\}}f'_{2,3}(a,b,n), \nonumber
\end{align}
and define linear maps $f_{2,i}$, $i = 1,2,3$, by $f_{2,i}(m) = f_{2,i}(a,b,n)m$ for $m \in M_{a,b,n}$. Then, for each $m \in M_{a,b,n}$, we have $f_2 m = (f_{2,1} + f_{2,2} + f_{2,3})m$. Thanks to this equality and Proposition \ref{EigenvectorsOfh1}, in order to compute $f_{2,i}(m)$, it is enough to decompose $f_2m$ into three $h_1$-eigenvectors with distinct eigenvalues. The computation becomes easier when $n=0$ since in this case, $f_{2,3}(m) = 0$. Also, it follows that $f_2m \in M_{a+1,b+1,0} \oplus M_{a+1,b,0}$ for $m \in M_{a,b,0}$. Repeating this, we have
\begin{align}\label{f_2^l}
f_2^{(n)}m \in \bigoplus_{k = 0}^n M_{a+n,b+k,0} \qu \text{for } n \in \Z_{\geq 0},\ m \in M_{a,b,0}.
\end{align}
This completes the proof of Lemma \ref{decomp of f_2^n v}.

\part{Crystal basis theory for $\Uj$}\label{part 2}
\section{Combinatorics}\label{Combinatorics}
\subsection{Partitions and Young tableaux}
A partition of $n \in \N$ of length $l \in \N$ is a nonincreasing sequence $\lm = (\lm_1,\ldots,\lm_l)$ of nonnegative integers satisfying $\sum_{i=1}^l \lm_i = n$; we call $n$ the size of $\lm$. Let $|\lm|$ and $\ell(\lm)$ denote the size and the length of $\lm$, respectively. We denote by $\Par_l(n)$ the set of partitions of $n$ of length $l$.

We often identify a partition with a Young diagram in a usual way. Let $(L,\preceq)$ be a totally ordered set. A semistandard tableau of shape $\lm \in \Par_l(n)$ in letters $L$ is a filling of the Young diagram $\lm$ with elements of $L$, which weakly increases (with respect to the total order $\preceq$) from left to right along the rows, and strictly increases from the top to the bottom along the columns. A semistandard tableau of shape $\lm$ is said to be standard if $L = \{ 1,\ldots,|\lm| \}$, and the filling strictly increases along the rows.

A bipartition of $n \in \N$ of length $(l;m) \in \N^2$ is an ordered pair $\bflm := (\bflm^-;\bflm^+)$ of partitions such that $\ell(\bflm^-) = l$, $\ell(\bflm^+) = m$, and $|\bflm^-|+|\bflm^+| = n$. Set $|\bflm| := |\bflm^-| + |\bflm^+|$. We denote by $P_{(l;m)}(n)$ the set of bipartitions of $n$ of length $(l;m)$. For totally ordered sets $(L^-;\preceq^-)$ and $(L^+;\preceq^+)$, a semistandard tableau of shape $\bflm \in P_{(l;m)}(n)$ in letters $(L^-;L^+)$ is an ordered pair $(T^-;T^+)$, where $T^\pm$ is a semistandard tableau of shape $\bflm^\pm$ in letters $L^\pm$. A semistandard tableau of shape $\bflm$ is said to be standard if $L^-, L^+ \subset \{ 1,\ldots,|\bflm| \}$, $L^- \sqcup L^+ = \{ 1,\ldots,|\bflm| \}$, and the fillings strictly increase along the rows.

Set
\begin{itemize}
\item $P(n) = P_r(n) := \Par_{2r+1}(n)$: the set of partitions of $n$ of length $2r+1$.
\item $P := \bigsqcup_{n \in \N} P(n)$: the set of partitions of length $2r+1$.
\item $\Pj(n) = \Pj_r(n) := P_{(r+1;r)}(n)$: the set of bipartitions of $n$ of length $(r+1;r)$.
\item $\Pj := \bigsqcup_{n \in \N} \Pj(n)$: the set of bipartitions of length $(r+1;r)$.
\item $\SST(\lm)$: the set of semistandard tableaux of shape $\lm \in P$ in letters $I := \{ -r < \cdots -1 < 0 < 1 < \cdots < r \}$.
\item $\ST(\lm)$: the set of standard tableaux of shape $\lm \in P$.
\item $\SST(\bflm)$: the set of semistandard tableaux of shape $\bflm \in \Pj$ in letters $(I^-;I^+)$, where $I^- := \{ 0 \prec^- -1 \prec^- \cdots \prec^- -r \}$, and $I^+ := \{ 1 \prec^+ \cdots \prec^+ r \}$.
\item $\ST(\bflm)$: the set of standard tableaux of shape $\bflm \in \Pj$.
\end{itemize}

For $\bflm \in \Pj$, we refer the $i$-th row of $\bflm^-$ to as the $-(i-1)$-th row of $\bflm$, and the $j$-th row of $\bflm^+$ to as the $j$-th row of $\bflm$. Also, for $i \in I$, set $\bflm_i$ to be the length of the $i$-th row of $\bflm$, i.e.,
$$
\bflm_i := \begin{cases}
\bflm^-_{-i+1} \qu & \IF i \leq 0, \\
\bflm^+_i \qu & \IF i > 0.
\end{cases}
$$

\subsection{Parametrization of the irreducible $\Uj$-modules}
Let $L(\bfa;\bfb)$ be the irreducible highest weight $\Uj$-module corresponding to $\bfa = (a_1,\ldots,a_r)$, $\bfb = (b_1,\ldots,b_r)$, with $a_1 \in \Z$, $a_i, b_j \in \Z_{\geq 0}$, and $b_i \leq a_i$ for $i = 2, \ldots, r$ and $j = 1,\ldots, r$. Set
\begin{align}
&\bflm^- := \left( \sum_{i=1}^r b_i, \sum_{i=2}^r b_i, \ldots, b_r, 0 \right) + a_1^-\rho_{r+1}, \nonumber\\
&\bflm^+ := \left( \sum_{i=2}^r (a_i- b_i), \sum_{i=3}^r (a_i-b_i), \ldots, a_r-b_r, 0 \right) - a_1^+\rho_{r}, \nonumber
\end{align}
where $a_1^- := \max\{ a_1-\left( 2\sum_{i=1}^r b_i - \sum_{i=2}^r a_i \right), 0 \}$, $a_1^+ := \min\{ a_1-\left( 2\sum_{i=1}^r b_i - \sum_{i=2}^r a_i \right), 0 \}$, $\rho_n := (1,1,\ldots,1)$ ($n$ components), and the addition is defined componentwise. The assignment $(\bfa;\bfb) \mapsto \bflm$ gives a bijection from $\{ (\bfa;\bfb) \mid a_1 \in \Z,\ b_1 \in \Z_{\geq 0},\ 0 \leq b_i \leq a_i,\ i \geq 2 \}$ to the set of bipartitions of length $(r+1;r)$ containing at least one $0$; the inverse map $\pi$ is given by
\begin{align}\label{pi}
a_1 = 2\bflm_0-\bflm_{-1}-\bflm_1,\ a_i = \bflm_{-(i-1)}-\bflm_{-i}+\bflm_{i-1}-\bflm_i,\ b_i = \bflm_{-(i-1)}-\bflm_{-i}.
\end{align}
We write $L(\bflm) := L(\pi(\bflm))$. If we define $\pi(\bflm)$ by equation \eqref{pi} for a bipartition $\bflm$ of length $(r+1;r)$ (not necessarily containing $0$), then we have $\pi(\bflm) = \pi(\bfmu)$ if and only if $\bfmu = \bflm + (n\rho_{r+1};n\rho_r)$ for some $n \in \Z$. We denote this relation by $\bflm \sim_\pi \bfmu$, and fix a complete set $\Pj_\pi$ of representatives for $\Pj/\sim_\pi$. Set $L(\bfmu) := L(\pi(\bfmu))$. We call $L(\bfmu)$ the irreducible highest weight $\Uj$-module with highest weight $\bfmu$. Note that there exists a natural bijection $\SST(\bflm) \rightarrow \SST(\bfmu)$ if $\bflm \sim_\pi \bfmu$.

\begin{ex}\normalfont
Let $r = 3$, $\bfa = (2,2,3)$, $\bfb = (2,0,1)$. Then, the corresponding bipartition is $(4,2,2,1;4,2,0)$, and the associated Young diagram is
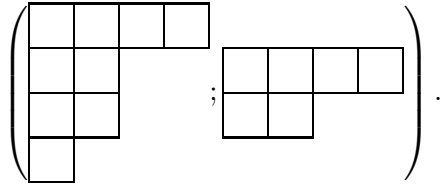
\begin{align}
\left(
\ytableausetup{centertableaux} \begin{ytableau} \empty & & & \\ & \\ & \\ \\ \end{ytableau};
\ytableausetup{centertableaux} \begin{ytableau}  \empty & & &  \\ & \\ \end{ytableau}
\right). \nonumber
\end{align}
\end{ex}

%
%

\subsection{Readings of tableaux}\label{jcrystal}
Let $\lm \in P$ and $T \in \SST(\lm)$. We denote by $\ME(T)$ the Middle-Eastern reading of $T$. Also, for $\bflm \in \Pj$ and $\bfT \in \SST(\bflm)$, set $\Read(\bfT) := (\EM(\bfT^-),\ME(\bfT^+))$, where $\EM(\bfT^-)$ is obtained by reversing $\ME(\bfT^-)$. For example,
\begin{align}
R\left(
\ytableausetup{centertableaux} \begin{ytableau} 0 & 0 & -1 & -4 \\ -1 & -2 \\ -3 & -3 \\ -4 \\ \end{ytableau};
\ytableausetup{centertableaux} \begin{ytableau} 1 & 2 & 2 & 4  \\ 3 & 4 \\ \end{ytableau}
\right) = (-4,-3,-3,-1,-2,0,0,-1,-4,4,2,2,1,4,3). \nonumber
\end{align}

Recall that $I = \{ -r,\ldots,r \}$. For $\bfs = (s_1,\ldots,s_n) \in I^{n}$ and $i \in \{ -r+1,\ldots,r \}$, define $\bfs_{i-\hf}$ to be the sequence of $(i-1)$'s and $i$'s obtained by the following way. First, ignore all $s_j$ such that $s_j \neq i-1,i$. Next, delete the adjacent pairs $(i-1,i)$. The resulting sequence is $\bfs_{i-\hf}$. Note that $\bfs_{i-\hf}$ is of the form $(i,\ldots,i,i-1,\ldots,i-1)$.

Also,  for $i \in \Ij \setminus \{1\}$, we define $\bfs_i$ as follows. First, consider the concatenated sequence $(\bfs_{-(i-\hf)}^{\rev},\bfs_{i-\hf})$, where $\bfs_{-(i-\hf)}^{\rev}$ is obtained by reversing $\bfs_{-(i-\hf)}$. Next, delete the adjacent pairs $(-(i-1),i)$. The resulting sequence is $\bfs_i$. Note that $\bfs_i$ is either of the form $(-i,\ldots,-i,i,\ldots,i,i-1,\ldots,i-1)$ or of the form $(-i,\ldots,-i,-(i-1),\ldots,-(i-1),i-1,\ldots,i-1)$.

Finally, we define $\bfs_1$ to be $\bfs_{0-\hf}^{\rev}$. For example, if $\bfs = (-4,-3,-3,-1,-2,0,0,-1,-4,4,2,2,1,4,3)$, then we have
\begin{align}
\begin{split}
&\bfs_{-(4-\hf)} = \bfs_{-3-\hf} = (-3,-4), \\
&\bfs_{4-\hf} = (4,4,3), \\
&\bfs_4 = (-4,4,3), \\
&\bfs_1 = \bfs_{0-\hf}^{\rev} = (-1,0).
\end{split} \nonumber
\end{align}

\section{Quasi-$\jmath$-crystal bases}\label{WeakCrystal}
\subsection{Crystal bases}
In this subsection, we briefly recall some basic properties of crystal bases for $\U$-modules in the full subcategory $\Oint$ of the BGG-category $\clO$ consisting of the integrable modules (see \cite{K90}, \cite{HK02}). We denote the Kashiwara operators and the other structure maps by
$$
\Etil_i,\ \Ftil_i,\ \vep_i,\ \vphi_i, \ \wt, \qu i \in \bbI.
$$

For each $n \in \N$, the set $I^n$ is equipped with a crystal structure of type $A_{2r}$ as follows. For $\bfs \in I^n$ and $i \in \bbI$, $\Etil_i \bfs$ (resp., $\Ftil_i \bfs$) is obtained from $\bfs$ by replacing the rightmost $i+\hf$ in $\bfs_i$ with $i-\hf$ (resp., the leftmost $i-\hf$ in $\bfs_i$ with $i+\hf$), and is $0$ if there are no $i+\hf$ (resp., $i-\hf$) in $\bfs_i$, where $0$ denotes a formal symbol. Also, $\vep_i(\bfs)$ and $\vphi_i(\bfs)$ are the numbers of $i+\hf$'s and $i-\hf$'s in $\bfs_i$, respectively. Finally, we have $\wt(\bfs) = \sum_{j = 1}^n s_j \vep_j$.

The set of isomorphism classes of irreducible $\U$-modules in $\Oint$ is parametrized by the set of partitions $\lm$ of length $2r+1$ with $\lm_{2r+1} = 0$; we denote by $L(\lm)$ the irreducible highest weight $\U$-module with highest weight $\lm$. For each $\mu \in P$, set $L(\mu) := L(\mu - \mu_{2r+1}\rho_{2r+1})$.

For each $\lm \in P$, $L(\lm)$ has a unique crystal basis $(\clL(\lm),\clB(\lm))$. $\clB(\lm)$ is identical to $\SST(\lm)$. The Kashiwara operators and other maps act on $\SST(\lm)$ as follows: For $T \in \SST(\lm)$, $\Etil_iT \in \SST(\lm)$ (resp., $\Ftil_i T \in \SST(\lm)$) is the unique semistandard tableau such that $\ME(\Etil_iT) = \Etil_i \ME(T)$ (resp., $\ME(\Ftil_i T) = \Ftil_i \ME(T)$) if $\Etil_i \ME(T) \neq 0$ (resp., $\Ftil_i \ME(T)$), and $\Etil_i T = 0$ (resp., $\Ftil_i T = 0$) otherwise. Also, we have
$$
\vep_i(T) = \vep_i(\ME(T)), \ \vphi_i(T) = \vphi_i(\ME(T)), \ \wt(T) = \wt(\ME(T)).
$$

\subsection{Quasi-$\jmath$-crystal bases}
Recall that $\Uj = \Uj_r$ has $(r-1)$ $\mathfrak{sl}_2$-triples: $(f_i, k_i, e_i)$ for $i = 2, \ldots, r$. Hence, one can define Kashiwara operators, $\ftil_i$ and $\etil_i$, in the same way as in the crystal basis theory for quantum groups. Also, by the results from Section \ref{r=1case}, we can define Kashiwara operators, $\ftil_1$ and $\etil_1$. Let us give the precise definition of these operators.

\begin{defi}\normalfont
Let $M$ be a $\Uj$-module. By the complete reducibility of $\Uj_1$-modules in $\Oj$, one can uniquely write $M \simeq \bigoplus_{\bflm \in (\Pj_1)_\pi} L(\bflm)^{\oplus m_\bflm}$ for some $m_\bflm \in \N$. Let $v_{\bflm,i}$, $1 \leq i \leq m_\bflm$ be a basis of the weight space of $L(\bflm)^{\oplus m_\bflm}$ of highest weight. We define linear operators $\ftil_1$ and $\etil_1$ on $M$ by
\begin{align}
\ftil_1(f_1^{(n)}v_{\bflm,i}) = f_1^{(n+1)}v_{\bflm,i}, \qqu \etil_i(f_1^{(n)}v_{\bflm,i}) = f_1^{(n-1)}v_{\bflm,i}. \nonumber
\end{align}
Note that this definition is independent of the choice of $v_{\bflm,i}$'s.
\end{defi}

Set $\A := \left\{ f/g \in \Q(p,q) \mid f, g \in p\Q[p,q,q\inv] + \Q[q],\ g \notin p \Q[p,q,q\inv] + q \Q[q] \right\}$; namely, $\A$ consists of all those $h \in \Q(p,q)$ for which $\lim_{q \rightarrow 0} (\lim_{p \rightarrow 0} h)$ exists. (Recall that $p$ and $q$ are independent.)

\begin{defi}\normalfont
Let $M$ be a $\Uj$-module and $\clL$ an $\A$-submodule of $M$. We say that $\clL$ is a quasi-$\jmath$-crystal lattice of $M$ if
\begin{enumerate}
\item[$(\qL1)$] $\clL$ is a free $\A$-module of rank $\dim_{\Q(p,q)} M$, and $\Q(p,q) \bigotimes_{\A} \clL = M$,
\item[$(\qL2)$] $\clL = \bigoplus_{\lambda \in \Lambda^\jmath} \clL_\lambda$, where $\clL_\lambda := \clL \cap M_\lambda$,
\item[$(\qL3)$] $\ftil_i(\clL) \subset \clL$ and $\etil_i(\clL) \subset \clL$ for all $i \in \Ij$.
\end{enumerate}
\end{defi}

If $\clL$ is a quasi-$\jmath$-crystal lattice of $M$, then the Kashiwara operators induce $\Q$-linear maps, denoted by the same symbols, on $\clL/q\clL$.

\begin{defi}\normalfont
Let $M$ be a $\Uj$-module, $\clL$ an $\A$-submodule of $M$, and $\clB$ a subset of $\clL/q\clL$. We say that $(\clL, \clB)$ is a quasi-$\jmath$-crystal basis if
\begin{enumerate}
\item[$(\qB1)$] $\clL$ is a quasi-$\jmath$-crystal lattice of $M$,
\item[$(\qB2)$] $\clB$ is a $\Q$-basis of $\clL/q\clL$,
\item[$(\qB3)$] $\clB = \bigsqcup_{\lambda \in \Lambda^\jmath} \clB_\lambda$, where $\clB_\lambda := \clB \cap (\clL_\lambda / q\clL_\lambda)$,
\item[$(\qB4)$] $\ftil_i(\clB) \subset \clB \sqcup \{0\}$ and $\etil_i(\clB) \subset \clB \sqcup \{0\}$ for all $i \in \Ij$,
\item[$(\qB5)$] for each $b,b' \in \clB$ and $i \in \Ij$, one has $\ftil_i(b) = b'$ if and only if $b = \etil_i(b')$.
\end{enumerate}
\end{defi}

\begin{defi}\normalfont
For a quasi-$\jmath$-crystal basis $(\clL,\clB)$ and $i \in \Ij$, we define three maps $\varphi_i : \clB \rightarrow \Z_{\geq 0}$, $\varepsilon_i : \clB \rightarrow \Z_{\geq 0}$, and $\wtj : \clB \rightarrow \Lambdaj$ by
$$
\varphi_i(b) := \max\{ n \mid \ftil_i^n(b) \neq 0 \}, \ \varepsilon_i(b) := \max\{ n \mid \etil_i^n(b) \neq 0 \}, \ \wtj(b) := \lambda \IF b \in \clB_\lambda.
$$
\end{defi}

\begin{ex}\normalfont
Let $r = 1$. For each $\bflm \in \Pj_1$, the irreducible $\Uj_1$-module $L(\bflm)$ has the following quasi-$\jmath$-crystal basis. Fix a highest weight vector $v \in L(\bflm)$. Let $\clL(\bflm)$ denote the $\A$-lattice spanned by $\{ f_1^{(n)} v \mid 0 \leq n \leq \bflm_0-\bflm_{-1} \}$, and set $\clB(\bflm) := \{ f_1^{(n)} v + \clL(\bflm)/q\clL(\bflm) \mid 0 \leq n \leq \bflm_0-\bflm_{-1} \}$. Then, the Kashiwara operators $\ftil_1$ and $\etil_1$ act on $\clL(\bflm)$ by:
\begin{align}
\ftil_1(f_1^{(n)} v) = f_1^{(n + 1)} v, \qu \etil_1(f_1^{(n)} v) = f_1^{(n - 1)} v. \nonumber
\end{align}
It is straightforward to check that $(\clL(\bflm),\clB(\bflm))$ is indeed a quasi-$\jmath$-crystal basis of $L(\bflm)$. In addition, one has $\varphi_1(f_1^{(n)}v + q\clL) = \bflm_0-\bflm_{-1}-n$, $\varepsilon_1(f_1^{(n)}v + q\clL) = n$, and $\wtj(f_1^{(n)}v + q\clL) = (2\bflm_0-\bflm_{-1}-\bflm_1-3n)\delta_1$.
\end{ex}

\begin{defi}\normalfont
Let $M$ be a $\Uj$-module and $(\clL, \clB)$ a quasi-$\jmath$-crystal basis of $M$. The quasi-$\jmath$-crystal graph associated with $(\clL, \clB)$ is the colored directed graph with vertex set $\clB$ and edges $b \xrightarrow[]{i} b'$, where $b, b' \in \clB,\ i \in \Ij$ are such that $\ftil_i b = b'$.
\end{defi}

We often identify $\clB$ with its quasi-$\jmath$-crystal graph. Hence, phrases such as ``$\clB$ is connected'' and ``a connected component of $\clB$'' make sense.

\begin{prop}
Let $M \in \Oj$ be a $\Uj$-module with a quasi-$\jmath$-crystal basis $(\clL,\clB)$. For each $i \in \Ij$ and $m \in \clL_\lm$, consider the expression $m = \sum_{j=0}^N f_i^{(j)}m_j$, where $m_j \in M_{\lm+j\gamma_i} \cap \Ker e_i$. Then, the following hold:
\begin{enumerate}
\item $m_j \in \clL$ for all $j = 0,\ldots,N$.
\item If $m + q\clL \in \clB$, then there exists a unique $j_0$ such that $u_j \in q\clL$ for all $j \neq j_0$, and $m + q\clL = m_{j_0} + q\clL$.
\end{enumerate}
\end{prop}

\begin{proof}
The assertion follows from the same argument as the ordinary crystal basis theory.
\end{proof}

\begin{prop}
Let $M \in \Oj$ be a $\Uj$-module with a quasi-$\jmath$-crystal basis $(\clL,\clB)$. Let $\lm \in \Lambdaj$ and set $a := (\lm,\beta_1)$. For each $u \in \clL_\lm \cap \Ker e_1$, consider the unique expression $u = \sum_{b=0}^N u_b$, where $u_b \in M_\lm$ is a $\Uj_1$-highest weight vector of such that $k_1u_b = q^au_b$ and $h_1u_b = [b]\{ a-b-1 \}u_b$. Then, the following hold:
\begin{enumerate}
\item $u_b \in \clL$ for all $b = 0,\ldots,N$.
\item If $u + q\clL \in \clB$, then there exists a unique $b_0$ such that $u_b \in q\clL$ for all $b \neq b_0$, and $u + q\clL = u_{b_0} + q\clL$.
\end{enumerate}
\end{prop}

\begin{proof}
We prove the assertion by induction on $N$. When $N=0$, there is nothing to prove. When $N > 0$, consider $\ftil_1^Nu \in \clL$. Since we have $\ftil_1^N u = f_1^{(N)}u = f_1^{(N)}u_N$, it holds that $f_1^{(N)}u_N \in \clL$. Hence we have $u_N = \etil_1^N f_1^{(N)}u_N \in \clL$. This implies that $u-u_N = \sum_{b=0}^{N-1}u_b$ and $u-u_N \in \clL \cap \Ker e_1$, and hence, by induction hypothesis, we have $u_b \in \clL$ for all $b$. Now, let us assume that $u + q\clL \in \clB$. Set $b_0 := \vphi_1(u+q\clL)$. Since $0 \neq \ftil_1^{b_0}u = \sum_{b=b_0}^N f_1^{(b_0)}u_b$, it holds that $0 \leq b_0 \leq N$. Then, we have $\sum_{b=b_0}^N f_1^{(b_0)}u_b \in \clL \setminus q\clL$, and $\sum_{b=b_0+1}^N f_1^{(b_0+1)}u_b = \ftil_1^{b_0+1}u \in q\clL$. Thus, we have $u_b \in q\clL$ for all $b > b_0$, and $f_1^{(b_0)}u_{b_0} + q\clL = \ftil_1^{b_0}(u+q\clL)$ (equivalently, $u_{b_0}+q\clL = u+q\clL$). Then, we have $u-u_{b_0} \in q\clL$, and hence, $u_b \in q\clL$ for all $b \neq b_0$. This completes the proof.
\end{proof}

Now, the following theorem can be proved in a similar way to the ordinary crystal basis theory.

\begin{theo}
Let $M \in \Oj$ be a $\Uj_1$-module. Then, $M$ has a quasi-$\jmath$-crystal basis $(\clL,\clB)$. Moreover, if $M \simeq \bigoplus_{\bflm \in (\Pj_1)_\pi} L(\bflm)^{\oplus m_\bflm}$ for some $m_\bflm \in \Z_{\geq 0}$, then there exists an isomorphism $M \rightarrow \bigoplus_{\bflm \in (\Pj_1)_\pi} L(\bflm)^{\oplus m_\bflm}$ of $\Uj_1$-modules which induces an isomorphism
$$
(\clL,\clB) \rightarrow \left( \bigoplus_{\bflm \in (\Pj_1)_\pi} \clL(\bflm)^{\oplus m_\bflm}, \bigoplus_{\bflm \in (\Pj_1)_\pi} \clB(\bflm)^{\oplus m_\bflm} \right).
$$
\end{theo}

\subsection{Tensor product rule}\label{5.2}
Recall that $\Uj$ is a right coideal of $\U$, i.e., $\Delta(\Uj) \subset \Uj \otimes \U$. Hence, we are interested in the $\Uj$-module structure of the tensor product of a $\Uj$-module and a $\U$-module. Let $\bfV = \bfV_r$ denote the vector representation of $\U$. It is spanned by $\{ u_i \mid i \in I \}$, and is equipped with a $\U$-module structure by:
\begin{align}
F_{j} u_i = \delta_{j-\hf,i} u_{i+1}, \qu E_{j} u_i = \delta_{j+\hf,i} u_{i-1}, \qu K_j u_i = q^{(\alpha_j,\epsilon_i)} u_i. \nonumber
\end{align}
If we set $\bfL = \bfL_r := \bigoplus_{i \in I} \A u_i$, $\bfB = \bfB_r := \{ u_i + q\bfL_r \mid i \in I  \}$, then, $(\bfL,\bfB)$ is an ordinary crystal basis of $\bfV$.


We first consider the case $r = 1$. Recall that the irreducible $\Uj_1$-module $L(\bflm)$, $\bflm \in \Pj_1$ has a quasi-$\jmath$-crystal basis $(\clL(\bflm),\clB(\bflm))$. If $L(\bflm) = L(a;b)$ for some $a \in \Z$ and $b \in \Z_{\geq }0$, then we write $\clL(a;b) = \clL(\bflm)$, $\clB(a;b) = \clB(\bflm)$.

\begin{prop}\label{f1n}
Let $a \in \Z$, $b \in \Z_{\geq 0}$. Then we have an isomorphism
\begin{align}
L(a;b) \otimes \bfV \simeq L(a+2;b+1) \oplus L(a-1;b) \oplus L(a-1;b-1) \nonumber
\end{align}
of $\Uj_1$-modules. Moreover, $(\clL(a;b) \otimes \bfL, \clB(a;b) \otimes \bfB)$ is a quasi-$\jmath$-crystal basis of $L(a;b) \otimes \bfV$.
\end{prop}

\begin{proof}
Let $v \in L(a;b)$ be a highest weight vector, and set
\begin{align}
\begin{split}
\vo &:= v \otimes u_0, \\
\vi &:= v \otimes u_1 - \frac{q^{-b+1}(q-q\inv)}{\{a-b-1\}} f_1v \otimes u_0 - pq^{a-2b} v \otimes u_{-1}, \\
\vmi &:= f_1v \otimes u_0 - q^b[b] v \otimes u_{-1} - pq^{a-b-2}[b] v \otimes u_1.
\end{split} \nonumber
\end{align}
Then, by direct calculation, we obtain
\begin{align}
\begin{split}
h_1 \vo &= [b+1]\{(a+2)-(b+1)-1\} \vo, \\
h_1 \vi &= [b]\{(a-1)-b-1\} \vi, \\
h_1 \vmi &= [b-1]\{(a-1)-(b-1)-1\} \vmi. \\
\end{split} \nonumber
\end{align}
These equations, together with Corollary \ref{CharacterizationOf1highestvec} and Theorem \ref{IrreducibleModulesForr=1}, show that $\Uj_1 \vo \simeq L(a+2;b+1)$, $\Uj_1 \vi \simeq L(a-1;b)$, and $\Uj_1 \vmi \simeq L(a-1;b-1)$. Since $\dim(L(a;b) \otimes \bfV) = 3b = (b+1) + b + (b-1) = \sum_{k=-1}^1 \dim \Uj_1 v \fbox{$k$}$, we see that $L(a;b) \otimes \bfV = \Uj_1 \vo \oplus \Uj_1 \vmi \oplus \Uj_1 \vi$. Also, we calculate as:
\begin{align}
\begin{split}
f_1^{(n)}(\vo) &= f_1^{(n-1)}v \otimes u_{-1} + q^n f_1^{(n)}v \otimes u_0 + pq^{a-n+1} f_1^{(n-1)}v \otimes u_1 \\
&\in \begin{cases}
v \otimes u_0 + q\clL(a;b) \otimes \bfL & \IF n = 0, \\
f_1^{(n-1)}v \otimes u_{-1} + q\clL(a;b) \otimes \bfL & \IF 0 \leq n \leq b+1,
\end{cases} \\
f_1^{(n)}(\vi) & = \frac{q^{-n}\{ a-b-n-1 \}}{\{ a-b-1 \}}f_1^{(n)}v \otimes u_1 - \frac{q^{-b+n+1}(q^{n+1}-q^{-n-1})}{\{ a-b-1 \}}f_1^{(n+1)}v \otimes u_0 \\
&\hspace{6.5cm} -\frac{pq^{a-2b}\{ a-b-n-1 \}}{\{ a-b-1 \}}f_1^{(n)}v \otimes u_{-1} \\
&\in f_1^{(n)}v \otimes u_1 + q\clL(a;b) \otimes \bfL \qu \IF 0 \leq n \leq b, \\
f_1^{(n)}(\vmi) &= q^{n}[n+1]f_1^{(n+1)}v \otimes u_0 - q^b[b-n]f_1^{(n)}v \otimes u_{-1} - pq^{a-b-n-2}[b-n]f_1{(n)}v \otimes u_1\\
&\in f_1^{(n+1)}v \otimes u_0 + q\clL(a;b) \otimes \bfL \qu \IF 0 \leq n \leq b-1.
\end{split} \nonumber
\end{align}
Since $\ftil_1^n(v \fbox{$k$}) = f_1^{(n)}(v \fbox{$k$})$, $k \in \{ 0,\pm1 \}$, these equations imply that the $\A$-span of $\{ \ftil_1^n(v \fbox{$k$}) \mid k \in \{ 0,\pm1 \},\ n \in \Z_{\geq 0} \}$ coincides with $\clL(a;b) \otimes \bfL$, and that $\{ \ftil_1^n(v \fbox{$k$}) + q\clL(a;b) \otimes \bfL \mid k \in \{ 0,\pm1 \},\ n \in \Z_{\geq 0} \} \setminus \{0\}$ is identical to $\clB(a;b) \otimes \bfB$. Now, it is easy to verify that $(\clL(a;b) \otimes \bfL, \clB(a;b) \otimes \bfB)$ is a quasi-$\jmath$-crystal basis of $L(a;b) \otimes \bfV$. This proves the proposition.
\end{proof}

We give the quasi-$\jmath$-crystal graph of $\clB(a;b) \otimes \bfB$:
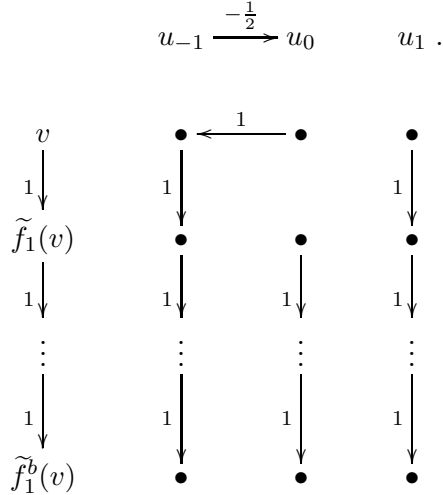
\begin{align}
\xymatrix{
                     & u_{-1} \ar[r]^{-\hf}         & u_0 & u_1     \\
v \ar[d]_1           & \bullet \ar[d]_1& \bullet \ar[l]_1 & \bullet \ar[d]_1 \\
\ftil_1(v) \ar[d]_1  & \bullet \ar[d]_1& \bullet \ar[d]_1 & \bullet \ar[d]_1 \\
\vdots \ar[d]_1      & \vdots  \ar[d]_1& \vdots  \ar[d]_1 & \vdots  \ar[d]_1\\
\ftil_1^b(v)         & \bullet         & \bullet          & \bullet
}. \nonumber
\end{align}

Let $N \in \N$. Applying the proposition above repeatedly, we see that the tensor product module $\bfV^{\otimes N}$ has a quasi-$\jmath$-crystal basis $(\bfL^{\otimes N},\bfB^{\otimes N})$; we denote $u_{i_1} \otimes \cdots \otimes u_{i_N} + q\bfL^{\otimes N} \in \bfB^{\otimes N}$, $i_1,\ldots,i_N \in I$, by $(i_1,\ldots,i_N)$. With this notation, we identify $\bfB^{\otimes N}$ with $I^N$.

\begin{lem}
Let $\bfs \in \bfB^{\otimes N}$. Then, $\etil_1 \bfs$ (resp., $\ftil_1 \bfs$) is obtained from $\bfs$ by replacing the rightmost $-1$ in $\bfs_1$ with $0$ (resp., the leftmost $0$ in $\bfs_1$ with $-1$), and is $0$ if there are no $-1$ (resp., $0$) in $\bfs_1$. Also, $\vep_1(\bfs)$ and $\vphi_1(\bfs)$ are the number of $-1$ and $0$ in $\bfs_1$, respectively.
\end{lem}

\begin{proof}
By induction on $N$, and apply Proposition \ref{f1n}.
\end{proof}

More generally, we obtain the following theorem. As in the ordinary crystal basis theory, the proof is given by embedding the crystal basis of a $\U_3$-module into $(\bfL^{\otimes N}, \bfB^{\otimes N})$ for a suitable $N$.

\begin{theo}\label{TensorRuleForr=1}
Let $M$ be a $\Uj_1$-module with a quasi-$\jmath$-crystal basis $(\clL, \clB)$, and $N$ a $\U_3$-module with a crystal basis $(\clL', \clB')$. Then, $M \otimes N$ has a quasi-$\jmath$-crystal basis $(\clL \otimes \clL', \clB \otimes \clB')$, on which the Kashiwara operators act as follows:
\begin{align}
\ftil_1(b \otimes b') = \begin{cases}
b \otimes \Etil_{-\hf}(b') \qu & \IF \varepsilon_1(b) < \varepsilon_{-\hf}(b'), \\
\ftil_1(b) \otimes b' \qu & \IF \varepsilon_1(b) \geq \varepsilon_{-\hf}(b'),
\end{cases} \nonumber\\
\etil_1(b \otimes b') = \begin{cases}
b \otimes \Ftil_{-\hf}(b') \qu & \IF \varepsilon_1(b) \leq \varepsilon_{-\hf}(b'), \\
\etil_1(b) \otimes b' \qu & \IF \varepsilon_1(b) > \varepsilon_{-\hf}(b').
\end{cases} \nonumber
\end{align}
\end{theo}

Now, we turn to the case of a general $r$. Recall that Kashiwara operators $\ftil_i$ and $\etil_i$ for $i \neq 1$ are defined by means of the $\mathfrak{sl}_2$-triple $( f_i,k_i,e_i )$. Therefore, the next proposition follows from a standard argument; see, for example, {\cite[Section 4.4]{HK02}}.

\begin{prop}\label{TensorRuleForGeneralr}
Let $M$ be a $\Uj$-module having a quasi-$\jmath$-crystal basis $(\clL,\clB)$. Then $(\clL \otimes \bfL, \clB \otimes \bfB)$ is a quasi-$\jmath$-crystal basis of $M \otimes \bfV$, on which the Kashiwara operators act as follows: $\ftil_1$ and $\etil_1$ acts as described in Theorem \ref{TensorRuleForr=1}; for $i \in \Ij \setminus \{1\}, b \in \clB, j \in \{ -r, -r+1, \ldots, r \}$,
\begin{align}
\begin{split}
\ftil_i(b \otimes u_j) &= \begin{cases}
0 & \IF j = i \AND \ftil_i^2(b) = 0, \\
b \otimes u_i & \IF j = i-1 \AND \ftil_i(b) = 0, \\
b \otimes u_{-i} & \IF j=-(i-1) \AND \etil_i(b) = 0, \\
\ftil_i(b) \otimes u_j & \OW,
\end{cases} \\
\etil_i(b \otimes u_j) &= \begin{cases}
b \otimes u_{i-1} & \IF j=i \AND \ftil_i(b) = 0, \\
0 & \IF j=-(i-1) \AND \etil_i^2(b) = 0, \\
b \otimes u_{-(i-1)} & \IF j=-i \AND \etil_i(b) = 0, \\
\etil_i(b) \otimes u_j & \OW.
\end{cases}
\end{split} \nonumber
\end{align}
\end{prop}

The action of $\ftil_i$ for $i \neq 1$ is visualized as:
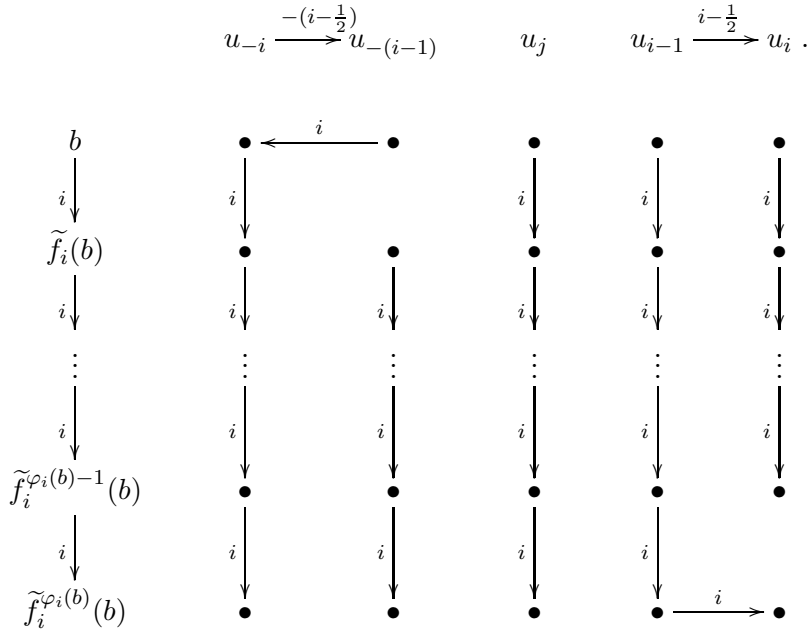
\begin{align}
\xymatrix{
 & u_{-i} \ar[r]^{-(i-\hf)} & u_{-(i-1)} & u_j & u_{i-1} \ar[r]^{i-\hf} & u_i \\
b \ar[d]_i & \bullet \ar[d]_i & \bullet \ar[l]_i & \bullet \ar[d]_i & \bullet \ar[d]_i & \bullet \ar[d]_i \\
\ftil_i(b) \ar[d]_i & \bullet \ar[d]_i & \bullet \ar[d]_i & \bullet \ar[d]_i & \bullet \ar[d]_i & \bullet \ar[d]_i \\
\vdots \ar[d]_i & \vdots \ar[d]_i & \vdots \ar[d]_i & \vdots \ar[d]_i & \vdots \ar[d]_i & \vdots \ar[d]_i \\
\ftil_i^{\varphi_i(b)-1}(b) \ar[d]_i & \bullet \ar[d]_i & \bullet \ar[d]_i & \bullet \ar[d]_i & \bullet \ar[d]_i & \bullet \\
\ftil_i^{\varphi_i(b)}(b) & \bullet & \bullet & \bullet & \bullet \ar[r]^i & \bullet
}. \nonumber
\end{align}

The following theorem describes the tensor product rule for the Kashiwara operators $\ftil$'s and $\etil$'s in full generality. The proof is given by embedding the crystal basis of a $\U$-module into $(\bfL^{\otimes N}, \bfB^{\otimes N})$ for a suitable $N$.

\begin{theo}\label{TensorRule}
Let $M$ be a $\Uj$-module having a quasi-$\jmath$-crystal basis $(\clL, \clB)$, and $N$ a $\U$-module having a crystal basis $(\clL', \clB')$. Then, $M \otimes N$ has a quasi-$\jmath$-crystal basis $(\clL \otimes \clL', \clB \otimes \clB')$, on which the Kashiwara operators act as follows: for $b \in \clB$ and $b' \in \clB'$,
\begin{align}
\ftil_1(b \otimes b') &= \begin{cases}
b \otimes \Etil_{-\hf}(b') \qu & \IF \varepsilon_1(b) < \varepsilon_{-\hf}(b'), \\
\ftil_1(b) \otimes b' \qu & \IF \varepsilon_1(b) \geq \varepsilon_{-\hf}(b'),
\end{cases} \nonumber\\
\etil_1(b \otimes b') &= \begin{cases}
b \otimes \Ftil_{-\hf}(b') \qu & \IF \varepsilon_1(b) \leq \varepsilon_{-\hf}(b'), \\
\etil_1(b) \otimes b' \qu & \IF \varepsilon_1(b) > \varepsilon_{-\hf}(b'),
\end{cases} \nonumber\\
\ftil_i(b \otimes b') &= \begin{cases}
b \otimes \Etil_{-(i - \hf)}(b') \qu & \IF \varepsilon_{i - \hf}(b') < \varphi_i(b) \AND \varepsilon_i(b) < \varepsilon_{-(i - \hf)}(b'),  \Or \\
& \IF \varepsilon_{i - \hf}(b') \geq \varphi_i(b) \AND \varepsilon_i(b) + \varepsilon_{i-\hf}(b') - \varphi_i(b) < \varepsilon_{-(i-\hf)}(b'), \\
\ftil_i(b) \otimes b' \qu & \IF \varepsilon_{i - \hf}(b') < \varphi_i(b) \AND \varepsilon_i(b) \geq \varepsilon_{-(i-\hf)}(b'), \\
b \otimes \Ftil_{i-\hf}(b') \qu & \IF \varepsilon_{i - \hf}(b') \geq \varphi_i(b) \AND \varepsilon_i(b) + \varepsilon_{i-\hf}(b') - \varphi_i(b) \geq \varepsilon_{-(i-\hf)}(b'),
\end{cases} \nonumber\\
\etil_i(b \otimes b') &= \begin{cases}
b \otimes \Ftil_{-(i - \hf)}(b') \qu & \IF \varepsilon_{i - \hf}(b') \leq \varphi_i(b) \AND \varepsilon_i(b) \leq \varepsilon_{-(i - \hf)}(b'),  \Or \\
& \IF \varepsilon_{i - \hf}(b') > \varphi_i(b) \AND \varepsilon_i(b) + \varepsilon_{i-\hf}(b') - \varphi_i(b) \leq \varepsilon_{-(i-\hf)}(b'), \\
\etil_i(b) \otimes b' \qu & \IF \varepsilon_{i - \hf}(b') \leq \varphi_i(b) \AND \varepsilon_i(b) > \varepsilon_{-(i-\hf)}(b'), \\
b \otimes \Etil_{i-\hf}(b') \qu & \IF \varepsilon_{i - \hf}(b') > \varphi_i(b) \AND \varepsilon_i(b) + \varepsilon_{i-\hf}(b') - \varphi_i(b) > \varepsilon_{-(i-\hf)}(b').
\end{cases} \nonumber
\end{align}
\end{theo}

\begin{cor}\label{crystal=qjcrystal}
Let $N \in \Oint$ be a $\U$-module with a crystal basis $(\clL',\clB')$. Then, $(\clL',\clB')$ is also a quasi-$\jmath$-crystal basis of $N$. Furthermore, for each $b \in \clB$ and $i \in \Ij$, we have the following:
\begin{align}
\begin{split}
\ftil_1(b) &= \Etil_{-\hf}(b), \\
\etil_1(b) &= \Ftil_{-\hf}(b), \\
\ftil_i(b) &= \begin{cases}
\Etil_{-(i - \hf)}(b) \qu & \IF \varepsilon_{i-\hf}(b) < \varepsilon_{-(i-\hf)}(b), \\
\Ftil_{i-\hf}(b) \qu & \IF \varepsilon_{i-\hf}(b) \geq \varepsilon_{-(i-\hf)}(b),
\end{cases} \nonumber\\
\etil_i(b) &= \begin{cases}
\Ftil_{-(i - \hf)}(b) \qu & \IF \varepsilon_{i-\hf}(b) \leq \varepsilon_{-(i-\hf)}(b), \\
\Etil_{i-\hf}(b) \qu & \IF \varepsilon_{i-\hf}(b) > \varepsilon_{-(i-\hf)}(b).
\end{cases} \nonumber
\end{split} \nonumber
\end{align}
\end{cor}

\begin{proof}
Apply Theorem \ref{TensorRule} for $M = L(\emptyset;\emptyset)$, which is the trivial module of $\Uj$.
\end{proof}

Recall that $\bfB^{\otimes N}$ is identified with $I^N = \{ -r,\ldots,r \}^N$, and its crystal structure is described in the beginning of this section. Applying Corollary \ref{crystal=qjcrystal} to the crystal basis $(\bfL^{\otimes N},\bfB^{\otimes N})$, we obtain a quasi-$\jmath$-crystal structure of $\bfB^{\otimes N}$.

\begin{cor}\label{qjcry for words}
The quasi-$\jmath$-crystal basis $\bfB^{\otimes N} = I^N$ obtained from Corollary \ref{crystal=qjcrystal} is described as follows: for $\bfs \in I^N$, $\etil_1 \bfs$ (resp., $\ftil_1 \bfs$) is obtained from $\bfs$ by replacing the rightmost $-1$ in $\bfs_1$ with $0$ (resp., the leftmost $0$ in $\bfs_1$ with $-1$), and is $0$ if there are no $-1$ (resp., $0$) in $\bfs_1$. For $i \in \Ij \setminus \{1\}$, $\etil_i \bfs$ is obtained from $\bfs$ by replacing the rightmost $i$ in $\bfs_i $with $i-1$ if $i \in \bfs_i$, or by replacing the rightmost $-i$ in $\bfs_i$ with $-(i-1)$ if $i \notin \bfs_i$, or is $0$ if $i,-i \notin \bfs_i$. Finally, $\ftil_i \bfs$ is obtained from $\bfs$ by replacing the leftmost $-(i-1)$ in $\bfs_i $with $-i$ if $-(i-1) \in \bfs_i$, or by replacing the leftmost $i-1$ in $\bfs_i$ with $i$ if $-(i-1) \notin \bfs_i$, or is $0$ if $i-1,-(i-1) \notin \bfs_i$.
\end{cor}

%

\section{Quasi-$\jmath$-crystal basis of $\bfV^{\otimes d}$}\label{jcry}
In this section, we construct a quasi-$\jmath$-crystal basis for each irreducible highest weight $\Uj$-module. For this purpose, we need some results concerning the left cell representations of the Hecke algebra of type $B$.
\subsection{Hecke algebra of type $B$}
Let $d \in \Z_{>0}$, and $W = W_d = \la s_0,s_1,\ldots,s_{d-1} \ra$ denote the Weyl group of type $B_d$, and $\Hc = \Hc(W)$ the associated Hecke algebra with unequal parameter $p,q$. Namely, $\Hc$ is the unital associative algebra over $\Z[p,p\inv,q,q\inv]$ spanned by $\{ H_w \mid w \in W \}$ with the product given by
$$
H_i H_w = \begin{cases}
H_{s_iw} \qu & \IF s_iw > w, \\
H_{s_iw} - (q_i-q_i\inv) H_w \qu & \IF w < s_iw,
\end{cases}
$$
where $<$ denotes the Bruhat order, and $q_i = q$ if $i = 1,\ldots,d-1$, while $q_0 = p$.

There exists a unique $\Z$-algebra automorphism $\ol{\ \cdot \ }$ of $\Hc$ such that $\ol{H_w} = H_{w\inv}\inv$.

\begin{theo}[{\cite[Theorem 1.1]{KL79}}, {\cite[Theorem 5.2]{L03}}]\label{KL-bases}
For each $w \in W$, there exists a unique $C_w \in \Hc$ such that
\begin{enumerate}
\item $\overline{C_w} = C_w$.
\item $C_w = H_w + \sum_{y < w} c_{y,w} H_y$ for some $c_{y,w} \in p\Z[p,q,q\inv] \oplus q\Z[q]$.
\end{enumerate}
\end{theo}

$\{ C_w \}_{w \in W}$ forms a linear basis of $\Hc(W_d)$; we call it the Kazhdan-Lusztig basis.

\subsection{Cell representations}
Let us recall from \cite{KL79} and \cite{L03} the notion of left cells of $W$ and the associated cell representations.
\begin{defi}\normalfont
Let $y,w \in W$.
\begin{enumerate}
\item $y \rightarrow_L w$ if the coefficient of $C_y$ in the expansion of $H_s C_w$ with respect to the basis $\{ C_x \mid x \in W \}$ is nonzero for some $s \in S$.
\item $y \leq_L w$ if there exist $y = y_0,y_1,\ldots,y_l = w \in W$ such that $y_{i-1} \rightarrow_L y_i$.
\item $y \Lcell w$ if $y \leq_L w$ and $w \leq_L y$.
\item $y <_L w$ if $y \leq_L w$ and $y \not\Lcell w$.
\item Each equivalence class of $W$ with respect to $\Lcell$ is called a left cell of $W$. We denote by $L(W)$ the set of left cells.
\end{enumerate}
\end{defi}

For each $X \in L(W)$ and $x \in X$, set
\begin{align}
&C_{\leq_L X} = \bigoplus_{y \leq_L x} \Z[p^{\pm1},q^{\pm1}] C_y, \qu C_{<_L X} = \bigoplus_{y <_L x} \Z[p^{\pm1},q^{\pm1}] C_y, \qu C^L_X = C_{\leq_L X} / C_{<_L X}. \nonumber
\end{align}
Note that these are independent of the choice of $x \in X$. We denote the image of $m \in C_{\leq_L X}$ under the quotient map $C_{\leq_L X} \rightarrow C^L_X$ by $[m]_X$. Then, $C^L_X$ has a basis $\{ [C_x]_X \mid x \in X \}$.

\begin{theo}[{\cite[Theorem 7.7]{BI03}}]
Let $X \in L(W)$. Then, $C^L_X$ is an irreducible $\Hc$-module.
\end{theo}

\subsection{Parabolic Kazhdan-Lusztig bases}\label{Parabolic Kazhdan-Lusztig bases}
Throughout this subsection, we fix a subset $J \subset \{ 0,1,\ldots,d-1 \}$ arbitrarily. Let $W_J$ denote the parabolic subgroup of $W$ generated by $\{ s_j \mid j \in J \}$, ${}^JW$ the minimal length coset representatives in $W_J \backslash W$, and $w_J \in W_J$ the longest element. Also, we set
\begin{align}
x_J\index{xJ@$x_J$} := \sum_{w \in W_J} q_{w_J}q_w\inv H_w \in \Hc. \nonumber
\end{align}

\begin{lem}\label{property of x_J}
Let $j \in J$. Then, the following hold.
\begin{enumerate}
\item $x_J H_j = q_j\inv x_J$.
\item $x_J = C_{w_J}$. In particular, $\ol{x_J} = x_J$.
\end{enumerate}
\end{lem}

\begin{proof}
The assertion $(1)$ follows from a direct calculation and the fact that $W_J = \{ w \in W_J \mid w < s_jw \} \sqcup \{ w \in W_J \mid s_jw < w \}$. The proof of $(2)$ can be found in \cite[Proposition 1.17 (2)]{X94}.
\end{proof}

By Lemma \ref{property of x_J} $(1)$, the right ideal $x_J \Hc$ of $\Hc$ has a basis $\{ x_J H_w \mid w \in {}^JW \}$. Also, by Lemma \ref{property of x_J} $(2)$, $x_J \Hc$ is closed under the involution $\ol{\ \cdot \ }$. Hence, we can construct an analog of the Kazhdan-Lusztig basis in the ideal $x_J \Hc$:

\begin{theo}[{\cite[Proposition 3.2]{Deo87}}]\label{parabolic KL basis}
For each $w \in {}^JW$, there exists a unique ${}^JC_w \in x_J \Hc$ such that
\begin{enumerate}
\item $\ol{{}^JC_w} = {}^JC_w$.
\item ${}^JC_w = x_J(H_w + \sum_{\substack{y \in {}^JW \\ y < w}} {}^Jc_{y,w} H_y)$ for some ${}^Jc_{y,w} \in p\Z[p,q^{\pm1}] \oplus q\Z[q]$.
\end{enumerate}
\end{theo}

Clearly, $\{ {}^JC_w \mid w \in {}^JW \}$ is a linear basis of $x_J \Hc$, called the parabolic Kazhdan-Lusztig basis.

\begin{prop}[{\cite[Proposition 3.4]{Deo87}}]\label{JCw}
Let $w \in {}^JW$. Then, we have ${}^JC_w = C_{w_Jw}$.
\end{prop}

\begin{proof}
We have
\begin{align}
{}^JC_w &= x_J \sum_{\substack{y \in {}^JW \\ y \leq w}} {}^Jc_{y,w}H_y = \sum_{\substack{y \in {}^JW \\ y \leq w}} \sum_{x \in W_J} q_{w_J}q_x\inv {}^Jc_{y,w} H_{xy}. \nonumber
\end{align}
This shows that ${}^JC_w - H_{w_Jw} \in \bigoplus_{z < w_Jw} (p\Z[p,q^{\pm1}] \oplus q\Z[q]) H_z$. Hence, by Theorem \ref{KL-bases}, ${}^JC_w$ coincides with $C_{w_Jw}$.
\end{proof}

\begin{prop}\label{from KL to parabolic KL}
Let $y \in W$. Then, we have
\begin{align}
x_J C_y = \sum_{\substack{w \in {}^JW \\ w_Jw \leq_L y}} \alpha_w {}^JC_w, \nonumber
\end{align}
for some $\alpha_w \in \Z[p^{\pm1},q^{\pm1}]$.
\end{prop}

\begin{proof}
Let us write
\begin{align}
x_J C_y = \sum_{w \in {}^JW} \alpha_w {}^JC_w = \sum_{w \in {}^JW} \alpha_w C_{w_Jw} \qu \Forsome \alpha_w \in \Z[\Gamma]. \nonumber
\end{align}
On the other hand, by the definition of $\leq_L$, we can write
\begin{align}
x_J C_y = \sum_{z \leq_L y} \beta_z C_z \qu \Forsome \beta_z \in \Z[\Gamma]. \nonumber
\end{align}
Combining these two equations, we obtain the desired assertion.
\end{proof}

\subsection{Functor $\clFj$}
Set $\bfH := \Q(p,q) \otimes_{\Z[p^{\pm1},q^{\pm1}]} \Hc$. Recall from \cite{BWW16} that we have established the $q$-Schur duality between $\Uj$ and $\bfH$, namely, we defined a right action of $\bfH$ on $\bfV^{\otimes d}$ which commutes with the left action of $\Uj$. As a right $\bfH$-module, $\bfV^{\otimes d}$ decomposes as
$$
\bfV^{\otimes d} \simeq \bigoplus_{f \in I_+^d} x_{J(f)} \bfH,
$$
where $I = \{ -r,\ldots,r \}$, $I_+^d := \{ f = (f_1,\ldots,f_d) \in I^d \mid 0 \leq f_1 \leq \cdots\leq f_d \}$ and $J(f) := \{ j \mid f_j = 0 \text{ or } f_j = f_{j+1} \}$. From this result, we obtain an exact functor $\clFj$ from the category of finite-dimensional $\bfH$-modules to $\Oj$ defined by $\clFj(M) := \bfV^{\otimes d} \otimes_{\bfH} M = \bigoplus_{f \in I_+^d} x_{J(f)} M$.

In \cite{BWW16}, it is proved that the Kazhdan-Lusztig basis $\{ {}^{J(f)} C_w \mid f \in I_+^d,\ w \in {}^{J(f)} W \}$ coincides with the $\jmath$-canonical basis $\{ b^{\jmath}_f \mid f \in I^d \}$ of $\bfV^{\otimes d} = \clFj(\bfH)$, i.e., we have $b^{\jmath}_{fw} = {}^{J(f)} C_w$ if $f \in I_+^d$ and $w \in {}^{J(f)}W$. This implies that $\bfL^{\otimes d} = \Span_{\A} \{ {}^{J(f)} C_w \mid f \in I_+^d,\ w \in {}^{J(f)}W \}$, and $\bfB^{\otimes d} = \{ {}^{J(f)}C_w + q\bfL^{\otimes d} \mid f \in I_+^d,\ w \in {}^{J(f)}W \}$.

\begin{prop}
Let $X \in L(W)$ and $x \in X$.
\begin{enumerate}
\item $\bfC_{\leq_L X} := \clFj(\Q(p,q) \otimes_{\Z[p^{\pm1},q^{\pm1}]} C_{\leq_L X})$ is a $\Uj$-submodule of $\bfV^{\otimes d}$ with a basis
$$\{ {}^{J(f)} C_w \mid f \in I^d_+,\ w \in {}^{J(f)}W \AND w_{J(f)}w \leq_L x \}.$$
\item $\bfC_{<_L X} := \clFj(\Q(p,q) \otimes_{\Z[p^{\pm1},q^{\pm1}]} C_{<_L X})$ is a $\Uj$-submodule of $\bfV^{\otimes d}$ with a basis
$$\{ {}^{J(f)} C_w \mid f \in I^d_+,\ w \in {}^{J(f)}W \AND w_{J(f)}w <_L x \}.$$
\item $\bfC^L_X := \bfC_{\leq_L X}/\bfC_{<_L X}$ is an irreducible $\Uj$-subquotient of $\bfV^{\otimes d}$ with a basis $\{ [{}^{J(f)} C_w]_X \mid f \in I^d_+ \AND w \in {}^{J(f)}W \cap w_{J(f)}X \}$, where $[m]_X$ denotes the image of $m \in \bfC_{\leq_L X}$ under the canonical projection $\bfC_{\leq_L X} \rightarrow \bfC^L_X$.
\end{enumerate}
\end{prop}

\begin{proof}
Parts $(1)$ and $(2)$ follows from Propositions \ref{JCw} and \ref{from KL to parabolic KL}. Part $(3)$ follows from $(1)$ and $(2)$.
\end{proof}

\begin{cor}
We obtain the irreducible decomposition $\bfV^{\otimes d} \simeq \bigoplus_{\substack{X \in L(W) \\ \bfC^L_X \neq 0}} \bfC^L_X$.
\end{cor}

\begin{prop}\label{qjc of CLX}
Let $X \in L(W)$. Then, $\bfC^L_X$ has a quasi-$\jmath$-crystal basis.
\end{prop}

\begin{proof}
Let $x \in X$. Set $\clL_{\leq_L X}$ to be the $\Ao$-span of $\{ {}^{J(f)} C_w \mid f \in I^d_+,\ w \in {}^{J(f)}W \AND w_{J(f)}w \leq_L x \}$, and $\clB_{\leq_L X} := \{ {}^{J(f)} C_w + q\clL_{\leq_L X} \mid f \in I^d_+,\ w \in {}^{J(f)}W \AND w_{J(f)}w \leq_L x \}$. Since, $\clL_{\leq_L X} \subset \bfL^{\otimes d}$, we have $\xtil_i(\clL_{\leq_L X}) \subset \bfL^{\otimes d}$ for all $x \in \{ e,f \}$, $i \in \Ij$. On the other hand, as the Kashiwara operators preserve the submodules, we have $\xtil_i(\clL_{\leq_L X}) \subset \bfL^{\otimes d} \cap \bfC_{\leq_L X} = \clL_{\leq_L X}$. This proves that $(\clL_{\leq_L X},\clB_{\leq_L X})$ is a quasi-$\jmath$-crystal basis of $\bfC_{\leq_L X}$. Similarly, one can prove that $\clL_{<_L X} := \Span_{\Ao} \{ {}^{J(f)} C_w \mid f \in I^d_+,\ w \in {}^{J(f)}W \AND w_{J(f)}w <_L x \}$ and $\clB_{<_L X} := \{ {}^{J(f)} C_w + q\clL_{<_L X} \mid f \in I^d_+,\ w \in {}^{J(f)}W \AND w_{J(f)}w <_L x \}$ forms a quasi-$\jmath$-crystal basis of $\bfC_{<_L X}$. Setting $\clL(X) := \Span_{\Ao} \{ [{}^{J(f)} C_w]_X \mid f \in I^d_+,\ w \in {}^{J(f)}W \cap w_{J(f)}X \}$, and $\clB(X) := \{ [{}^{J(f)} C_w]_X + q\clL(X) \mid f \in I^d_+,\ w \in {}^{J(f)}W \cap w_{J(f)}X \}$, we see that $(\clL(X),\clB(X))$ is a quasi-$\jmath$-crystal basis of $\bfC^L_X$.
\end{proof}

\begin{prop}\label{irr decomp of Vd}
Suppose that $\bfV^{\otimes d} \simeq \bigoplus_{t \in T} L(\bflm_t)$ for some index set $T$ and $\bflm_t \in \Pj$. Then, for each $t \in T$, there exist highest weight vectors $v_t \in \bfV^{\otimes d}$ satisfying the following:
\begin{enumerate}
\item $L_t := \Uj v_t \simeq L(\bflm_t)$.
\item $(\clL_t,\clB_t)$ is a quasi-$\jmath$-crystal basis of $L_t$ isomorphic to $(\clL(X_t),\clB(X_t))$, where $\clL_t := \bfL^{\otimes d} \cap L_t$, $\clB_t := \bfB^{\otimes d} \cap \clL_t/q\clL_t$, and $X_t \in L(W)$ such that $\bfC^L_{X_t} \simeq L_t$.
\item $\bfL^{\otimes d} = \bigoplus_{t \in T} \clL_t$, $\bfB^{\otimes d} = \bigsqcup_{t \in T} \clB_t$.
\end{enumerate}
\end{prop}

\begin{proof}
Let us write $\bfV^{\otimes d} \simeq \bigoplus_{s \in S} L(\lm_s)$ (isomorphism of $\U$-modules) for some index set $S$ and $\lm \in P(d)$. By the ordinary crystal basis theory, there exist $\U$-highest weight vectors $v_s \in \bfV^{\otimes d}$, $s \in S$ such that $L_s := \U v_s \simeq L(\lm_s)$, $(\clL_s,\clB_s)$ is the unique crystal basis of $L_s$ containing $v_s$, where $\clL_s := \bfL^{\otimes d} \cap L_s$ and $\clB_t := \bfB^{\otimes d} \cap \clL_t/q\clL_t$, and that we have $\bfL^{\otimes d} = \bigoplus_{s \in S} \clL_s$, $\bfB^{\otimes d} = \bigsqcup_{s \in S} \clB_s$. Then, there exists a unique bilinear form $(\cdot,\cdot)$ on $\bfV^{\otimes d}$ satisfying the following;
\begin{align}
\begin{split}
&(v_s,v_s) = 1 \qu \Forall s \in S, \\
&(xm,n) = (m,\tau(x)n) \qu \Forall x \in \U,\ m,n \in \bfV^{\otimes d},
\end{split} \nonumber
\end{align}
where $\tau$ is the anti-algebra automorphism of $\bfU$ defined by
$$
\tau(E_i) = qF_iK_i\inv, \qu \tau(F_i) = q\inv K_iE_i, \qu  \tau(K_i) = K_i.
$$
It is known that $\clL_s = \{ m \in L_s \mid (m,m) \in \A \}$, and for each $u \in \clL_s$ with $u + q\clL_s \in \clB_s$, we have $(u,u) \in 1 + q\A$ for all $s \in S$.

Recall that $\bfV^{\otimes d} \simeq \bigoplus_{\substack{X \in L(W) \\ \bfC^L_X \neq 0}} L(\bflm_X)$, where $\bflm_X \in \Pj$ is such that $\bfC^L_X \simeq L(\bflm_X)$. Hence, we may identify $T$ with $\{ X \in L(W) \mid \bfC^L_X \neq 0 \}$, and equip $T$ with a partial order $\leq_L$. Now, we prove the assertions $(1)$ and $(2)$ by induction on $t$ with respect to $\leq_L$. When $t$ is minimal, the assertion follows from the fact that $\bfC^L_t$ is an irreducible submodule of $\bfV^{\otimes d}$ isomorphic to $L(\bflm_t)$ spanned by some Kazhdan-Lusztig basis elements.

Let us assume that $t \in T$ is not minimal, the assertions $(1)$ and $(2)$ hold for all $t' <_L t$, and that $\bfC_{<_L t}$ is generated by $v_{t'}$, $t' <_L t$. Consider the $\Uj$-submodule $\bfC_{\leq_L t} \subset \bfV^{\otimes d}$. Since $\bfC_{\leq_L t}$ is spanned by some Kazhdan-Lusztig basis elements, the bilinear form $(\cdot,\cdot)|_{\bfC_{\leq_L t}}$ is nondegenerate. Let $C \subset \bfC_{\leq_L t}$ denote the orthogonal complement of $\bfC_{<_L t}$ with respect to $(\cdot,\cdot)|_{\bfC_{\leq_L t}}$. Then, $C$ is spanned by vectors of the form ${}^{J(f)} C_w + u_{f,w}$ with $f \in I_+^d$, $w \in {}^{J(f)}W$, $w_{J(f)}w <_L x$ for some $x \in t$, $u_{f,w} \in \bigoplus_{t' <_L t} q\clL(\bflm_t)$. Moreover, $C$ is a $\Uj$-submodule of $\bfC_{\leq_L t}$ isomorphic to $\bfC^L_t$. Hence, there exists $v_t \in \bfV^{\otimes d}$ satisfying the assertions $(1)$ and $(2)$.

By the construction above, the assertion $(3)$ is now obvious. Thus, the proof completes.
\end{proof}

\begin{rem}\label{B_t = Left cell}\normalfont
By the construction, we have $\clB_t = \{ f w \mid f \in I^d_+,\ w \in {}^{J(f)}W \cap w_{J(f)} t \}$ for all $t \in T$.
\end{rem}

\subsection{Left cells for type $A$ and $B$}
Now, let us explicitly describe the left cells of the Weyl group $\Ss_d$ of type $A_{d-1}$ and $W_d$ of type $B_d$. We let $W_d$ act on $\Z^d$ by
$$
(i_1,\ldots,i_d) s_j = \begin{cases}
(-i_1,i_2,\ldots,i_d) \qu & \IF j = 0, \\
(i_1,\ldots,i_{j+1},i_j,\ldots,i_d) \qu & \IF j \neq 0.
\end{cases}
$$ And we regard $\Ss_d$ as the subgroup of $W_d$ generated by $s_j$, $j \neq 0$. First, let us recall Schensted bumping algorithm. Given a semistandard tableau $T$ of shape $\lm \in P$ and a colored box $\fbox{$i$}$, define a new semistandard tableau $T \leftarrow \fbox{$i$}$ as follows;
\begin{enumerate}
\item Let $j$ be the smallest number satisfying $T(j,1) \geq i$.
\item Replace $\fbox{$T(j,1)$}$ by $\fbox{$i$}$; if there is no such $j$, then put $\fbox{$i$}$ at the bottom of the first column, and stop the algorithm.
\item Repeat $(1)$-$(2)$ for the next column with the role of $\fbox{$i$}$ replaced by $\fbox{$T(j,1)$}$.
\end{enumerate}

\begin{defi}\normalfont
Given a sequence $\bfw = (w_1,\ldots,w_d)$ of elements of a totally ordered set $(L,\leq)$ and an increasing sequence $\bfr = (r_1,\ldots,r_d)$ of positive integers, we define two tableaux $P(\bfw)$ and $Q_\bfr(\bfw)$ as follows. First, $P(\bfw)$ is defined to be $( \cdots ((\fbox{$w_1$} \leftarrow \fbox{$w_2$}) \leftarrow \fbox{$w_3$}) \cdots ) \leftarrow \fbox{$w_n$}$; hence $P(\bfw)$ is a semistandard tableau in letters $L$ of some shape. Next, $Q_\bfr(\bfw)$ is defined to be the standard tableau in letters $\{ r_1,\ldots,r_d \}$ of the same shape as $P(\bfw)$ whose $(i,j)$-entry is $r_k$ if a new box is added in a position $(i,j)$ when we bump the box $\fbox{$w_k$}$ to the semistandard tableau $P(w_1,\ldots,w_{k-1})$.

We call $P(\bfw)$ the insertion tableau of $\bfw$, and $Q_\bfr(\bfw)$ the recording tableau of $\bfw$ in letters $\bfr$.
\end{defi}

\begin{defi}\normalfont
Let $x \in \Ss_d$ and $y \in W_d$ and define $x_i,y_i \in \{ 1,\ldots,d \}$ and $\vep_i \in \{ +,- \}$ by $(1,\ldots,d)x = (x_1,\ldots,x_d)$, and $(1,\ldots,d)y = (\vep_1y_1,\ldots,\vep_d,y_d)$.
\begin{enumerate}
\item Set $P(x)$ to be the insertion tableau of $(x_1,\ldots,x_d)$, and $Q(x)$ to be the recording tableau of $(x_1,\ldots,x_d)$ in letters $(1,\ldots,d)$.
\item Let $\{ 1,\ldots,d \} = \{ i_1,\ldots,i_k \} \sqcup \{ j_1,\ldots,j_l \}$ be such that $i_1 < \cdots < i_k$, $j_1 < \cdots < j_l$, $\vep_{i_n} = +$, and $\vep_{j_m} = -$. Then, set $P^-(y)$ to be the insertion tableau of $(-y_{j_1},\ldots,-y_{j_l})$, $Q^-(y)$ the recording tableau of $(-y_{j_1},\ldots,-y_{j_l})$ in letters $(j_1,\ldots,j_l)$, $P^+(y)$ the insertion tableau of $(y_{i_1},\ldots,y_{i_k})$, and $Q^+(y)$ the recording tableau of $(y_{i_1,\ldots,y_{i_k}})$ in letters $(i_1,\ldots,i_k)$.
\end{enumerate}
\end{defi}

\begin{prop}\label{Left cells of Wn}
Let $x,y \in \Ss_d$ and $z,w \in W_d$.
\begin{enumerate}
\item $x$ and $y$ are in the same left cell of $\Ss_d$ if and only if $Q(x) = Q(y)$.
\item $z$ and $w$ are in the same left cell of $W_d$ if and only if $(Q^-(z),Q^+(z)) = (Q^-(w),Q^+(w))$.
\end{enumerate}
\end{prop}

\begin{proof}
$(1)$ is found in \cite[Theorem 6.5.1]{BB05}. Let us prove $(2)$. Let $y \in W_n$, and $i_1,\ldots,i_k$, $j_1,\ldots,j_l$ be as above. Set $A^-(y)$ to be the insertion tableau of $(y_{j_1},\ldots,y_{j_l})$, $B^-(y)$ the recording tableau of $(y_{j_1},\ldots,y_{j_l})$ in letters $(j_1,\ldots,j_l)$, $A^+(y) := P^+(y)$, and $B^+(y) := Q^+(y)$. Then, by \cite[Theorem 7.7]{BI03}, $z$ and $w$ are in the same left cell of $W_n$ if and only if $(B^+(z),B^-(z)) = (B^+(w),B^-(w))$. Hence, it suffices to show that we have $(B^+(z),B^-(z)) = (B^+(w),B^-(w))$ if and only if $(Q^-(z),Q^+(z)) = (Q^-(w),Q^+(w))$. To do so, let us investigate the relation between $B^-(y)$ and $Q^-(y)$. Let $x_1 < \cdots < x_l$ be the rearrangement of $y_{j_1},\ldots,y_{j_l}$. Then, there exists a unique $\sigma_y \in \Ss_l$ such that $y_{j_m} = x_{\sigma_y(m)}$. By the definition of the recording tableaux, we have $B^-(y) = Q(\sigma_y)$ and $Q^-(y) = Q(\sigma_y w^{\Ss_l}_0)$, where $w^{\Ss_l}_0 \in \Ss_l$ denotes the longest element. Therefore, the equivalence of the two equations $B^-(z) = B^-(w)$ and $Q^-(z) = Q^-(w)$ is equivalent to the equivalence of the two equations $Q(\sigma_z) = Q(\sigma_w)$ and $Q(\sigma_z w^{\Ss_l}_0) = Q(\sigma_w w^{\Ss_l}_0)$, which follows from the fact that for each left cell $X$ of $\Ss_l$, the set $Xw^{\Ss_l}_0$ is also a left cell, and then, use $(1)$.
\end{proof}

\subsection{Irreducible decomposition of a quasi-$\jmath$-crystal bases of $\bfV^{\otimes d}$}\label{irr decomp of V^d}
Let $X \in L(W)$, and $\clB(X)$ be as in the proof of Proposition \ref{qjc of CLX}. Now, let us describe $\clB(X)$ in terms of bitableaux. For this purpose, we need a crystal version of Robinson-Schensted correspondence, which states the following: under the identification $\bfB^{\otimes d} = I^d = \{ -r,\ldots,r \}^d$, the map $u_1 \otimes \cdots \otimes u_d \mapsto (P(u_1,\ldots,u_d),Q_{1,\ldots,d}(u_1,\ldots,u_d))$ gives an isomorphism of the ordinary crystals of type $A_{2r}$.
$$
\bfB^{\otimes d} \rightarrow \bigoplus_{\lm \in P(d)} \bigoplus_{Q \in \ST(\lm)} \bigl( \SST(\lm) \times \{Q\} \bigr),
$$
where on the right-hand side, the Kashiwara operators act only on the first factor.

Next, let us see the crystal structure of $\bfB^{\otimes d}$ as a crystal of type $A_r \times A_{r-1}$. When $d=1$, we have the following irreducible decomposition; $\bfB = \bfB^- \oplus \bfB^+$, where $\bfB^- = \{ u_i \mid -r \leq i \leq 0 \}$, $\bfB^+ = \{ u_i \mid 1 \leq i \leq r \}$ with the crystal graph
$$
-r \xrightarrow[]{-(r-\hf)} \cdots \xrightarrow[]{-(2-\hf)} -1 \xrightarrow[]{-(1-\hf)} 0 \qu \qu 1 \xrightarrow[]{2-\hf} 2 \xrightarrow[]{3-\hf} \cdots \xrightarrow[]{r-\hf} r.
$$
Hence, we have
$$
\bfB^{\otimes d} = \bigoplus_{(\vep_1,\ldots,\vep_d) \in \{ +,- \}^d} (\bfB^{\vep_1} \otimes \cdots \otimes \bfB^{\vep_d}),
$$
and for each $(\vep_1,\ldots,\vep_d) \in \{ +,- \}^d$, we have
$$
\bfB^{\vep_1} \otimes \cdots \otimes \bfB^{\vep_d} \simeq (\bfB^-)^{\otimes k} \otimes (\bfB^+)^{\otimes l},
$$
where $k$ and $l$ denotes the number of $-$ and $+$ in $(\vep_1,\ldots,\vep_d)$, respectively. For each $u = u_1 \otimes \cdots \otimes u_d$, there exist unique $f_u \in I_d^+$ and $w_u \in W$ such that $f_u w_u = (u_1,\ldots,u_d)$, and $\ell(w_u)$ is maximal with this property. Therefore, the map
$$
u = u_1 \otimes \cdots \otimes u_d \mapsto \bigl( P^-(w_u), Q^-(w_u) \bigr) \otimes \bigl( P^+(w_u), Q^+(w_u) \bigr)
$$
gives an isomorphism of crystals of type $A_r \times A_{r-1}$:
\begin{align}
\begin{split}
\bfB^{\vep_1} \otimes \cdots \otimes \bfB^{\vep_d} &\simeq \left( \bigoplus_{\lm^- \in P(k)} \bigoplus_{Q^- \in \ST_{i_1,\ldots,i_k}(\lm^-)} (\clB(\lm^-) \times \{Q^-\}) \right) \\
&\hspace{2cm}\otimes \left( \bigoplus_{\lm^+ \in P(l)} \bigoplus_{Q^+ \in \ST_{j_1,\ldots,j_l}(\lm^+)} (\clB(\lm^+) \times \{Q^+\}) \right),
\end{split} \nonumber
\end{align}
where $i_1 < \cdots < i_k$ is such that $\vep_{i_n} = -$, and $j_1 < \cdots < j_k$ is such that $\vep_{j_m} = +$. Thus, we obtain the irreducible decomposition of $\bfB^{\otimes d}$ as a crystal of type $A_r \times A_{r-1}$:
$$
\bfB^{\otimes d} = \bigoplus_{\bflm \in \Pj(d)} \bigoplus_{(Q^-;Q^+) \in \ST(\bflm)} \{ u_1 \otimes \cdots \otimes u_d \mid Q(u_1 \otimes \cdots \otimes u_d) = (Q^-;Q^+) \},
$$
where $Q(u_1 \otimes \cdots \otimes u_d) := (Q^-(w_u),Q^+(w_u))$ with the notation above. For each $\bflm \in \Pj(d)$ and $(Q^-;Q^+) \in \ST(\bflm)$, set
$$
\SST(\bflm) \times \{ (Q^-;Q^+) \} := \{ u_1 \otimes \cdots \otimes u_d \mid Q(u_1 \otimes \cdots \otimes u_d) = (Q^-;Q^+) \}.
$$
Note that the Kashiwara operators act only on $\SST(\bflm)$, and therefore, it is isomorphic to $\SST(\bflm)$, which is connected (as a crystal of type $A_r \times A_{r-1}$).

\begin{prop}\label{comp of B^d as AA}
The connected components of $\bfB^{\otimes d}$ as a crystal of type $A_r \times A_{r-1}$ are $\SST(\bflm) \times \{ (Q^-;Q^+) \}$, $\bflm \in \Pj(d)$, $(Q^-;Q^+) \in \ST(\bflm)$.
\end{prop}

Recall from Proposition \ref{Left cells of Wn} that the left cells are described in terms of the recording tableaux. Thus, we obtain the following.

\begin{theo}\label{jcry decomposition of bfBd}
The connected components of $\bfB^{\otimes d}$ as a crystal of type $A_r \times A_{r-1}$ coincide with $\{ \clB_t \mid t \in T \}$, where $\clB_t$ and $T$ are as in Proposition \ref{irr decomp of Vd}. Moreover, if $\SST(\bflm) \times \{ (Q^-;Q^+) \} = \clB_t$ for some $\bflm,(Q^-;Q^+),t$, then we have $L_t \simeq L(\bflm)$, where $L_t$ is as in Proposition \ref{irr decomp of Vd}. In particular, $L(\bflm)$ is an irreducible component of $\bfV^{\otimes d}$ if and only if $\bflm \sim_\pi \bfmu$ for some $\bfmu \in \Pj(d)$.
\end{theo}

\begin{proof}
The first statement follows from Remark \ref{B_t = Left cell}, Proposition \ref{Left cells of Wn} $(2)$, and Proposition \ref{comp of B^d as AA}. Let $t \in T$, $\bflm \in \Pj(d)$ and $(Q^-;Q^+) \in \ST(\bflm)$ be such that $\SST(\bflm) \times \{ (Q^-;Q^+) \} = \clB_t$. Then, the connected component $\SST(\bflm) \times \{ (Q^-;Q^+) \}$ contains a unique element $b_\bflm$ such that
$$
\Ftil_{-(i-\hf)}b_\bflm = \Etil_{j-\hf}b_\bflm = 0 \Forall i = 1,\ldots,r \AND j = 2,\ldots,r.
$$
Under the isomorphism $\SST(\bflm) \times \{ (Q^-;Q^+) \} \simeq \SST(\bflm)$ of crystals of type $A_r \times A_{r-1}$, the element $b_\bflm$ is identified with $(T_\bflm^-;T_\bflm^+) \in \SST(\bflm)$ defined by
$$
T_\bflm^-(i,j) = -(i-1), \qu T_\bflm^+(i,j) = i,
$$
where $T(i,j)$ denotes the entry of the box of $T$ which lies in the $i$-th row and $j$-th column. Since $\wtj(b_\bflm)$ is maximal among $\wtj(\clB_t)$, we have $b_t := v_t + q\clL_t = b_\bflm$, where $v_t \in L_t$ is the $\Uj$-highest weight vector. Suppose that $L_t = L(\bfmu)$ for some $\bfmu \in \Pj$. Then, we have $\vphi_1(b_t) = \bfmu_0-\bfmu_{-1}$. Recall from the proof of Lemma \ref{proof of (*)} that $\tau_2\inv(v_t)$ is a $\Uj_1$-highest weight vector of highest weight $(\bfmu_0,\bfmu_{-2};\bfmu_2)$. This implies that $\vphi_1(\ftil_2^{\max)}b_t) = \bfmu_0-\bfmu_{-2}$. Also, from the proof of Theorem \ref{Classification}, $T_i\inv(v_t)$ is a $\Uj_2$-highest weight vector of highest weight $(\bfmu_0,\bfmu_{-(i-1)},\bfmu_{-i};\bfmu_{i-1},\bfmu_i)$ for all $i \geq 3$. Hence, we have $\vphi_1(\ftil_2^{\max} (\ftil_3^{\max} \ftil_2^{\max}) (\ftil_4^{\max}\ftil_3^{\max}) \cdots (\ftil_i^{\max}\ftil_{i-1}^{\max})b_t) = \bfmu_0-\bfmu_{-i}$. Applying the same argument to $T_\bflm$, we obtain the following:
$$
\wtj(b_t) = \wtj(T_\bflm), \qu \bfmu_0-\bfmu_{-i} = \bflm_0 - \bflm_{-i} \Forall i = 1,\ldots,r.
$$
Solving this system of equations, we conclude that $\bfmu \sim_\pi \bflm$, and hence, $L_t \simeq L(\bflm)$.
\end{proof}

\begin{cor}
The assignment $(\bfa,\bfb) \mapsto [L(\bfa;\bfb)]$, where $[L(\bfa;\bfb)]$ denotes the isomorphism class of $L(\bfa;\bfb)$, gives a bijection from $\{ (\bfa,\bfb) \in \Z^r \times \Z_{\geq 0}^r \mid a_i \geq b_i,\ i \in \Ij \setminus\{1\} \}$ to the set of isomorphism classes of irreducible $\Uj$-modules in $\Oj$.
\end{cor}

\begin{cor}
The isomorphism classes of $\Oj$ are parametrized by $\Pj/\sim_\pi$.
\end{cor}

\begin{cor}\label{quasi-j-crystal on SST}
Let $\bflm \in \Pj(d)$. Then, $\SST(\bflm)$ is equipped with a unique quasi-$\jmath$-crystal structure for which the map $R:\SST(\bflm) \rightarrow \bfB^{\otimes d}$ defined in subsection \ref{jcrystal} is compatible all of structure maps for quasi-$\jmath$-crystal bases. Also, the irreducible highest weight $\Uj$-module $L(\bflm)$ with highest weight $\bflm$ has a quasi-$\jmath$-crystal basis which is identical to $\SST(\bflm)$. Moreover, $\SST(\bflm)$ is equipped with a quasi-$\jmath$-crystal basis structure given below.
\begin{enumerate}
\item For each $T \in \SST(\bflm)$, $\wtj(T) = \sum_{i = -r}^r (\text{the number of occurrence of $i$ in $T$}) \epsilon_i$.
\item For each $T \in \SST(\bflm)$ and $i \in \Ij$, we have
\begin{align}
\begin{split}
\ftil_1(T) &= \Etil_{-\hf}(T), \\
\etil_1(T) &= \Ftil_{-\hf}(T), \\
\ftil_i(T) &= \begin{cases}
\Etil_{-(i - \hf)}(T) \qu & \IF \varepsilon_{i-\hf}(T) < \varepsilon_{-(i-\hf)}(T), \\
\Ftil_{i-\hf}(T) \qu & \IF \varepsilon_{i-\hf}(T) \geq \varepsilon_{-(i-\hf)}(T),
\end{cases} \nonumber\\
\etil_i(T) &= \begin{cases}
\Ftil_{-(i - \hf)}(T) \qu & \IF \varepsilon_{i-\hf}(T) \leq \varepsilon_{-(i-\hf)}(T), \\
\Etil_{i-\hf}(T) \qu & \IF \varepsilon_{i-\hf}(T) > \varepsilon_{-(i-\hf)}(T),
\end{cases} \nonumber
\end{split} \nonumber
\end{align}
here, we identify $T \in \SST(\bflm)$ with $R(T) \in \bfB^{\otimes d}$, on which the actions of $\Etil_i, \Ftil_i$, $i \in \bbI$ make sense.
\end{enumerate}
\end{cor}

\subsection{$\jmath$-crystal bases}
In general, a quasi-$\jmath$-crystal graph of an irreducible $\Uj$-module is neither connected nor unique. In this subsection, we introduce the notion of $\jmath$-crystal bases as quasi-$\jmath$-crystal bases satisfying some additional conditions. And we prove the existence and uniqueness theorem for $\jmath$-crystal bases, and that they are connected.

Let $\bflm \in \Pj(d)$, and take a left cell $X \in L(W)$ satisfying $\bfC^L_X \simeq L(\bflm)$. Recall that $\bfC^L_X$ has a basis $\{ [{}^{J(f)} C_w]_X \mid f \in I_d^+,\ w \in X \cap w_{J(f)}{}^{J(f)} W \}$, and it is in one-to-one correspondence with $\SST(\bflm)$; we denote by $b_T$ the basis element corresponding to $T \in \SST(\bflm)$. For each $i \in \{ 2,\ldots,r \}$, we define linear endomorphisms $\etil_{i'}$ and $\ftil_{i'}$ on $L(\bflm)$ by
\begin{align}
\begin{split}
\etil_{i'}(b_T) &= \begin{cases}
b_{\Etil_{i-\hf} T} \qu & \IF \etil_j T = 0 \Forall j = 1,\ldots,i-1 \AND \Etil_{j-\hf} T = 0 \Forall j = 2,\ldots,i-1, \\
0 \qu & \OW,
\end{cases} \\
\ftil_{i'}(b_T) &= \begin{cases}
b_{T'} \qu & \IF \etil_{i'} b_{T'} = b_T, \\
0 \qu & \OW.
\end{cases}
\end{split} \nonumber
\end{align}

Note that the condition $\etil_j T = 0 \Forall j =1,\ldots,i-1 \AND \Etil_{j-\hf} T = 0 \Forall j = 2,\ldots,i-1$ is equivalent to $\etil_j T = 0 \Forall j =1,\ldots,i-1 \AND \etil_{j'} T = 0 \Forall j = 2, \ldots, i-1$.

Let $X' \in L(W)$ be such that $C^L_{X'} \simeq C^L_X$. Since the linear map $[C_{w'}]_{X'} \mapsto [C_w]_X$, $w' \in X'$, $w \in X$ with $P^\pm(w') = P^\pm(w)$ gives an isomorphism $C^L_{X'} \rightarrow C^L_X$ of $\Hc$-modules, the definition of $\etil_{i'}$ and $\ftil_{i'}$ are independent of the choice of $X$ as long as we have $\bfC^L_X \simeq L(\bflm)$.

Also, we define linear endomorphisms $\etil_{i'}$ and $\ftil_{i'}$, $i \in \{ 2,\ldots,r \}$ on each $\Uj$-module in $\Oj$ by the complete reducibility.

\begin{rem}\normalfont
In a future work, we will give more intrinsic definitions of $\etil_{i'}$ and $\ftil_{i'}$.
\end{rem}

\begin{defi}\normalfont
Let $M \in \Oj$ be a $\Uj$-module with a quasi-$\jmath$-crystal basis $(\clL,\clB)$. We say that $(\clL,\clB)$ is a $\jmath$-crystal basis if it satisfies the following:
\begin{enumerate}
\item[($\jmath$C 1)] $\clL$ is preserved by the operators $\etil_{i'}$ and $\ftil_{i'}$, $i \in \{ 2,3,\ldots,r \}$.
\item[($\jmath$C 2)] We have $\etil_{i'}(\clB) \subset \clB \sqcup \{0\}$ and $\ftil_{i'}(\clB) \subset \clB \sqcup \{0\}$ for all $i' \in \{ 2,3,\ldots,r \}$.
\end{enumerate}
\end{defi}

Let $\bflm \in \Pj$, and $v \in L(\bflm)$ be a highest weight vector. Set
\begin{align}
\begin{split}
&\clL(\bflm) := \Span_{\Ao} \{ \ftil_{i_1} \cdots \ftil_{i_l}v \mid l \in \Z_{\geq 0},\ i_1,\ldots,i_l \in \Ij \sqcup \{ 2',\ldots,r' \} \}, \\
&\clB(\bflm) := \{ \ftil_{i_1} \cdots \ftil_{i_l}v + q\clL(\bflm) \mid l \in \Z_{\geq 0},\ i_1,\ldots,i_l \in \Ij \sqcup \{ 2',\ldots,r' \} \} \setminus \{0\}.
\end{split} \nonumber
\end{align}

\begin{theo}
Let $\bflm \in \Pj$. Then, $(\clL(\bflm),\clB(\bflm))$ is a unique $\jmath$-crystal basis of $L(\bflm)$.
\end{theo}

\begin{proof}
Let $X \in L(W)$ be such that $\bfC^L_X \simeq L(\bflm)$. By the definition of $\etil_{i'}$ and $\ftil_{i'}$, it is clear that they preserve $\clL(X)$, and induce maps $\clB(X) \rightarrow \clB(X) \sqcup \{0\}$. Therefore, $(\clL(X),\clB(X))$ is a $\jmath$-crystal basis of $\bfC^L_X$.

Next, we show that $\clB(X)$ is connected as a $\jmath$-crystal basis. To do so, it is convenient to identify $\clB(X)$ with $\SST(\bflm)$. Let $T_0$ denote the unique highest weight vector of $\SST(\bflm)$. For $T \in \SST(\bflm)$, set
$$
d(T) := \sum_{i,j} (|T(i,j)| - |T_0(i,j)|),
$$
where $T(i,j)$ denotes the integer in the $(i,j)$-box of $T$. Then, we have $d(T) \geq 0$, and $d(T) = 0$ if and only if $T = T_0$. We prove that $T$ is connected to $T_0$ by induction on $d(T)$. When $d(T) = 0$, we have $T = T_0$, and there is nothing to prove. When $d(T) > 0$, there exists $i \in \Ij$ such that $\etil_i T \neq 0$ or $\etil_{i'} T \neq 0$. Hence, $T$ is connected to either $\etil_i T$ or $\etil_{i'} T$. Since $d(\etil_i T) = d(\etil_{i'} T) = d(T) - 1$, the induction proceeds. This proves that $\SST(\bflm)$ is connected. Moreover, we obtain
$$
\SST(\bflm) = \{ \ftil_{i_1} \cdots \ftil_{i_l} T_0 \mid l \in \Z_{\geq 0},\ i_1,\ldots,i_l \in \Ij \sqcup \{ 2',\ldots,r' \} \} \setminus \{0\}.
$$

Let $v \in \bfC^L_X$ be the unique highest weight vector satisfying $v + q\clL(X) = T_0$. By above argument, for each $T \in \SST(\bflm)$, there exist $i_1,\ldots,i_l \in \Ij \sqcup \{ 2',\ldots, r' \}$ such that $T = \ftil_{i_1} \cdots \ftil_{i_l} T_0$. This implies that $\ftil_{i_1} \cdots \ftil_{i_l} v + q\clL(X) = T$, and therefore, $\clL(X)$ is spanned by such vectors. Thus, the proof completes.
\end{proof}

\begin{cor}
$(\bfL^{\otimes d}, \bfB^{\otimes d})$ is a $\jmath$-crystal basis of $\bfV^{\otimes d}$. Under the identification $\bfB^{\otimes d} = I^d$, the actions of $\etil_i,\ftil_i$, $i \in \Ij$ are described by Corollary \ref{qjcry for words}, and those of $\etil_{i'},\ftil_{i'}$, $i \in \Ij \setminus \{1\}$ are given as follows: $\etil_{i'} \bfs$ is obtained from $\bfs$ by replacing the rightmost $i$ in $\bfs_i$ with $i-1$ if $\etil_j \bfs = 0$ for all $j = 1,\ldots,i-1$ and $\etil_{j'} \bfs = 0$ for all $j = 2,\ldots,i-1$, and is $0$ otherwise. Finally, 
$$
\ftil_{i'} \bfs = \begin{cases}
\bfs' \qu  \IF \etil_{i'} \bfs' = \bfs, \\
0 \qu  \OW.
\end{cases}
$$
\end{cor}

Now, the existence and uniqueness theorem for $\jmath$-crystal basis can be proved in the same way as the ordinary crystal basis theory.

\begin{theo}
Let $M \in \Oj$ be a $\Uj$-module. Then, $M$ has a $\jmath$-crystal basis $(\clL,\clB)$. If $M \simeq \bigoplus_{\bflm \in \Pj_\pi} L(\bflm)^{\oplus m_\bflm}$ for some $m_\bflm \in \Z_{\geq 0}$, then there exists an isomorphism $M \rightarrow \bigoplus_{\bflm \in \Pj_\pi} L(\bflm)^{\oplus m_\bflm}$ inducing an isomorphism $(\clL,\clB) \rightarrow (\bigoplus_{\bflm \in \Pj_\pi} \clL(\bflm)^{\oplus m_\bflm}, \bigoplus_{\bflm \in \Pj_\pi} \clB(\bflm)^{\oplus m_\bflm})$.
\end{theo}

\section{Applications}\label{Applications}
In this section, we consider how a given $\Uj$-module decomposes into irreducible modules. By the existence and uniqueness of $\jmath$-crystal bases, together with the connectedness (with a single source) of the $\jmath$-crystal basis of an irreducible $\Uj$-module, the problem is reduced to determining the highest weight vectors in the $\jmath$-crystal basis of a given module. We will frequently use results in \cite{Kw09}.

\subsection{Irreducible decomposition of $\bfV^{\otimes d}$}
Set $\IIj := \Ij \sqcup \{ 2',\ldots,r' \}$. The connected components of $\bfB^{\otimes d}$ are in one-to-one correspondence with $\bfs \in \bfB^{\otimes d}$ such that $\etil_i(\bfs) = 0$ for all $i \in \IIj$. Such $\bfs$'s are characterized as follows.

\begin{prop}\label{DoubleYamanouchi}
Let $\bfs = (s_1,\ldots,s_N) \in \bfB^{\otimes d}$. For each $-r \leq j \leq r$ and $1 \leq n \leq d$, set $c_j^{\leq n}(\bfs) := \sharp \{ 1 \leq m \leq n \mid s_m = j \}$ and $c_j^{\geq n}(\bfs) := \sharp \{ n \leq m \leq d \mid s_m = j \}$. Then the following are equivalent:
\begin{enumerate}
\item $\etil_i(\bfs) = 0$ for all $i \in \IIj$.
\item $c_0^{\geq n}(\bfs) \geq c_{-1}^{\geq n}(\bfs)$, $c_{-(j-1)}^{\geq n}(\bfs) \geq c_{-j}^{\geq n}(\bfs)$, and $c_{j-1}^{\leq n}(\bfs) \geq c_j^{\leq n}(\bfs)$ for all $1 \leq n \leq d$ and $j \in \Ij \setminus \{1\}$.
\end{enumerate}
\end{prop}

\begin{proof}
This follows easily from the $\jmath$-crystal structure of $\bfB^{\otimes d}$.
\end{proof}

We call an element $\bfs \in \bfB^{\otimes d}$ satisfying condition $(2)$ of Proposition \ref{DoubleYamanouchi} a Yamanouchi biword, since $\bfs$ is a Yamanouchi word when we read only letters $1,2,\ldots,r$ and so is $\bfs^{\rev}$ when we read only letters $0,-1,\ldots,-r$ and then ignore negative signs.

\begin{rem}\label{RScor}\normalfont
What we call a Yamanouchi word is called a lattice permutation in \cite{Kw09}. For a partition $\lambda \in P$, we denote by $\Yam(\lambda)$ the set of Yamanouchi words in letters $1,\ldots,2r+1$ of shape $\lambda$, that is, the number of appearances of $i$ in the word equals $\lambda_i$ for all $i \in \{ 1,\ldots,2r+1 \}$. By the Robinson-Schensted correspondence, we have $\sharp \Yam(\lambda) = \sharp \ST(\lambda)$.
\end{rem}

\begin{prop}\label{FromGraphToHighestWeight}
Let $\bfs \in \bfB^{\otimes d}$ be a Yamanouchi biword. Then, the connected component of $\bfB^{\otimes d}$ containing $\bfs$ is isomorphic to $\SST(\bflm)$, where $\bflm$ is a bipartiotion given by
$$
\bflm_i := \sharp \{ m \mid s_m = i \}, \qu -r \leq i \leq r.
$$
\end{prop}

\begin{proof}
By the complete reducibility of $\bfV^{\otimes d}$, the connected component of $\bfB^{\otimes d}$ containing $\bfs$ is isomorphic to $\SST(\bflm)$ for some $\bflm \in \Pj(d)$. Since $\etil_i(\bfs) = 0$ for all $i \in \IIj$, we may identify $\bfs$ with $R(T_{\bflm})$, where $T_{\bflm} \in \SST(\bflm)$ denotes the unique highest weight vector. By comparing $\bfs$ and $R(T_\bflm)$, we obtain $\bflm_i = \sharp \{ m \mid s_m = i \}$ for all $-r \leq i \leq r$. This proves the proposition.
\end{proof}

We denote by $\Yam(\bflm)$ the set of Yamanouchi biwords $\bfs$ in $\bfB^{\otimes |\bflm|}$ satisfying $\bflm_i := \sharp \{ m \mid s_m = 0 \}$, and call each element in $\Yam(\bflm)$ a Yamanouchi biword of shape $\bflm$. By Theorem \ref{jcry decomposition of bfBd}, we know $\sharp \Yam(\bflm) = \sharp \ST(\bflm)$. However, one can prove this equation directly as follows.

Let $\bflm \in \Pj$ and $(T^-;T^+) \in \ST(\bflm)$. We write $L^- = \{ p_1,\ldots,p_{|\bflm^-|} \}$ and $L^+ = \{ q_1,\ldots,q_{|\bflm^+|} \}$, with $p_1 < \cdots < p_{|\bflm^-|}$, $q_1 < \cdots < q_{|\bflm^+|}$. Let $T$ denote the standard Young tableau of shape $\bflm^-$ obtained from $T^-$ by replacing each $p_i$ with $i$. Define $T'$ similarly by using $T^+$. Then, the map $\ST(\bflm) \rightarrow \ST(\bflm^-) \times \ST(\bflm^+) \times \{ (q_1,\ldots,q_{|\bflm^+|}) \mid 1 \leq q_1 < \cdots < q_{|\bflm^+|} \leq |\bflm| \}$ defined by $(T^-,T^+) \mapsto (T,T',(q_1,\ldots,q_{|\bflm^+|}))$ is a bijection.

\begin{theo}
Let $\bflm \in \Pj(d)$. Then, the multiplicity of the irreducible component of $\bfV^{\otimes d}$ isomorphic to $L(\bflm)$ is equal to $\sharp \ST(\bflm)$. Namely, we have an isomorphism
\begin{align}
\bfV^{\otimes d} \simeq \bigoplus_{\bflm \in \Pj(d)} L(\bflm)^{\oplus \sharp \ST(\bflm)} \nonumber
\end{align}
of $\Uj$-modules.
\end{theo}

\begin{proof}
By Proposition \ref{FromGraphToHighestWeight}, the multiplicity of the irreducible component of $\bfV^{\otimes d}$ isomorphic to $L(\bflm)$ is equal to $\sharp \Yam(\bflm)$. Here, the set $\Yam(\bflm)$ is in one-to-one correspondence with $\Yam(\bflm^-) \times \Yam(\bflm^+) \times \{ (q_1, \ldots, q_{|\bflm^+|}) \mid 1 \leq q_1 \leq \cdots \leq q_{|\bflm^+|} \leq d \}$ under the assignment
\begin{align}
\bfs = ( s_1,\ldots,s_d) \mapsto \left((|s_{p_{|\bflm^-|}}|,\ldots,|s_{p_1}|), (s_{q_1},\ldots,s_{q_{|\bflm^+|}}), (q_1,\ldots,q_{|\bflm^+|}) \right), \nonumber
\end{align}
where $s_{p_1},\ldots,s_{p_{|\bflm^-|}} \leq 0$ with $p_1 < \cdots < p_{|\bflm^-|}$, and $s_{q_1},\ldots,s_{q_{|\bflm^+|}} \geq 1$ with $q_1 < \cdots < q_{|\bflm^+|}$. From this and the bijection $\ST(\bflm) \rightarrow \ST(\bflm^-) \times \ST(\bflm^+) \times \{ (q_1,\ldots,q_{|\bflm^+|}) \mid 1 \leq q_1 \leq \cdots \leq q_{|\bflm^+|} \leq |\bflm| \}$, we obtain
\begin{align}
\begin{split}
\sharp \Yam(\bflm) 
&= \sharp \Yam(\bflm^-) \cdot \sharp \Yam(\bflm^+) \cdot \binom{d}{|\bflm^+|} \\
&= \sharp \ST(\bflm^-) \cdot \sharp \ST(\bflm^+) \cdot \binom{d}{|\bflm^+|}  \qqu (\text{By Remark } \ref{RScor})\\
&= \sharp \ST(\bflm),
\end{split} \nonumber
\end{align}
as desired. This proves the theorem.
\end{proof}

By the double centralizer property between $\Uj$ and $\bfH$, we have an irreducible decomposition
\begin{align}\label{RS}
\bfV^{\otimes d} \simeq \bigoplus_{\bflm \in \Pj(d)} L(\bflm) \boxtimes V(\bflm) 
\end{align}
as a $\Uj$-$\bfH$-bimodule, where $V(\bflm)$'s are pairwise nonisomorphic irreducible $\bfH$-modules. According to \cite{H74}, the irreducible $\bfH$-modules are classified by the bipartitions of size $d$. Hoefsmit constructed the irreducible modules by giving the representation matrices for the generators of $\bfH$ explicitly. Later, Dipper and James \cite{DJ92} realized the irreducible $\bfH$-modules $S^{\lm,\mu}$ as ideals of $\bfH$, where $\lm,\mu$ are partitions with $|\lm| + |\mu| = d$. In Appendix \ref{appxB}, we prove that $V(\bflm) \simeq S^{\bflm^-,\bflm^+}$.

\subsection{Littlewood-Richardson rule for $\Uj$}
In this subsection, we consider the irreducible decomposition of $L(\bflm) \otimes L(\mu)$ for $\bflm \in \Pj$, $\mu \in P$. In terms of $\jmath$-crystal bases, we will determine the Yamanouchi biwords in $\clB(\bflm) \otimes \clB(\mu) \subset \bfB^{\otimes |\bflm|+|\mu|}$; here $\clB(\mu) \simeq \SST(\mu)$ denotes the crystal basis of $L(\mu)$ embedded in $\bfB^{\otimes |\mu|}$ by the Middle-Eastern reading. Let $\LR^{\bfnu}_{\bflm,\mu}(r)$ denote the multiplicity of $L(\bfnu)$ in $L(\bflm) \otimes L(\mu)$; clearly, it is equal to the number of Yamanouchi biwords in $\clB(\bflm) \otimes \clB(\mu)$ of shape $\bfnu$.

Let us briefly recall the Littlewood-Richardson rule for ordinary crystal bases in type $A_{n-1}$. Let $\Par_n$ denote the set of partitions of length $n$. For $\mu,\eta,\xi \in \Par_n$, let $\LR^\mu_{\eta,\xi}(n)$ denote the multiplicity of $L(\mu)$ in $L(\eta) \otimes L(\xi)$ as a $U_q(\mathfrak{sl}_n)$-module. A semistandard tableau $T$ of shape $\mu/\eta$ is called a Littlewood-Richardson tableau of shape $\mu/\eta$ with content $\xi$ if $T$ contains $\xi_i$ $i$'s, and if $ME(T)$ is a Yamanouchi word (\cite{Kw09}). Hence $\LR^\mu_{\eta,\xi}(n)$ equals the number of Littlewood-Richardson tableaux of shape $\mu/\eta$ with content $\xi$ in $n$ letters. Also, it is known that the multiplicity of $\clB(\xi)$ in $\clB(\mu/\eta)$ is equal to $\LR^\mu_{\eta,\xi}(n)$.

\begin{theo}\label{LRrule}
Let $\bflm,\bfnu \in \Pj$, $\mu \in P$. Then, we have
\begin{align}
\LR^{\bfnu}_{\bflm,\mu}(r) = \sum_{\substack{\eta \in \Par_{r+1} \\ \eta \subset \mu}} \sum_{\xi \in \Par_r(|\mu/\eta|)} \LR^{\bfnu^-}_{\eta,\bflm^-}(r+1) \LR^\mu_{\eta,\xi}(2r+1) \LR^{\bfnu^+}_{\bflm^+,\xi}(r). \nonumber
\end{align}
\end{theo}

\begin{proof}
Let $(T_1;T_2) \in \clB(\bflm)$ and $T \in \clB(\mu)$. If we read only letters $\leq 0$ in $T$, then it is also a semistandard tableau $T'$ of shape, say $\eta \subset \mu$. Since there are $r+1$ kinds of letters $\leq 0$, we may assume that $\ell(\eta) = r+1$. Suppose that $(T_1;T_2) \otimes T$ is a Yamanouchi biword of shape $\bfnu$. By the definition of Yamanouchi biwords, $(\ME(T_2), \ME(T/T'))$ is a Yamanouchi word of shape $\bfnu^+$ in letters $1,\ldots,r$, and $(\EM(T_1),\ME(T'))^{\rev} = (\EM(T'),\ME(T_1))$ is a Yamanouchi word of shape $\bfnu^-$ in letters $0,1,\ldots,r$ if we ignore negative signs. In addition, by Proposition \ref{TensorRule}, we have $\Ftil_{-(i-\hf)}(T') = 0$ for all $i \in \Ij$. This implies that $\EM(T')$ is a Yamanouchi word of shape $\eta$ if we ignore negative signs, and that $T'$ is determined uniquely by $\eta$ and this condition; hence, we write $T_\eta = T'$. With this notation, for an arbitrary partition $\eta \subset \mu$ of length $r+1$, let $Y(\eta)$ be the number of $(T_2,T)$ such that $T' = T_\eta$ and $(\ME(T_2), \ME(T/T'))$ is a Yamanouchi word of shape $\bfnu^+$ in letters $1,\ldots,r$, and $Z(\eta)$ the number of $T_1$ such that $(\EM(T_\eta),\ME(T_1))$ is a Yamanouchi word of shape $\bfnu^-$ in letters $0,1,\ldots,r$ if we ignore negative signs. Then, by the above, we obtain
\begin{align}
\LR^{\bfnu}_{\bflm,\mu}(r) = \sum_{\substack{\eta \in \Par_{r+1} \\ \eta \subset \mu}} Y(\eta) \cdot Z(\eta); \nonumber
\end{align}
here, $Y(\eta)$ is equal to the cardinality of $\Yam(\bfnu^+) \cap (\clB(\bflm^+) \otimes \clB(\mu/\eta))$ , where $\clB(\mu/\eta)$ denotes the set of semistandard tableaux of shape $\mu/\eta$ in letters $1,\ldots,r$. Therefore, we see that $Y(\eta) = \sum_{\xi \in \Par_r(|\mu/\eta|)} \LR^{\bfnu^+}_{\bflm^+,\xi}(r) \cdot \LR^{\mu}_{\eta,\xi}(2r+1)$ by the Littlewood-Richardson rule for ordinary crystal bases in type $A$.

In order to compute $Z(\eta)$, let us count the number $Z'(\eta)$ of Yamanouchi words in $\clB(\eta) \otimes \clB(\bflm^-)$ of shape $\bfnu^-$ in letters $0,1,\ldots,r$. By the tensor product rule for ordinary crystal bases, if $T_3 \otimes T_4 \in \clB(\eta) \otimes \clB(\bflm^-)$ is a Yamanouchi word, then so is $T_3$. Since $\EM(T_\eta)$ is a Yamanouchi word of shape $\eta$ in letters $0,1,\ldots,r$ if we ignore negative signs, $Z(\eta)$ is equal to $Z'(\eta)$, which, in turn, equals $\LR^{\bfnu^-}_{\eta,\bflm^-}(r+1)$ by the Littlewood-Richardson rule for ordinary crystal basis in type $A$.

Summarizing, we conclude that
\begin{align}
\LR^{\bfnu}_{\bflm,\mu}(r) = \sum_{\substack{\eta \in \Par_{r+1} \\ \eta \subset \mu}} \sum_{\xi \in \Par_r(|\mu/\eta|)} \LR^{\bfnu^+}_{\bflm^+,\xi}(r) \cdot \LR^{\mu}_{\eta,\xi}(2r+1) \cdot \LR^{\bfnu^-}_{\eta,\bflm^-}(r+1), \nonumber
\end{align}
as desired. This proves the theorem.
\end{proof}

In particular, if we take $\bflm$ to be $(\emptyset;\emptyset)$, then the tensor product module $L(\emptyset;\emptyset) \otimes L(\mu)$ is just $L(\mu)$ regarded as a $\Uj$-module. Hence, Theorem \ref{LRrule} gives the branching rule for $\U$-modules restricted to $\Uj$:

\begin{cor}
The multiplicity of $L(\bfnu)$ in $L(\mu)$ is equal to $\LR^{\mu}_{\bfnu^-,\bfnu^+}(r)$.
\end{cor}

\begin{proof}
By Theorem \ref{LRrule}, we have
\begin{align}
\LR^{\bfnu}_{(\emptyset;\emptyset),\mu}(r) = \sum_{\eta,\xi} \LR^{\bfnu^-}_{\eta,\emptyset}(r+1) \LR^{\mu}_{\eta,\xi}(r) \LR^{\bfnu^+}_{\emptyset,\xi}(r). \nonumber
\end{align}
However, $\LR^{\bfnu^-}_{\eta,\emptyset}(r+1) = \delta_{\eta,\bfnu^-}$ and $\LR^{\bfnu^+}_{\emptyset,\xi}(r) = \delta_{\xi,\bfnu^+}$. This proves the corollary.
\end{proof}

\begin{rem}\normalfont
The multiplicity of $L(\bfnu)$ in $L(\mu)$ can be computed directly by counting the number of the Yamanouchi biword of shape $\bfnu$ in $\clB(\mu)$.
\end{rem}

\appendix
\section{Irreducible decomposition of $\bfV^{\otimes d}$ as a $\Uj$-$\bfH$-bimodule}\label{appxB}
\subsection{The action of $\bfH$ on $\bfV^{\otimes d}$}
We denote by $\Ss_d$ and $\Ss_{a,d-a}$ the subgroup of $W_d$ generated by $s_i$, $i \neq 0$, and $s_i$, $i \neq 0,a$, respectively. Let $\bfH(\Ss_d)$ and $\bfH(\Ss_{a,d-a})$ be the subalgebra of $\bfH$ generated by $H_i$, $i \neq 0$, and $H_i$, $i \neq 0,a$, respectively.

Following \cite{BWW16}, let $\bfH$ act on $\bfV^{\otimes d}$. For a map $f : \{ 1,\ldots,d \} \rightarrow \{ -r,-r+1,\ldots,r \}$, we set $M_f := u_{f(1)} \otimes \cdots \otimes u_{f(d)} \in \bfV^{\otimes d}$. The Weyl group $W_d$ acts on the set of maps from $\{ 1,\ldots,d \}$ to $\{ -r,-r+1,\ldots,r \}$ by:
\begin{align}
\begin{split}
(f \cdot s_j)(i) &= \begin{cases}
f(j+1) & \IF i=j, \\
f(j) & \IF i=j+1, \\
f(i) & \OW,
\end{cases} \\
(f \cdot s_0)(i) &= \begin{cases}
-f(1) & \IF i=1, \\
f(i) & \OW.
\end{cases}
\end{split} \nonumber
\end{align}
Then the Hecke algebra $\bfH$ acts on $\bfV^{\otimes d}$ by:
\begin{align}
\begin{split}
M_f \cdot H_j &= \begin{cases}
q\inv M_f & \IF f(i) = f(i+1), \\
M_{f \cdot s_j} & \IF f(i) < f(i+1), \\
M_{f \cdot s_j} + (q\inv - q)M_f & \IF f(i) > f(i+1),
\end{cases} \\
M_f \cdot H_0 &= \begin{cases}
p\inv M_f & \IF f(1)=0, \\
M_{f \cdot s_0} & \IF f(1)>0, \\
M_{f \cdot s_0} + (p\inv - p)M_f & \IF f(1)<0.
\end{cases}
\end{split} \nonumber
\end{align}

\subsection{Irreducible $\bfH(\Ss_d)$-modules}\label{B.2}
Let us recall from \cite{G86} how to construct the irreducible $\bfH(\Ss_d)$-modules. Note that our convention differs from that in \cite{G86}; because of this, we construct right $\bfH(\Ss_d)$-modules, while Gyoja treated left $\bfH(\Ss_d)$-modules. Let $\lambda = (\lambda_1 \geq \lambda_2 \geq \cdots \geq \lambda_k \geq 0)$ be a partition of $d$ and $\lambda'$ its transposed partition. Let $T_+(\lambda)$ be the standard tableau of shape $\lambda$ defined by $T_+(i,j) = \lambda_1 + \cdots \lambda_{i-1} + j$, and $T_-(\lambda)$ the standard tableau of shape $\lambda$ defined by $T_-(i,j) = \lambda'_1 + \cdots \lambda'_{j-1} + i$. Also, let $I_+$ be the set of those $s_i$ which preserves each row of $T_+$, $I_-$ the set of those $s_i$ which preserves each column of $T_-$, $\Ss_\pm$ the subgroup of $\Ss_d$ generated by $I_\pm$, and $\bfH(\Ss_\pm)$ the subalgebra of $\bfH(\Ss_d)$ corresponding to $\Ss_\pm$. Set
\begin{align}
e_+ := \sum_{x \in \Ss_+} q^{-\ell(x)}H_x, \qqu e_- := \sum_{y \in \Ss_-} (-q)^{\ell(y)}H_y. \nonumber
\end{align}

Let $[+,-] \in \Ss_d$ be the unique element such that $T_+ \cdot [+,-] = T_-$. Then, for each $x \in \Ss_+$ and $y \in \Ss_-$, one has $\ell(x[+,-]y) = \ell(x) + \ell([+,-]) + \ell(y)$. By \cite[Section 2]{G86}, the following holds.

\begin{theo}[\cite{G86}]
The right ideal $S^\lambda$ of $\bfH(\Ss_d)$ generated by $e_+ H_{[+,-]} e_-$ is an irreducible $\bfH(\Ss_d)$-module. Moreover, the set $\{ S^\lambda \mid \lambda \vdash d \}$ provides a complete list of nonisomorphic irreducible $\bfH(\Ss_d)$-modules.
\end{theo}

By this theorem, we can realize each $S^\lambda$, $\lambda \vdash d$, as a submodule of $\bfV^{\otimes d}$. We define a map $f_\lambda$ by: $f_\lambda(i) = j$ if $\lambda_{j-1} < i \leq \lambda_{j}$. It is easy to verify that the $\bfH(\Ss_d)$-submodule generated by $M_{f_\lambda}$ is isomorphic to $e_+\bfH(\Ss_d)$. Therefore, the $\bfH(\Ss_d)$-submodule generated by $M_{\lambda,+} := M_{f_\lambda} \cdot (H_{[+,-]} e_-)$ is isomorphic to $S^\lambda$. Since $\ell(x[+,-]y) = \ell(x) + \ell([+,-]) + \ell(y)$ for all $x \in \Ss_+$ and $y \in \Ss_-$, we see that
\begin{align}
M_{\lambda,+} = \sum_{y \in \Ss_-} (-q)^{\ell(y)} M_{f_\lambda \cdot [+,-]y}. \nonumber
\end{align}
Also, by the definitions of $f_\lambda$ and $[+,-]$, it follows that
\begin{align}
M_{f_\lambda \cdot [+,-]} &= (u_1 \otimes u_2 \otimes \cdots \otimes u_{\lambda'_1}) \otimes (u_1 \otimes u_2 \otimes \cdots \otimes u_{\lambda'_2}) \otimes \cdots \otimes (u_1 \otimes u_2 \otimes \cdots \otimes u_{\lambda'_k}). \nonumber
\end{align}
These imply that $M_{\lambda,+} \in M_{f_\lambda \cdot [+,-]} + q\bfL^{\otimes d}$.

By the quantum Schur-Weyl duality of type $A$, the irreducible $\bfH(\Ss_d)$-module $M_{\lambda,+} \bfH(\Ss_d) \simeq S^\lambda$ is contained in the direct sum of some copies of the irreducible highest weight $\U$-module with highest weight corresponding to a partition, say $\mu$. Applying Kashiwara operators $\Etil_i$'s on $M_{\lambda,+}$ repeatedly, one can easily verify that $\mu = \lambda$.

Exchanging the roles of $H_i$ and $H_i\inv$, we obtain $M_{\lambda,-} \in \bfV^{\otimes d}$ such that
\begin{align}
M_{\lambda,-} \in (u_{-\lambda'_1} \otimes \cdots \otimes u_{-2} \otimes u_{-1}) &\otimes (u_{-\lambda'_2} \otimes \cdots \otimes u_{-2} \otimes u_{-1}) \otimes \cdots \nonumber \\
&\otimes (u_{-\lambda'_k} \otimes \cdots \otimes u_{-2} \otimes u_{-1}) + q\bfL^{\otimes d} \nonumber
\end{align}
and $M_{\lambda,-} \bfH(\Ss_d) \simeq S^\lambda$.

\subsection{Irreducible $\bfH$-modules}
In this subsection, we construct the irreducible $\bfH$-modules following \cite{DJ92}. For $1 \leq i < j \leq d-1$, we set
\begin{align}
s_{i,j} := s_i s_{i+1} \cdots s_{j-1}, \qu s_{j,i} := s_{i,j}\inv. \nonumber
\end{align}
Fix two nonnegative integers $a,b$ such that $a+b = d$, and set $w_{a,b} := (s_{d,1})^b \in \Ss_{d}$. Also, we define $v_{a,b} \in \bfH$ by
\begin{align}
v_{a,b} := \prod_{i=1}^a(p + H_{s_{i,1}}H_0H_{s_{1,i}}) H_{w_{a,b}} \prod_{j=1}^b(1-pH_{s_{j,1}}H_0H_{s_{1,j}}). \nonumber
\end{align}

Let $\lambda \vdash a$ and $\mu \vdash b$. By Appendix \ref{B.2}, one can construct the irreducible $\bfH(\Ss_a)$-module $S^\lambda$ in the subalgebra of $\bfH$ generated by $H_1,\ldots,H_{a-1}$, and the irreducible $\bfH(\Ss_b)$-module $S^{\mu}$ in the subalgebra generated by $H_{a+1},\ldots,H_{n-1}$. It follows that $S^\lambda \cdot S^\mu \subset \bfH(\Ss_{a,b})$. Set
\begin{align}
S^{\lambda,\mu} := S^\lambda \cdot S^\mu \cdot v_{a,b} \bfH = v_{a,b} S^\mu \cdot S^\lambda \bfH. \nonumber
\end{align}

\begin{theo}[\cite{DJ92}]
The set $\{ S^{\lambda,\mu} \mid 0 \leq a \leq d,\ \lambda \vdash a,\ \mu \vdash d-a \}$ provides a complete list of pairwise nonisomorphic irreducible $\bfH$-modules.
\end{theo}

Let us find a good generator of $S^{\lambda,\mu}$ in $\bfV^{\otimes d}$. Define a map $f_{\lambda,\mu}$ by:
\begin{align}
f_{\lambda,\mu}(i) = \begin{cases}
f_\lambda(i) & \IF 1 \leq i \leq a, \\
f_\mu(i-a) & \IF a+1 \leq i \leq d.
\end{cases} \nonumber
\end{align}
By Appendix \ref{B.2}, we have
\begin{align}
S^\lambda \cdot S^\mu \simeq M_{f_{\lambda,\mu}} S^\lambda \cdot S^\mu \subset M_{f_{\lambda,\mu}} \bfH(\Ss_d), \nonumber
\end{align}
and hence,
\begin{align}
S^{\lambda,\mu} \simeq M_{f_{\lambda,\mu}} S^\lambda \cdot S^\mu v_{a,b} \bfH = M_{f_{\lambda,\mu}} v_{a,b} S^\mu \cdot S^\lambda \bfH. \nonumber
\end{align}
Also, we see that
\begin{align}
M_{f_{\lambda,\mu}} v_{a,b} \in u_{f_\mu(1)} \otimes \cdots \otimes u_{f_\mu(b)} \otimes u_{-f_\lambda(1)} \otimes \cdots \otimes u_{-f_\lambda(a)} + p\bfL^{\otimes d}. \nonumber
\end{align}
Therefore, $M_{f_{\lambda,\mu}} v_{a,b} S^\mu \cdot S^\lambda \bfH$ is generated by $M_{\mu,+} \otimes M_{\lambda,-}$, which is of the form
\begin{align}
M_{\mu,+} \otimes &M_{\lambda,-} \in (u_1 \otimes \cdots \otimes u_{\mu'_1}) \otimes (u_1 \otimes \cdots \otimes u_{\mu'_2}) \otimes \cdots \otimes (u_1 \otimes \cdots \otimes u_{\mu'_l}) \nonumber\\
&\otimes (u_{-\lambda'_1} \otimes \cdots \otimes u_{-1}) \otimes (u_{-\lambda'_2} \otimes \cdots \otimes u_{-1}) \otimes \cdots \otimes (u_{-\lambda'_k} \otimes \cdots \otimes u_{-1}) + q\bfL^{\otimes d}. \nonumber
\end{align}

By the quantum Schur-Weyl duality of type $B$, the irreducible $\bfH$-modules $M_{\mu,+} \otimes M_{\lambda,-} \bfH \simeq S^{\lambda,\mu}$ is contained in the direct sum of some copies of the irreducible highest weight $\Uj$-module $L(\bflm)$ for some $\bflm \in \Pj(d)$. By the descriptions of $M_{\mu,+} + q\bfL^{\otimes a}$ and $M_{\lambda,-} + q\bfL^{\otimes b}$, it is clear that
\begin{align}
\etil_1^{\lambda_1} \cdots \etil_r^{\lambda_r} &(M_{\mu,+} \otimes M_{\lambda,-} + q\bfL^{\otimes d} \in \bfB^{\otimes d}) \nonumber\\
&=  (u_1 \otimes \cdots \otimes u_{\mu'_1}) \otimes (u_1 \otimes \cdots \otimes u_{\mu'_2}) \otimes \cdots \otimes (u_1 \otimes \cdots \otimes u_{\mu'_l}) \nonumber\\
&\qu \otimes (u_{-(\lambda'_1-1)} \otimes \cdots \otimes u_{0}) \otimes (u_{-(\lambda'_2-1)} \otimes \cdots \otimes u_{0}) \otimes \cdots \otimes (u_{-(\lambda'_k-1)} \otimes \cdots \otimes u_{0}) + q\bfL^{\otimes d} \nonumber
\end{align}
is a Yamanouchi biword of shape $(\lambda;\mu)$, and hence, we conclude that $L(\bflm) = L(\lambda;\mu)$.

\begin{theo}
As a $\Uj$-$\bfH$-bimodule, $\bfV^{\otimes d}$ is decomposed as follows:
\begin{align}
\bfV^{\otimes d} \simeq \bigoplus_{\bflm \in \Pj(d)} L(\bflm) \boxtimes S^{\bflm^-,\bflm^+}. \nonumber
\end{align}
\end{theo}

\begin{cor}
For each partitions $\lm,\mu$ with $|\lm| + |\mu| = d$, we have
$$
\clFj(S^{\lm,\mu}) = \begin{cases}
L(\lm;\mu) \qu & \IF (\lm;\mu) \in \Pj(d), \\
0 \qu & \OW.
\end{cases}
$$
\end{cor}


\begin{thebibliography}{99}
\bibitem[AKR17]{AKR17}, N. Aldenhoven, E. Koelink, and P. Rom\'{a}n, Branching rules for finite-dimensional $U_q(\mathfrak{su}(3))$-representations with respect to a right coideal subalgebra, Algebr. Represent. Theory 20 (2017), no. 4, 821--842. 

\bibitem[BW13]{BW13} H. Bao and W. Wang, A new approach to Kazhdan-Lusztig theory of type $B$ via quantum symmetric pairs, arXiv:1310.0103v2. 

\bibitem[BWW16]{BWW16} H. Bao, W. Wang, and H. Watanabe, Multiparameter quantum Schur duality of type $B$, arXiv:1609.01766. To appear in Proc. Amer. Math. Soc.

\bibitem[BB05]{BB05} A. Bj\"{o}rner and F. Brenti, Combinatorics of Coxeter Groups, Graduate Texts in Mathematics, 231. Springer, New York, 2005. xiv+363 pp.

\bibitem[BI03]{BI03} C. Bonnaf\'{e} and L. Iancu, Left cells in type $B_n$ with unequal parameters, Represent. Theory 7 (2003), 587-609.

\bibitem[DJ92]{DJ92} R. Dipper and G. James, Representations of Hecke algebras of type $B_n$, J. Algebra 146 (1992), no. 2, 454-481. 

\bibitem[Deo87]{Deo87} V. V. Deodhar, On some geometric aspects of Bruhat orderings. II. The parabolic analogue of Kazhdan-Lusztig polynomials, J. Algebra 111 (1987), no. 2, 483-506.

\bibitem[D93]{D93} M. J. Dyer, Hecke algebras and shellings of Bruhat intervals, Compositio Math. 89 (1993), no. 1, 91-115. 

\bibitem[GAP16]{GAP16} The GAP group, GAP - Groups, algorithms, and programming, Version 4.8.3, 2016.

\bibitem[G86]{G86} A. Gyoja, A $q$-analogue of Young symmetrizer, Osaka J. Math. 23 (1986), no. 4, 841-852. 

\bibitem[H74]{H74} P. N. Hoefsmit, Representations of Hecke algebras of finite groups with BN-pairs of classical type, Thesis (Ph.D.)-The University of British Columbia (Canada). ProQuest LLC, Ann Arbor, MI, 1974.

\bibitem[HK02]{HK02} J. Hong and S. J. Kang, Introduction to Quantum Groups and Crystal Bases, Graduate Studies in Mathematics Vol. 42., American Mathematical Society, Providence, RI, 2002. xviii+307 pp.

\bibitem[J86]{J86} M. Jimbo, A $q$-analogue of $U(\mathfrak{gl}(N+1))$, Hecke algebra, and the Yang-Baxter equation, Lett. Math. Phys. 11 (1986), no. 3, 247-252. 


\bibitem[K90]{K90} M. Kashiwara, Crystalizing the 
$q$-analogue of universal enveloping algebras, Comm. Math. Phys. 133 (1990), no. 2, 249--260. 



\bibitem[KL79]{KL79} D. Kazhdan and G. Lusztig, Representations of Coxeter groups and Hecke algebras, Invent. Math. 53 (1979), no. 2, 165-184. 

\bibitem[Ko14]{Ko14} S. Kolb, Quantum symmetric Kac-Moody pairs, Adv. Math. 267 (2014), 395-469. 


\bibitem[KP11]{KP11} S. Kolb and J. Pellegrini, Braid group actions on coideal subalgebras of quantized enveloping algebras, J. Algebra 336 (2011), 395-416. 

\bibitem[Kw09]{Kw09} J. H. Kwon, Crystal graphs and the combinatorics of Young tableaux, Handbook of Algebra Vol. 6, 473-504, Handb. Algebr., 6, Elsevier/North-Holland, Amsterdam, 2009.

\bibitem[Le99]{Le99} G. Letzter, Symmetric pairs for quantized enveloping algebras, J. Algebra 220 (1999), no. 2, 729-767. 

\bibitem[LS91]{LS91} S. Z. Levendorskii and Y. S. Soibelman, The quantum Weyl group and a multiplicative formula for the $R$-matrix of a simple Lie algebra, Funct. Anal. Appl. 25 (1991), no. 2, 143-145 .

\bibitem[Lu90]{L90} G. Lusztig, Quantum groups at roots of $1$, Geom. Dedicata, 33 (1990), 89-113.

\bibitem[Lu94]{L94} G. Lusztig, Introduction to Quantum Groups, Reprint of the 1994 edition. Modern Birkh\"{a}user Classics. Birkh\"{a}user/Springer, New York, 2010. xiv+346 pp.

\bibitem[L03]{L03} G. Lusztig, Hecke algebras with unequal parameters, CRM Monograph Series, 18. American Mathematical Society, Providence, RI, 2003. vi+136 pp.

\bibitem[X94]{X94} N. Xi, Representations of Affine Hecke Algebras, Lecture Notes in Mathematics, 1587. Springer-Verlag, Berlin, 1994. viii+137 pp.
\end{thebibliography}
\end{document}